\documentclass[11pt, reqno,amsmath,amsthm,amssymb,amscd]{amsart}
\usepackage{mathrsfs,amssymb, amscd,amsmath,amsthm}
\usepackage[enableskew,vcentermath]{youngtab}
\usepackage{multicol}\multicolsep=0pt
\usepackage{tikz}

\newcommand{\scK}{\mathscr{K}}

\newcommand{\B}{\mathcal{B}}

 \providecommand{\og}{``}
\providecommand{\fg}{''} \providecommand{\smfandname}{and}

\usepackage{amssymb}
\usepackage{mathrsfs,amssymb}
\usepackage{multicol}\multicolsep=0pt
\usepackage[enableskew,vcentermath]{youngtab}

\usepackage[sort]{cite}
\usepackage{xcolor,graphicx}

\def\crulefill{\leavevmode\leaders\hrule height 1pt\hfill\kern 0pt}
\long\def\QUERY#1{%
\leavevmode\newline%
\noindent$\star\star\star$\thinspace\textsf{Comment/Query}\crulefill\newline%
   \space #1\newline\hbox to 120mm{\crulefill}$\star\star\star$\newline}
\newtheorem{Theorem}{Theorem}[section]
\newtheorem{Lemma}[Theorem]{Lemma}

\newtheorem{Prop}[Theorem]{Proposition}

 \theoremstyle{definition}
\newtheorem{example}[Theorem]{Example}

\newtheorem{Defn}[Theorem]{Definition}

\newtheorem{Assumptions}[Theorem]{Assumption}

\numberwithin{equation}{section}
\theoremstyle{definition}

\newtheorem{THEOREM}{Theorem}

\def\Case#1{\medskip\noindent\textbf{Case #1}:\leavevmode\newline}
\def\subcase#1{\bigskip\noindent\textbf{Subcase #1}:\leavevmode\newline}

\makeatletter
\def\enumerate{\begingroup\ifnum\@enumdepth>3\@toodeep\else
      \advance\@enumdepth\@ne
      \edef\@enumctr{enum\romannumeral\the\@enumdepth}%
      \topsep\z@\parskip\z@
      \list{\csname label\@enumctr\endcsname}
        {\@nmbrlisttrue\let\@listctr\@enumctr
         \parsep\z@\itemsep\z@\topsep\z@
         \setcounter{\@enumctr}{0}
         \def\makelabel##1{\hss\llap{\rm ##1}}
       }\fi}

\makeatother

\let\bar=\overline
\let\epsilon=\varepsilon
\def\({\big(}
\def\){\big)}

\def\0{\underline{0}}

\DeclareMathOperator{\End}{End}

\DeclareMathOperator{\Rad}{Rad}

\def\Std{\mathscr{T}^{std}}

\def\n{\mathfrak n}
\def\s{\mathfrak s}

\def\t{\mathfrak t}

\def\bft{\t}

\def\Hom{\text{Hom}}

\def\U{\mathbf U}

\def\bbZ{\mathbb Z}

{\catcode`\|=\active
  \gdef\set#1{\mathinner{\lbrace\,{\mathcode`\|"8000%
                                   \let|\midvert #1}\,\rbrace}}
  \gdef\seT#1{\mathinner{\Big\lbrace\,{\mathcode`\|"8000%
                                   \let|\midverT #1}\,\Big\rbrace}}
}
\def\midvert{\egroup\mid\bgroup}
\def\midverT{\egroup\,\Big|\,\bgroup}

\def\Set[#1]#2|#3|{\Big\{\ #2\ \Big| \
           \vcenter{\hsize #1mm\centering #3}\Big\}}

\def\qed{\hfill\mbox{$\Box$}}

\newcommand{\scI}{\mathscr{I}}

\def\Hom{{\rm Hom}}

\def\mfg{{\mathfrak g}}

\def\Set{{\rm Set}}

\newcommand{\eps}{\epsilon}

\def\frp{\mathfrak p}
\def\frg{\mathfrak g}

\def\Std{\mathscr{T}^{std}}%
\def\n{\mathfrak n}%
\def\s{\mathfrak s}%
\def\t{\mathfrak t}%
\def\bft{\t}%
\def\Hom{\text{Hom}}%
\def\U{\mathbf U}%
\def\textsf#1{{\textit{#1}}}%

\def\frso{\mathfrak{so}}
\def\frh{\mathfrak{h}}
\def\frn{\mathfrak{n}}
\def\frl{\mathfrak{l}}
\def\fru{\mathfrak{u}}
\def\frb{\mathfrak{b}}

\begin{document}
\baselineskip14pt
\title[{\tiny  Decomposition numbers of cyclotomic Brauer algebras }]
{Decomposition numbers of  cyclotomic Brauer algebras over the complex field, I}
\author{ Mengmeng Gao and  Hebing Rui, 
(with an appendix by  Wei Xiao)}

\address{M.G.  School of Mathematical Science, Tongji University,  Shanghai, 200092, China}\email{1810414@tongji.edu.cn}
\address{H.R.  School of Mathematical Science, Tongji University,  Shanghai, 200092, China}\email{hbrui@tongji.edu.cn}
\address{W. Xiao, College of Mathematics and Statistics, Shenzhen Key Laboratory of Advanced Machine Learning and Applications, Shenzhen University, Shenzhen, 518060, Guangdong, China}\email{xiaow@szu.edu.cn}
\thanks{H. Rui is supported  partially by NSFC (grant No.  11571108).  M. Gao is supported  partially by NSFC (grant No.  12301038). }
\date{\today}

\begin{abstract}  Following Nazarov's suggestion~\cite{Naz1}, we refer to  the cyclotomic Nazarov-Wenzl algebra as the cyclotomic Brauer algebra.   
When the cyclotomic Brauer algebra  is isomorphic to the   endomorphism algebra of  $M_{I_i, r}$-- the tensor product of a simple scalar-type parabolic Verma module with the natural module in the parabolic BGG category $\mathcal O$  of types $B_n$, $C_n$ and $D_n$, its decomposition numbers  can theoretically   be computed, based on general results from \cite{AST} and \cite[Corollary~5.10]{RS}.  

 This paper aims to  establish explicit connections between the parabolic Verma modules that appear as subquotients of $M_{I_i, r}$ and the right cell modules of the cyclotomic Brauer algebra   under  condition~\eqref{simple111}.   It allows us to 
explicitly decompose $M_{I_i, r}$  into a direct sum of indecomposable tilting modules by identifying  their  highest  weights and multiplicities. Our result demonstrates that  the   decomposition numbers of such a  cyclotomic Brauer algebra  can be explicitly computed using  the  parabolic Kazhdan-Lusztig polynomials  of types $B_n$, $C_n$, and $D_n$ with suitable parabolic subgroups~\cite{So}.  Finally, condition~\eqref{simple111}  is well-supported  by a result of  Wei Xiao presented in Section~6. 

\end{abstract}
\maketitle
\baselineskip14pt

\section{Introduction}

 Throughout this paper, we  work over  $\mathbb C$. All  algebras and categories 
are  defined over $\mathbb C$.

   In his groundbreaking  paper  \cite{Ari}, Ariki established a remarkable result stating that   $$K_0(\bigoplus_{r=0}^\infty \mathscr H_{a, r}(\mathbf u)\text{-mod})\otimes_\mathbb Z\mathbb C$$ is isomorphic to an  integral highest weight $\mfg$-module. Here  $\mathscr H_{a, r}(\mathbf u)$ denotes the cyclotomic Hecke algebra of type $G(a, 1, r)$ with parameters $\mathbf u=(u_1, u_2, \ldots, u_a)$, and  $\mfg$ is either $\mathfrak{sl}_\infty$ or $\hat{\mathfrak {sl}}_e$~\cite{Ari}. In this context,  $e$ represents  the quantum characteristic of $q$,  a parameter within  $\mathscr H_{a, r}(\mathbf u)$.
   
   Ariki further demonstrated  that the dual canonical basis elements and  canonical basis elements  of  the integral highest weight module correspond to 
simple $\mathscr H_{a, r}(\mathbf u)$-modules and their projective covers, respectively. When $a=1$, this  result  confirms Lascoux-Leclerc-Thibon's conjecture regarding the decomposition numbers of the  Hecke algebra over $\mathbb C$ at a primitive  $e$th root of unity. 
  
For  two positive integers $a$ and $r$, and two  families of parameters $\mathbf u=(u_1, u_2, \cdots, u_a)$, and
$\omega=(\omega_i)_{i\in \mathbb N}$, Ariki, Mathas and Rui~\cite{AMR} 
 introduced  a class of associative algebras, known as  the cyclotomic Nazarov-Wenzl  algebras $\mathcal B_{a, r}(\mathbf u)$, aiming to  replace the cyclotomic Hecke algebras in Ariki's framework.

The cyclotomic Nazarov-Wenzl algebra is a cyclotomic quotient of the  affine Wenzl algebra in \cite{Naz}. Based on   Nazarov's suggestion~\cite{Naz1}, we refer to the affine Wenzl algebra, and the  cyclotomic Nazarov-Wenzl algebra  as  the  \textsf{affine Brauer algebra, and the cyclotomic Brauer algebra}, respectively.

It was proven in~\cite{AMR}   that    $\mathcal B_{a, r}(\mathbf u)$ reaches  its  maximal dimension $a^r (2r-1)!!$ if and only if $\omega$ is $\mathbf u$-admissible, as defined  in   \cite[Definition~3.6]{AMR}. Moreover,  it follows from \cite{G09}  that    the representation theory of $\mathcal B_{a, r}(\mathbf u)$ is fully governed  under the $\mathbf u$-admissible condition. Therefore, it suffices to study representations of  $\mathcal B_{a, r}(\mathbf u)$ within this   framework. 

With  $\mathbf u$-admissibility of $\omega$,   $\mathcal B_{a, r}(\mathbf u)$ is a (weakly) cellular algebra over the poset $\Lambda_{a, r}$,  which  consists of all pairs $(f, \lambda)$. Here    $\lambda:=(\lambda^{(1)}, \lambda^{(2)}, \ldots, \lambda^{(a)})$ ranges over all  $a$-multipartitions of $r-2f$, and   $0\le f\le \lfloor r/2\rfloor$~\cite[Theorem~7.17]{AMR}.  

This paper uses  an alternative weakly  cellular basis for $\mathcal B_{a, r}(\mathbf u)$ in Theorem~\ref{cellular-1}. By Theorem~\ref{bncell1}, we have  another family of  right cell modules $C(f, \lambda)$, for all $(f, \lambda)\in \Lambda_{a, r}$, along with  a family of simple modules $D(f, \lambda)$, for $(f, \lambda)\in \bar \Lambda_{a, r}$ under the assumption $\omega_0\neq 0$, where 
 \begin{equation}\label{simple123} \bar\Lambda_{a, r}=\{(f, \lambda)\in \Lambda_{a, r}\mid \sigma^{-1}(\lambda)  \text{ is $\mathbf {u}$-restricted in the sense of \eqref{difort}}\},\end{equation}
and $\sigma$ denotes  the \textsf{generalized Mullineaux involution} in \cite[Remark~5.10]{RS1}.

Our goal is to  compute 
\begin{equation}\label{dec1} [C(f, \lambda): D(\ell, \mu)],
\end{equation}
  the decomposition number representing  the multiplicity of $D(\ell, \mu)$ in a composition series of $C(f, \lambda)$  for any  $(f, \lambda)\times (\ell, \mu)\in \Lambda_{a, r}\times \bar \Lambda_{a, r}$. 
  
  The approach is based on~\cite[Theorem~5.4]{RS} stated in Theorem~\ref{thmA}, which established the fundamental connection   between the cyclotomic Brauer algebras and the parabolic BGG category $\mathcal O$ in types $B_n, C_n$ and $D_n$.  To formulate it, we introduce some necessary notions. 

  Let $\mfg$ be either  symplectic Lie algebra $\mathfrak{sp}_{2n}$ or orthogonal Lie algebra $\mathfrak{so}_{2n}$ or $\mathfrak{so}_{2n+1}$. Define the parabolic subalgebra $\mathfrak p_{I_i}\subset \mfg$ corresponding to the subsets $I_1$ and $I_2$, where 
\begin{equation}\label{i1i2d} I_1=\Pi\setminus \{\alpha_{p_1}, \alpha_{p_2}, \ldots, \alpha_{p_{k}} \}\text{ and  $I_2= I_1\cup \{ \alpha_{n}\}$,}\end{equation}  and $0=p_0<p_1<p_2 <\cdots <p_{k-1}<p_k=n$. Here   $\Pi=\{\alpha_1, \alpha_2, \ldots, \alpha_n\}$ is the set of simple roots of $\mfg$. 
  Define  \begin{equation}\label{pdom}\Lambda^{\frp_{I_i}}=\{\lambda\in \mathfrak h^*\mid \langle \lambda, \alpha^\vee\rangle \in \mathbb N\ \  \text{for all $\alpha\in I_i$}\},\end{equation} 
where $\mathfrak h^*$ is the weight space of $\mfg $. 
Let $V$ denote   the natural  $\mathfrak g$-module, and define 
\begin{equation}\label{mir} M_{I_i, r} :=M^{\frp_{I_i}}(\lambda_{I_i, \mathbf c})\otimes V^{\otimes r},\end{equation} where  $M^{\frp_{I_i}}(\lambda_{I_i, \mathbf c})$ is  the parabolic Verma module with the highest weight  
\begin{equation}\label{deltac}\lambda_{I_i, \mathbf c}=\sum_{j=1}^{k} c_j(\epsilon_{p_{j-1}+1}+\epsilon_{p_{j-1}+2}+\cdots+\epsilon_{p_j})\in\Lambda^{\frp_{I_i}},\end{equation}  with  $(c_1,\ldots, c_k)\in \mathbb C^k$ such that $c_k=0$ if $i=2$. 
Denote by $\Phi$  the root system of $\mfg$.
  \begin{THEOREM}\label{thmA} \cite[Theorem~5.4]{RS}  Suppose  $\Phi\neq B_n$  if $i=1$, and $M^{\frp_{I_i}}({\lambda_{I_i, \mathbf c} })$ is simple (and hence tilting). If  $p_t-p_{t-1}\ge 2r$ for all $1\le t\le k$,     then    $  \text{End}_{\mathcal O^{\frp_{I_i}}}(M_{I_i, r})\cong \mathcal  B_{a, r}^{\text{op}}(\mathbf u)$. Here    $\mathcal  B_{a, r}(\mathbf u)$ is the cyclotomic Brauer algebra with the parameters  $\mathbf u=(u_1,\ldots, u_a)$
  such that  $\omega$ is $\mathbf u$-admissible, where $u_1, u_2, \ldots, u_a$ are given in \eqref{ujjj}.  Furthermore, \begin{equation}\label{defa} a=\begin{cases} 2k &\text{if $i=1$,}\\  2k-1 & \text{if $i=2$.}\\
  \end{cases} \end{equation}\end{THEOREM}

\begin{Assumptions}\label{keyassu}\textsf{
 $M^{\frp_{I_i}}(\lambda_{I_i, \mathbf c })$ is simple, and $p_t-p_{t-1}\ge 2r$, $1\le t\le k$.}  \end{Assumptions}
 From this point on, we always keep Assumption~\ref{keyassu}. This allows us to use  Theorem~\ref{thmA}, freely.
 
 For any $M\in\mathcal O^{\frp_{I_i}}$ such that $M$ admits a finite parabolic Verma flag, let $(M: M^{\frp_{I_i}}(\lambda)) $ denote  the multiplicities of $ M^{\frp_{I_i}}(\lambda)$ as  a subquotient of $M$.  Since $M^{\frp_{I_i}}(\lambda_{I_i, \mathbf c})$ is simple,  $M_{I_i, r}$ is a tilting module. Consequently,   each indecomposable direct summand of $M_{I_i, r}$ is an indecomposable tilting module. Write \begin{equation}\label{decten1} M_{I_i, r}=\bigoplus_{\mu} T^{\frp_i}(\mu)^{\oplus n_\mu} ,\end{equation} where 
 $T^{\frp_i}(\mu)$ is the indecomposable tilting module with the highest weight $\mu$.
It follows from  \cite[\S 4]{AST} that $\text{End}_{\mathcal O^{\frp_{I_i}}}(M_{I_i, r})$ is a cellular algebra with respect to the poset $( \mathscr{I}_{i, r}, \leq)$, where $\leq$ is the dominance order defined on $\mathfrak h^*$ such that $\lambda\le \mu$ if $\mu-\lambda\in \mathbb N\Pi$, and 
\begin{equation}\label{sciI}
     \mathscr{I}_{i, r}=\{\mu\in\frh^*\mid \ (M_{I_i, r} : M^{\frp_{I_i}}(\mu))\neq 0\}.
 \end{equation} The left cell modules are given by 
$$S(\lambda):= \Hom_{\mathcal O^{\frp_{I_i}}} (M^{\frp_{I_i}}(\lambda), M_{I_i, r} ),  \text{ $\lambda\in \mathscr{I}_{i, r}$}.$$ 
It follows from \cite{GL} that there exists an invariant form  $\phi_\lambda$   on each  $S(\lambda)$. Thanks to \cite[Theorem 4.11]{AST}, 
$$D(\lambda):=S(\lambda)/\Rad \phi_\lambda \neq 0$$  if and only if $n_\lambda\neq 0$. 
Further, all non-zero $D(\lambda)$ form a pair-wise non-isomorphic 
simple modules for $\End_{\mathcal O^{\frp_{I_i}}}(M_{I_i, r})$.  

The principal indecomposable modules are  given by 
$$P(\lambda):= \Hom_{\mathcal O^{\frp_{I_i}}} (T^{\frp_{I_i}}(\lambda), M_{I_i, r} ),  $$
where $T^{\frp_{I_i}}(\lambda)$ ranges over all non-isomorphic indecomposable direct summands of $M_{I_i, r}$. Further, by  \cite{AST},  $P(\lambda)$ is the projective cover of $D(\lambda)$. It was proven in \cite[Corollary~5.10]{RS} that \begin{equation}\label{dec123}  [C(\lambda): D(\mu)]=(T^{\mathfrak p_{I_i}}(\hat \mu):M^{\mathfrak p_{I_i}}(\hat \lambda))\end{equation} 
 for all $\lambda, \mu\in \scI_{i, r}$ with $n_\mu\neq 0$. From Theorem~\ref{thmA}, $S(\lambda)$, $D(\lambda)$ and $P(\lambda)$ can be viewed as right $\mathcal B_{a, r}(\mathbf u)$-modules.

Since the information on  the indecomposable direct summands $T^{\frp_{I_i}}(\mu)$ of $M_{I_i, r}$ in \eqref{decten1} is incomplete, the multiplicities $[C(\lambda): D(\mu)]$,  $(T^{\frp_{I_i}}(\mu) :M^{\frp_{I_i}}(\lambda))$ and $n_\mu$ remain unknown in principal.

We introduce the partial ordering on $\Lambda^{\frp_{I_i}}$ such that  \begin{equation}\label{preceq} \lambda\preceq \mu\end{equation} indicates  the existence of  a sequence $\lambda=\gamma^0, \gamma^1, \ldots, \gamma^j=\mu$ in $\Lambda^{\frp_{I_i}}$  satisfying  that  the simple $\mathfrak g$-module $L(\gamma^{l-1})$ with the highest weight $\gamma^{l-1}$  appears as a composition factor of $M^{\frp_{I_i}} (\gamma^l)$, for all $1\le l\le j$. Write $\lambda\prec \mu$ if $\lambda\preceq\mu$ and $\lambda\neq \mu$. 
 We expect 
 
{ \begin{equation}\label{simple111} \text{$\scI_{i, {j}}$ is saturated in the sense that $\mu\in \scI_{i, {j}}$ if $\mu\preceq \lambda$ for some $ \lambda\in \scI_{i, j}$, $0\le j\le r$.}\end{equation}}

\begin{THEOREM}\label{main111} Suppose $0< f \leq \lfloor r/2 \rfloor$, and  $\mu \in \mathscr{I}_{i, r} \setminus \mathscr{I}_{i, r-2f}$. Under  condition \eqref{simple111}, we have   
$$
\text{Hom}_{\mathcal{O}^{\frp_{I_i}}}(M^{\frp_{I_i}}(\mu), M_{I_i,r}) \cong \text{Hom}_{\mathcal{O}^{\frp_{I_i}} }(M^{\frp_{I_i}}(\mu), M_{I_i,r}/M_{I_i,r}\langle  E^f \rangle),$$  
as  right $\mathcal{B}_{a,r}(\mathbf u)$-modules,
where $\mathcal{B}_{a,r}(\mathbf u)$ is the cyclotomic Brauer algebra in Theorem~\ref{thmA}, and $\langle   E^f\rangle$ is the two-sided ideal of $\mathcal B_{a, r}(\mathbf u)$ generated by $ E^f:=E_{r-1} E_{r-3}\cdots  E_{r-2f+1}$.\end{THEOREM}

For any $(f, \lambda)\in \Lambda_{a, r}$, let $\hat\lambda\in \Lambda^{\mathfrak p_{I_i}}$  be   defined as in \eqref{hatlambda}. In Theorem~\ref{main123}, we classify singular vectors in $M_{I_i, r}/M_{I_i, r} \langle E^{f+1}\rangle $ with the highest weight $\hat\lambda$ using explicit construction of right cell modules for $\mathcal B_{a, r}(\mathbf u)$ in Proposition~\ref{bas}. This result is of independent interest in its own right. Applying it, we  prove the following theorem. Keep in mind that $\mathcal B_{a, r}(\mathbf u)$ is the cyclotomic Brauer algebra in Theorem~\ref{thmA}.

\begin{THEOREM}\label{main4} Under condition~\eqref{simple111}, 
    $$\Hom_{\mathcal O^{\frp_{I_i}}}(M^{\frp_{I_i}}(\hat \lambda), M_{I_i, r}/M_{I_i, r} \langle E^{f+1} \rangle )\cong C(f, \lambda')$$ as right $\mathcal B_{a, r}(\mathbf u)$-modules
    for any  $(f, \lambda)\in \Lambda_{a, r}$,     
    where $\lambda'=(\mu^{(1)}, \mu^{(2)}, \ldots, \mu^{(a)})$ is the conjugate of $\lambda$ in the sense that  $\mu^{(i)}$ is the conjugate of the partition $\lambda^{(a-i+1)}$,  $1\le i\le a$.\end{THEOREM}

Theorem~\ref{main4} depends on condition \eqref{simple111}, as we use Theorem~\ref{main111} to compute the dimension of $\Hom_{\mathcal{O}^{\frp_{I_i}}}(M^{\frp_{I_i}}(\hat \lambda), M_{I_i, r}/M_{I_i, r})$ in the proof of Theorem~\ref{main4}.
  Using  Theorems~\ref{main111}, and \ref{main4}, we obtain Theorem~\ref{first}(1),  which represents the  most challenging aspect of this paper. 
Notably, Theorem~\ref{first}(2)-(4) follow as  direct  consequences of Theorem~\ref{first}(1).
\begin{THEOREM}\label{first} Under condition~\eqref{simple111},  and assuming  $i$ is either $1$ or $2$, we have 
 \begin{itemize}\item[(1)] $\Hom_{\mathcal O^{\frp_{I_i}}}(M^{\frp_{I_i}}(\hat \lambda), M_{I_i, r})\cong C(f, \lambda')$ as  right $\mathcal B_{a, r}(\mathbf u)$-modules, where $(f, \lambda)\in \Lambda_{a, r}$.  
 \item [(2)] $D(\hat \lambda)\cong D(f, \lambda')$ for all $(f, \lambda')\in \bar{\Lambda}_{a, r}$.
\item [(3)]   $M_{I_i, r}=\bigoplus_{(f, \lambda')\in \bar \Lambda_{a, r} } T^{\frp_i}(\hat\lambda )^{\oplus \dim D(f, \lambda')} $.
\item[(4)]  $ [C(f, \lambda'): D(\ell, \mu')]=(T^{\mathfrak p_{I_i}}(\hat \mu):M^{\mathfrak p_{I_i}}(\hat \lambda))$ 
 for all   $((f, \lambda'), (\ell, \mu'))  \in \Lambda_{a,r} \times  \bar\Lambda_{a,r}$.
\end{itemize}     \end{THEOREM}

The dimension of $D(f, \lambda')$ can be  determined using Theorem~\ref{first}(4). Specifically,  this dimension  can be explicitly calculated  using the parabolic Kazhdan-Lusztig polynomials of types $B_n$, $C_n$, and $D_n$~\cite{So}. 
 
 Let  $\Phi^+$ denote the set of positive roots associated with $\mathfrak g$. Define $\Phi_{I_i}=\Phi\cap \mathbb Z I_i$, and let $\rho$ represent   half the  sum of all positive roots. To illustrate that condition \eqref{simple111} is well-justified, we need the following assumption, which ensures that $M^{\frp_{I_i}} (\lambda_{I_i, \mathbf c})$ is simple~\cite[Theorem~9.12]{Hum}.
\begin{Assumptions}\label{simple11} Assume that $\langle \lambda_{I_i, \mathbf c} +\rho, \beta^\vee \rangle \not\in \mathbb Z_{>0}  $ for all $\beta\in \Phi^+\setminus \Phi_{I_i}$,  where   $i\in \{1, 2\}$ with the condition that   $i\neq 1$ if $\Phi= B_n$.   \end{Assumptions}

The following result will be proved in Section 6, as an appendix to the paper. 

. 

\begin{THEOREM} (W. Xiao) \label{saturated} 
Under Assumption~\ref{simple11},
{$\scI_{i, j}$ is saturated with respect to the partial ordering $\preceq$,  for all $0\le j\le r$.}
\end{THEOREM}

Rui and Song will compute the decomposition numbers of  $\mathcal B_{a, r}((-1)^a\mathbf u)$ with arbitrary parameters $$(-1)^a \mathbf u=((-1)^a u_1, (-1)^a u_2, \cdots, (-1)^a u_a)$$ such that  $\omega$ is $(-1)^a\mathbf u$-admissible. The influential paper~\cite{ES1} motivates the approach, where Erig and Stroppel embed   the Brauer algebra \cite{Bra} (i.e. the level one cyclotomic Brauer algebra) into a level two cyclotomic Brauer algebra. 

Rui and  Song  will  embed the  $\mathcal B_{a, r}((-1)^a\mathbf u)$ with arbitrary parameters $(-1)^a\mathbf u $
into another cyclotomic Brauer algebra $\mathcal B_{2a, r}(\tilde {\mathbf u})$ as an idempotent truncation. The parameters $\tilde{\mathbf u}$  is given by 
$$\tilde{ \mathbf u}=(u_1, u_2, \ldots, u_a, u_{a+1}, \ldots u_{2a})$$   where $u_{a+1}, u_{a+2}, \ldots$,  $ u_{2a}$ are appropriately  chosen parameters.
They further  prove that the algebra    $\mathcal B_{2a, r}(\tilde {\mathbf u})$ is  isomorphic to the endomorphism algebra of a suitable  
tilting module in the parabolic category $\mathcal O$ for an appropriate  parabolic subalgebra of $\mathfrak{so}_{2n}$. Consequently,  the decomposition numbers of $\mathcal B_{2a, r}(\tilde {\mathbf u})$, and thereby those  of $\mathcal B_{a, r}((-1)^a\mathbf u)$, can, in principal,  be computed using \eqref{dec123}. 

To obtain explicit information about  these decomposition numbers, they carefully analyze the condition for which $\scI_{1, r}$ is saturated with respect to  the partial ordering $\preceq$. This analysis enables them to establish the result in Theorem~\ref{first}(4) for $\mathcal{B}{2a, r}(\tilde{\mathbf{u}})$, and consequently, derive explicit information about the decomposition numbers for $\mathcal{B}_{a, r}((-1)^a \mathbf{u})$ with arbitrary parameters. This is achieved using the parabolic Kazhdan-Lusztig polynomials of type $D_n$, associated with a parabolic subgroup of type $A$.

 Of course, they assume that $\omega_0\neq 0$ for $\mathcal B_{a, r}((-1)^a\mathbf u)$, too.
Certainly, these results  depend on  Theorem~\ref{main4}, Theorem~\ref{first} and the classification of singular vectors for 
the $\mathfrak{so}_{2n}$-module $M_{I_1, r}/M_{I_1, r}\langle E^{f+1}\rangle$ in Section~4.
  Details will be given in the forthcoming sequel~\cite{RS-de}.

The cyclotomic 
Brauer category was introduced in ~\cite{RS}. It serves as the needed analog  of the degenerate cyclotomic Hecke algebra. To study representations of the cyclotomic Brauer category,   Song and two of us  introduced the notion of  a \textsf{weakly triangular category}, where   the path algebra of such a category  is equipped with  an \textsf{upper-finite weakly triangular decomposition}~\cite{GRS1}.  We note that an equivalent notion, called the \textit{triangular basis}  was later proposed  
in the third version of \cite{BS}  five months after \cite{GRS1} appeared on the Arxiv.

Let $A$ denote the path algebra associated with  the  cyclotomic Brauer category, and let $A^\Delta$-mod  denote  the full subcategory of   locally finite-dimensional left $A$-modules where  each object admits a finite standard flag. It was proved in \cite{GRS1} that  $$K_0(A^\Delta\text{-mod})\otimes_{\mathbb Z} \mathbb C$$ can be viewed  as the $\frg^\theta$-module $M$, where $M$ is an integral highest weight $\mathfrak g$-module with  $\mathfrak g=\mathfrak {sl}_\infty$ and $(\mathfrak g, \mathfrak g^\theta)$ forming a symmetric pair. This result can be regarded   as a  counterpart of a weaker version of Ariki's renowned  work on the cyclotomic Hecke algebras. 

Inspired  by \cite{Ari},   
we conjecture that the elements of  $\imath$-canonical basis in \cite{BW1, B} for $M$ correspond to  projective covers of  simple $A$-modules, while the elements of dual  $\imath$-canonical basis  correspond to  simple $A$-modules. 
As the cyclotomic Brauer algebras $\mathcal B_{a, r}(\mathbf u)$ are isomorphic to the centralized subalgebras of $A$ for all non-negative integers $r$, 
we hope that the finding  on decomposition numbers of $\mathcal B_{a, r}(\mathbf u)$ with arbitrary parameters will support the completion of this  project. 

The paper is organized as follows.  Section 2 reviews  some elementary results on cyclotomic Brauer   and degenerate cyclotomic Hecke algebras. Section~3 is about the parabolic  category $\mathcal O$ in types $B_n$, $C_n$, and $D_n$, where  we establish   Theorem~B. Section~4  classifies singular vectors in certain quotient modules of $M_{I_i, r}$, while Section~5 proves Theorem~C and Theorem~D. 
Section~6 includes an appendix by  Wei Xiao with a proof of Theorem~E. This result confirms that condition~\eqref{simple111} is well-justified.

\section{The cyclotomic Brauer algebra }
\subsection{Cyclotomic Brauer algebras}  
\begin{Defn} \label{cba1}~\cite[Definition~2.13]{AMR} Let  $a, r$ denote  two positive integers. The  cyclotomic Brauer algebra  $\mathcal B_{a, r}(\mathbf u)$ 
is an  associative algebra generated by elements  $E_i, S_i, X_j$, for  $1\!\le\! i\!\le\! r\!-\!1$, and $1\!\le\! j\!\leq\! r$,
	subject to the relations
	\begin{multicols}{2}
		\begin{enumerate}
			\item [(1)] $S_i^2=1$,  
			\item[(2)] $S_iS_j=S_jS_i$, for $|i-j|>1$,
			\item[(3)] $S_iS_{i+1}S_i\!=\!S_{i+1}S_iS_{i+1}$, 
			\item[(4)] $S_iX_j=X_jS_i$, for $j\neq i,i+1$,
			\item[(5)]  $E_1X_1^kE_1=\omega_k E_1, \forall k\in \mathbb N$,
			\item[(6)] $S_iE_j=E_jS_i$, for  $|i-j|>1$,
            \item[(7)] $E_iE_j=E_jE_i$, for  $|i-j|>1$,
			\item[(8)]  $E_iX_j =X_j E_i$, for  $j\neq i,i+1$,
			\item [(9)] $X_iX_j=X_jX_i$,
			
			\item[(10)]  $S_iX_i-X_{i+1}S_i=E_i-1$,
			\item[(11)]  $X_i S_i-S_i X_{i+1}=E_i-1 $, 
			\item[(12)] $E_i S_i=E_i=S_iE_i$,
			\item[(13)] $S_iE_{i+1}E_i=S_{i+1}E_i$,
			\item[(14)] $E_i E_{i+1}S_i =E_i S_{i+1}$, 
			\item[(15)] $E_i E_{i+1}E_i =E_{i+1}$,
			\item[(16)] $ E_{i+1}E_i E_{i+1} =E_i$, 
			\item [(17)]  $E_i(X_i+X_{i+1})=(X_i+X_{i+1})E_i=0$,
 			\item[(18)]  $(X_1-u_1)(X_1-u_2)\cdots (X_1-u_a)=0$,
		\end{enumerate}
	\end{multicols} where $\omega_i$ and $u_j$ are scalars in $\mathbb C$ for all $i\in \mathbb N$ and $1\le j\le a$. \end{Defn} 
When $a=1$, this algebra  is the Brauer algebra as defined  in~\cite{Bra}. The decomposition numbers for  the Brauer algebra over $\mathbb C$  were computed in \cite{CDVM, CDVM1}, and  a conceptual explanation (up to a permutation of cell modules) in the framework of Lie theory  was given in \cite{ES1}. 
 
 Throughout this paper, we always  assume $a>1$.  The following result is well-known. 
 \begin{Lemma}\label{anti-inv} There is a $\mathbb C$-linear anti-involution $\tau: \mathcal B_{a, r}(\mathbf u)\rightarrow \mathcal B_{a, r}(\mathbf u)$ fixing  generators $ S_i, E_i$ and $X_j$, for all $1\le i\le r-1$ and $1\le j\le r$. 
    \end{Lemma}

According to  
\cite[Definition~3.6,~Lemma~3.8]{AMR}, the family of scalars 
$\omega= (\omega_i)\in \mathbb C^\mathbb N$ is called \textsf{$\mathbf u$-admissible} if   \begin{equation}\label{uadm} 
 u-\frac{1}{2}+\sum_{i=0}^\infty \frac {\omega_i} {u^i} =(u-\frac{1}{2}(-1)^a)\prod_{i=1}^a \frac{u+u_i}{u-u_i}.\end{equation}
It is proven in \cite[Theorem~5.5]{AMR} that
$\mathcal B_{a, r}(\mathbf u)$ reaches maximal dimension $a^r(2r-1)!!$ if and only if  
 $\omega$  is $\mathbf u$-admissible. Moreover,  
from ~\cite{G09}, we know  the representation theory of $\mathcal B_{a, r}(\mathbf u)$ is fully governed  under the $\mathbf u$-admissible condition. This approach has been applied to classify finite-dimensional simple modules of affine Birman-Murakami-Wenzl algebras over an algebraically closed field~\cite{R}. See~\cite[Remark~3.11]{R} for the result on the affine Brauer algebra.

From this point on,  we always  assume that  \textsf {$\omega$ is  $\mathbf u$-admissible}.

\subsection{Degenerate cyclotomic Hecke algebras} The  degenerate cyclotomic Hecke algebra $\mathscr H_{a, r}(\mathbf u)$ with the parameters $\mathbf u=(u_1, u_2, \ldots, u_a)$ is the   associative algebra generated by elements 
 $s_i$, $x_j$ for  $1\le i\le r-1$, and $1\le j\le r$, subject to the relations: 
\begin{multicols}{2}
		\begin{enumerate}
			\item [(1)] $s_i^2=1$,  
			\item[(2)] $s_is_j=s_js_i$ for  $|i-j|>1$,
			\item[(3)] $s_is_{i+1}s_i\!=\!s_{i+1}s_is_{i+1}$,
			\item[(4)] $s_ix_j=x_js_i$, for  $j\neq i,i+1$,
			\item [(5)] $x_ix_j=x_jx_i$, 
			\item[(6)]  $s_ix_i-x_{i+1}s_i=-1$,
			\item[(7)]  $x_i s_i-s_i x_{i+1}=-1 $,
			\item[(8)]  $(x_1-u_1)(x_1-u_2)\cdots (x_1-u_a)=0$.
		\end{enumerate}
	\end{multicols} 

 Let $\langle E_1\rangle $ be  the two-sided ideal of $\mathcal B_{a, r}(\mathbf u)$ generated by $E_1$. 
It follows from \cite{AMR} that   \begin{equation}\label{cycHiso}\mathcal B_{a, r}(\mathbf u)/\langle E_1\rangle \cong \mathscr H_{a, r}(\mathbf u),\end{equation}
as $\mathbb C$-algebra isomorphism. 
The required isomorphism sends   $\bar S_i$  and  $\bar X_j$ in $\mathcal B_{a, r}(\mathbf u)/\langle E_1\rangle $ to $s_i$ and $x_j$, respectively.

  We adopt the  standard terminology for   compositions,  
  $a$-multipartitions, Young diagrams, tableaux, and  standard tableaux, and related concepts as outlined in   \cite{Ma} and \cite{Kle}. 
  So, $\Lambda^+_a(r)$ denotes the set of all $a$-multipartitions of $r$, and $Y(\lambda)$ (resp., $\Std(\lambda)$) denotes  the Young diagram (resp.,   the set of all standard $\lambda$-tableaux) for every 
   $\lambda\in \Lambda^+_a(r)$. The set $\Lambda_a^+(r)$ is a partially ordered set under the dominance order $\trianglerighteq$ such that $\lambda\trianglerighteq \mu$ indicates 
   $$
   \sum_{t=1}^{s-1}|\lambda^{(t)}|+\sum_{j=1}^h\lambda_h^{(s)}\geq \sum_{t=1}^{s-1}|\mu^{(t)}|+\sum_{j=1}^h\mu_h^{(s)}
   $$
   for all $1\leq s\leq a$ and all $h\geq 0$, where $|\lambda^{(t)}|:=\sum_{j}\lambda_j^{(t)}$.
   There are two special standard  $\lambda$-tableaux $\t^\lambda$ and $\t_\lambda$. 
 For example, if $\lambda=((3,2), (3,1))$, then \begin{equation}\label{tla}
\t^{\lambda}=\left( \ \ \young(123,45),\  \young(678,9)\ \ \right) \quad \text{ and \ }
\t_{\lambda}=\left(\ \ \young(579,68), \ \young(134,2)\ \ \right).\end{equation}

Let $\mathfrak S_r$  be the  symmetric group in $r$ letters $\{1, 2, \cdots, r\}$. 
Then $\mathfrak S_r$  acts on the right of a $\lambda$-tableau by permuting its entries. For example,
\begin{equation}\label{tlaw}
\t^{\lambda}w=\left( \ \ \young(312,45),\  \young(678,9)\ \ \right),\end{equation}
 if  $w=s_1s_2$ and $\lambda=((3,2), (3,1))$.
We write $d(\s)=w$  if $\t^\lambda w=\s$ for any $\lambda$-tableau $\s$. In particular, 
denote $d(\t_\lambda)$ by $w_\lambda$.

For any $\lambda=(\lambda^{(1)}, \lambda^{(2)}, \ldots, \lambda^{(a)})$, 
define \begin{equation}\label{blam} [\lambda]=[b_0, b_1, \cdots, b_a],\end{equation}  where   $b_0=0$ and $b_i=\sum_{j=1}^i |\lambda^{(j)}|$.
We  use   $\mathfrak S_{[\lambda]}$ to denote  $\mathfrak S_{b_1-b_0}\times \cdots \times \mathfrak S_{b_a-b_{a-1}}$, 
 and refer to it as  the Young subgroup with respect to the composition $(b_1-b_0, \ldots, b_a-b_{a-1})$ of  $r$. 
 Let $w_{[\lambda]}\in \mathfrak S_r$  be  defined as
\begin{equation}\label{wll}(b_{i-1}+l)w_{[\lambda]}=r-b_i+l, \text{ for all $i$ with $b_{i-1}<b_i$, $1\le l\le b_i-b_{i-1}$.}\end{equation}
For example, if  $[\lambda]=[0, 4, 8, 9]$, then 
$$w_{[\lambda]}=\begin{pmatrix}  1 & 2 &3 &4 &5 &6 &7 &8 &  9\\
6& 7& 8& 9&  2& 3& 4& 5&  1\\ \end{pmatrix}.$$ 

Define  $w_{(i)}$ such that  $\bft^i  w_{(i)}= \bft_i$, where 
$\bft^i$  denotes the $i$th subtableau of $\bft^\lambda$, and  $\bft_i$ denotes the $i$th subtableau $\bft_\lambda w_{[\lambda]}^{-1}$. Similarly,  define   $\tilde w_{(i)}$ such that  $\tilde \bft^i \tilde w_{(i)}=\tilde \bft_i$. where 
$\tilde \bft^i$ denotes the $i$th subtableau of $\bft^\lambda w_{[\lambda]}$, and  $\tilde \bft_i$ denotes  the $i$th subtableau of  $\bft_\lambda$.  By \cite[(1.4)]{DR},  $ w_{(i)}w_{[\lambda]}=w_{[\lambda]}\tilde w_{(a-i+1)}$, and hence 
\begin{equation}\label{wlaex} w_\lambda=w_{(1)} w_{(2)}\cdots w_{(a)} w_{[\lambda]}=w_{[\lambda]} \tilde w_{(a)} \tilde w_{(a-1)}\cdots \tilde w_{(1)}.\end{equation}

The row stabilizer $\mathfrak S_\lambda$ of $\t^\lambda$ is    the
Young  subgroup $$\mathfrak S_{\lambda}=\mathfrak S_{\lambda^{(1)}}\times \mathfrak S_{\lambda^{(2)}} \times \cdots \times \mathfrak S_{\lambda^{(a)}}  ,$$ where  $\mathfrak S_{\lambda^{(i)}} $ is the row stabilizer of $\t^i$. 
It can also be viewed as the Young subgroup concerning the composition  $\lambda^{(1)}\vee \lambda^{(2)}\cdots \vee \lambda^{(a)}$, obtained from $\lambda$ by concatenation. 
Define \begin{equation}\label{xym} x_{\lambda}= \sum_{w\in \mathfrak S_{\lambda}} w   , \text{ and } \  
 y_{\lambda}=\sum_{w\in \mathfrak S_{\lambda}} (-1)^{l(w)} w   
 \end{equation}
where 
$l(w)$ is the length of $w$. 
For any  $u_1, u_2, \cdots, u_a\in \mathbb C$, and  any $\lambda\in \Lambda^+_a(r)$, define   \begin{equation}\label{piu} \pi_{[\lambda]}=\prod_{i=1}^{a-1} \pi_{b_i}(u_{i+1}), \text{  and   } \tilde{\pi}_{[\lambda]}=\prod_{i=1}^{a-1} \pi_{b_i}(u_{a-i})
,\end{equation}
	where ${b_i}$ is given in \eqref{blam}, $\pi_0(u)=1$, and $\pi_c(u)= (x_1-u)(x_2-u)\cdots (x_c-u)$  if $c$ is a positive integer.  It is known that 
    \begin{equation} \label{pic} \pi_c(u) s_i=s_i\pi_c(u), \text{for any $i\neq c.$}\end{equation}
Let   $ m_\lambda=\pi_{[\lambda]} x_{ \lambda}$,  and   $n_\lambda=\tilde{\pi}_{[\lambda]} y_{ \lambda}$.

\begin{Theorem}\label{basisofhecke}\cite[Theorem~6.3]{AMR}\cite[Theorem~2.1]{RS1} Let $\mathscr H_{a, r}(\mathbf u)$ be the  degenerate cyclotomic Hecke algebra with the parameters $\mathbf u=(u_1, u_2, \ldots, u_a)$. 
\begin{itemize}\item[(1)] $\{m_{\s\t}\mid \s,\t\in\mathscr{T}^{std}(\lambda),\lambda\in \Lambda_a^+(r)\}$ 
  is a cellular basis of $\mathscr H_{a, r}(\mathbf u)$ in the sense of   \cite[Definition~1.1]{GL}, where $m_{\s\t}=d(\s)^{-1} m_\lambda d(\t)$
\item[(2)]  $\{n_{\s\t}\mid \s,\t\in\mathscr{T}^{std}(\lambda),\lambda\in \Lambda_a^+(r)\}$ is a cellular basis of $\mathscr H_{a, r}(\mathbf u)$, where $n_{\s\t}=d(\s)^{-1} n_\lambda d(\t)$.\end{itemize} The required anti-involution is the $\mathbb C$-linear anti-involution that  fixes the  generators $x_1$ and $s_i$ for all $ 1\le i\le r-1$.
\end{Theorem}
  Following \cite{GL}, let  $C(\lambda)$ denote  the cell module of $\mathscr H_{a, r}(\mathbf u) $ concerning the cellular basis in Theorem~\ref{basisofhecke}(1). There is an invariant form, say $\phi_\lambda$ defined on $C(\lambda)$.
  Define $D(\lambda)=C(\lambda)/\text{Rad} \phi_\lambda$.


Suppose that  $u_1, u_2, \ldots, u_a$  are in the same orbit in the sense that $u_i-u_j\in\mathbb Z$ for all $1\leq i< j \leq a$.  By \cite[Theorem 5.4]{Kle}, $D(\lambda)\neq 0$ if and only if $\lambda$ is $\mathbf u$-restricted in the sense  \cite[(3.14)]{Kle}. 
Re-arranging $u_1, u_2, \ldots, u_a$, we can assume 
   $u_i-u_j\in\mathbb N$ for all $1\leq i\leq j \leq a$. By   \cite[{Example 3.2}, Theorem~5.4]{Kle},  $\lambda$ is $\mathbf u$-restricted if and only if 
  \begin{equation}\label{Kles}  \lambda^{(i)}_{u_i-u_{i+1}+j} \le\lambda^{(i+1)}_j\quad 
  \text{for all $j \ge 1$, $a-1\ge  i\ge 1$}.\end{equation}

  When $\mathbf u=( u_1,  u_2, \ldots, u_a)$ is a disjoint union of certain orbits. Write $\mathbf u=\mathbf u_1\cup\ldots \cup\mathbf u_b $ for some $b$ such that   $\mathbf u_i$ and $\mathbf u_j$ are in different  orbits for all $1\le i<j\le b$. Write  $\mathbf u_j=(u_{j_1},\ldots,u_{j_{a_j}})$.
  By the  Morita equivalence theorem \cite[Theorem 1.1, Proposition 4.11(ii)]{DM} for the degenerate cyclotomic Hecke algebra, 
  \begin{equation}\label{difort} \text{$D^\lambda\neq 0$ if and only if each $\lambda_{j}=(\lambda^{(j_1)}, \lambda^{(j_2)}, \ldots, \lambda^{(j_{a_j})})$ is $\mathbf u_j$-restricted }\end{equation} for all $1\le j\le b$, where 
$\lambda=(\lambda^{(1)}, \lambda^{(2)}, \ldots, \lambda^{(a)})$.  See remarks after \cite[Theorem~8.5]{AMR} in which
Ariki, Mathas and Rui  stated that there is a Morita equivalence theorem for degenerate cyclotomic
Hecke algebra, which is   similar to those for cyclotomic Hecke algebra in \cite{DM}.

  Similarly, let $\tilde C(\lambda)$ be the cell module defined via the cellular basis in Theorem~\ref{basisofhecke}(2), and let  \begin{equation}\label{tilD} \tilde D(\lambda)=\tilde C(\lambda)/\text{Rad} \tilde \phi_\lambda,\end{equation} where $\tilde\phi_\lambda$ is the invariant form defined on $\tilde C(\lambda)$. Then  all  non-zero $\tilde D(\lambda)$ also form a complete set of pair-wise non-isomorphic simple $\mathscr H_{a, r}(\mathbf u)$-modules. It follows from \cite[Theorems~5.3, 5.9]{RS1} that 
  \begin{equation}\label{kles1} D(\lambda)\cong \tilde D(\sigma(\lambda))
  \end{equation}
where    $\sigma$  is known as the generalized Mullineux involution.
See \cite[Remark~5.10]{RS1} for an explicit explanation. This involution was obtained in \cite{JC} for the non-degenerate cyclotomic Hecke algebras.

  For each $\lambda\in \Lambda_a^+(r)$, 
the classical Specht module is  $S^\lambda:=m_\lambda w_\lambda n_{\lambda'}\mathscr H_{a, r}(\mathbf u)$, where   $\lambda'$ is  the conjugate of $\lambda$, defined as in Theorem~\ref{main4}.
Then  \begin{equation}\label{spectiso} \tilde C(\lambda')\cong S^\lambda \quad \text{for any $\lambda\in \Lambda_a^+(r)$.}\end{equation} This result was proved in \cite[Theorem~2.9]{DR} for non-degenerate cyclotomic Hecke algebras. The degenerate case can be handled similarly. 

\subsection{A weakly cellular basis of $\mathcal B_{a, r}(\mathbf u)$}  
 For any positive integers $a, r$, define \begin{equation}\label{lambdaa}  \Lambda_{a, r}=\{(f,\lambda)\mid 0\leq f\leq \lfloor r/2 \rfloor, \lambda \in\Lambda_a^+(r-2f)\}. \end{equation}
  There is a  partial order  $\trianglerighteq$ on  the set $ \Lambda_{a, r}$ such that $$\text{$(f,\lambda)\trianglerighteq (h, \mu)$ if $f>h$ or $h=f$ and $\lambda\trianglerighteq  \mu$.}$$ 
 For any  $(f, \lambda) \in \Lambda_{a, r}$, define $\mathbf N_{ a}=\{0,1,\ldots, a-1\}$, and $\delta(f,\lambda)=\Std(\lambda)\times \mathbb N_a^{f}\times   \mathcal D^f_{r}$,  where 
\begin{itemize}\item [(1)] $\mathbb N_{a}^{f}=\{\xi\in \mathbf N_a^r\mid \xi_i \neq 0 \text{ only if } i=r-1,r-3,\cdots,r-2f+1\}$,  \item[(2)] 
$\mathcal D_{r}^f=\{d\in\mathfrak S_r\mid \t^{\tau}d=(\t_1,\t_2)\in \mathscr T^{row, 1}(\tau)\}$,  where $\tau=((r-2f ), (2^f))$, and  $\mathscr T^{row, 1}(\tau)$ is  the set of row standard $\tau$-tableaux such that 
the first column of $\t_2$ is increasing from top to bottom.\end{itemize}

From this point on, unless otherwise stated,   we also use $n_{\s\t}$ to denote the corresponding element in $\B_{a, r}(\mathbf u)$.  More explicitly, it   is obtained from the element  in Theorem~\ref{basisofhecke}   by using $X_i$ and $S_j$ instead of $x_i$ and $s_j$, respectively.

For any $(\s,\xi,e),(\t,\eta,d)\in\delta(f,\lambda)$, define 
 \begin{equation}\label{bncell} C_{(\s,\xi, e),(\t,\eta, d)}=e^{-1}X^{\xi}E^f n_{\s\t}X^{\eta}d,\end{equation}  
	 where $X^\eta=\prod_{i=1}^f X_{r-2i+1}^{\eta_{r-2i+1}} $, 
     $E^0=1$, and $ E^f=E^{f-1} E_{r-2f+1}$ if $f>0$. 
The following result follows from \cite[Theorem~7.17]{AMR}, where Ariki, Mathas and Rui used $m_{\s\t}$ for all admissible $\s$ and $ \t$.

\begin{Theorem}\cite[Theorem~7.17]{AMR}\label{cellular-1} The set $$\{C_{(\s,\xi,e),(\t,\eta,d)}\,|\,
	(\s,\xi,e),(\t,\eta,d)\in\delta(f,\lambda),
	\forall (f,\lambda)\in\Lambda_{a, r}\}$$  is a weakly cellular basis of  $\mathcal B_{a, r}(\mathbf u)$ 
in the sense of \cite{G09}\footnote{The cellular basis of $\mathcal B_{a, r}(\mathbf u)$ is indeed a weakly cellular basis in the sense of \cite{G09}.}.  \end{Theorem}

For each  $(f, \lambda)\in \Lambda_{a, r}$, let $\phi_{f, \lambda}$ be the invariant form defined on $C(f, \lambda)$, where $C(f, \lambda)$ is  the right cell module with respect to the weakly cellular basis  described  in Theorem~\ref{cellular-1}. Define $$D(f, \lambda)=C(f, \lambda)/\text{Rad} \phi_{f, \lambda}.$$ 

\begin{Theorem}\label{bncell1}
    \label{simplecla}  Suppose $(f, \lambda)\in \Lambda_{a, r}$ and $\omega_0\neq 0$. Then 
    \begin{itemize} \item[(1)] $D(f, \lambda)\neq 0$ if and only if $\tilde D(\lambda)\neq 0$.
   \item [(2)] $D(f, \lambda)\neq 0$ if and only if $\sigma^{-1} (\lambda)$ is  $\mathbf u$-restricted in the sense of 
    \eqref{difort}, where $\sigma$ is the generalized Mullineaux involution in \eqref{kles1}. 
    \end{itemize} 
\end{Theorem}
\begin{proof} The statement (1) is a special case of \cite[Theorem~3.12]{RSi1}, 
and (2) follows from  (1) and \eqref{difort}.
\end{proof}

When $\mathcal B_{a, r}(\mathbf u)$ is the cyclotomic Brauer algebra in Theorem~\ref{thmA},  $\omega_0= N$ if $\mathfrak g= \mathfrak{so}_N$ and $\omega_0=-N$ if $\mathfrak g=\mathfrak{sp}_{N}$. This is the reason why we assume $\omega_0\neq0$ in Theorem~\ref{bncell1}.  
The following result holds no matter whether $\omega_0=0$ or not.

\begin{Prop}\label{bas}
    \label{wallbcell} For each $(f, \lambda)\in \Lambda_{a, r}$,  let
	$S^{f,\lambda}=E^f m_{\lambda}w_{\lambda} n_{\lambda'} \B_{a, r}(\mathbf u) \pmod {\langle E^{f+1}\rangle }$, where 
 $\langle E^{f+1}\rangle $ is the two-sided ideal of $\mathcal B_{a, r}(\mathbf u)$ generated by $E^{f+1}$. Then $$S^{f,\lambda}\cong C(f,\lambda')$$ as right $\B_{a, r}(\mathbf u)$-modules. Moreover,   $\{E^f m_{\lambda}w_{\lambda} n_{\lambda'}d(\t)X^{\xi}d \pmod { \langle E^{f+1}\rangle } \mid (\t, \xi ,d)\in \delta(f, \lambda')\}$ forms  a basis of $S^{f,\lambda}$.
	\end{Prop}
\begin{proof}
    The result can be proven using  arguments similar to those used in the proof of  \cite[Proposition~3.9]{RSu}.
    We leave the details to the reader.\end{proof}
Motivated by  Proposition~\ref{bas}, we will classify singular vectors in certain quotient modules of $M_{I_i, r}$ in Section~4. 
   
\section{parabolic category $\mathcal O$ in types  $B_n, C_n$ and $D_n$} 

\subsection{The symplectic and orthogonal Lie algebras}

Throughout, let  $V$ denote   the $N$-dimensional  complex space. The general linear Lie algebra
$\mathfrak{gl}_N$ is defined as $\End_{\mathbb C}(V) $
 with  the Lie bracket $[\  ,   \ ]$,    defined by $[x, y]=xy-yx$ for all $x, y\in \mathfrak {gl}_N$.   
Define  \begin{equation}\label{defofg}
\mathfrak g=\{g\in\mathfrak{gl}_N \mid (gx,y)+(x,gy)=0 \text{ for all $x,y\in V$}\}, 
\end{equation}
where $ (\ ,\ ): V\otimes V\rightarrow \mathbb C$ is   the non-degenerate bilinear  form  on $V\otimes V$ that satisfies
$$(x,y)=\varepsilon(y,x),$$ with  $\epsilon\in \{-1, 1\}$.
When $\epsilon=1$,  $\mathfrak g$ is  the \textsf{orthogonal Lie algebra} $\mathfrak{so}_N$. When $\epsilon=-1$,  $\mfg$ is   the
\textsf{symplectic Lie algebra} $ \mathfrak {sp}_N$, and in this case,  
$N$ has to be even. We denote $\epsilon$ by $\epsilon_{\mathfrak{g}}$ to emphasize the specific Lie algebra.
The natural $\frg$-module $V$  has  a basis  \begin{equation}\label{vbasis}\{v_i\mid i\in \underline N\}\end{equation} such that  
\begin{equation} \label{invri1} (v_i,v_j)=\delta_{i,-j}=\varepsilon_{\mathfrak g}(v_j,v_i), \quad i\ge 0,\end{equation} 
where 
$$\underline N=\begin{cases} (-n, -(n-1), \cdots, -1,  1, \cdots, (n-1), n)   
& \text{if  $N=2n $,}\\ (-n, -(n-1), \cdots, -1, 0, 1, \cdots, (n-1), n)   & \text{if  $N=2n+1 $.}\\ \end{cases} $$
Then $V$ is self  dual with dual basis $\{v_i^\ast\mid i\in \underline N\}$ such that $v_i^\ast (v_j)=\delta_{i, j}$.
Thanks to \eqref{invri1}, \begin{equation}\label{dualb} v_i^\ast=\begin{cases} v_{-i} &\text{if $\frg\neq \mathfrak{sp}_{2n}$,} \\
\text{sgn}(i)v_{-i} & \text{if  $\frg=\mathfrak{sp}_{2n}$,}\\
\end{cases} 
\end{equation}
Let $e_{i, j} $ denote the  matrix unit such that $e_{i, j} v_k=\delta_{j, k} v_i$, and define     
\begin{equation}\label{fij} f_{i,j}=e_{i,j}-\theta_{i,j} e_{-j,-i},\end{equation} where 
$\theta_{i,j}= 1$ if $\mathfrak{g}=\mathfrak {so}_N$, and $\theta_{i,j}=\text{sgn}(i)\text{sgn}(j)$ if $\mathfrak{g}=\mathfrak {sp}_N $. The Lie algebra   $\mathfrak g$  has  basis: 
\begin{equation} \label{rtvec1}\begin{cases} 
 \{f_{i,i}\mid 1\le i\le n\}\cup\{f_{\pm i, \pm j}\mid 1\le i<j\le n\}\cup \{f_{0, \pm i}\mid 1\le i\le n\} & \text{if $\mathfrak g=\mathfrak{so}_{2n+1}$,}\\ 
\{f_{i, i}, f_{-i, i}, f_{i, -i}\mid 1\le i\le n\}\cup \{f_{\pm i, \pm j}\mid 1\le i<j\le n\} & \text{if $\mathfrak g=\mathfrak{sp}_{2n} $,}\\ 
\{f_{i, i}\mid 1\le i\le n\}\cup \{f_{\pm i, \pm j}\mid 1\le i<j\le n\} &\text{if $\mathfrak g=\mathfrak{so}_{2n} $.}\\ \end{cases}
\end{equation}
There is a  standard  triangular decomposition   $$\mathfrak g=\mathfrak n^{-}\oplus \mathfrak h\oplus \mathfrak n^+,$$ where  $\mathfrak h:=   \bigoplus_{i=1}^n \mathbb C h_i$ is the standard Cartan subalgebra  with   
$ h_i=f_{i,i}$, and  $\mathfrak n^+$ has   basis: 
\begin{equation}\label{posrv}
\begin{cases}  \{f_{ i,\pm j}\mid 1\le i<j\le n \}\cup \{f_{0,-i}\mid 1\le i\le n \} & \text{if $\mathfrak{g}=\mathfrak {so}_{2n+1} $,}\\
\{f_{i,-i}\mid 1\le i\le n \}\cup\{f_{ i,\pm j}\mid 1\le i<j\le n \} &\text{if $\mathfrak{g}=\mathfrak {sp}_{2n} $,}\\  \{f_{ i,\pm j}\mid 1\le i<j\le n\} &\text{if $\mathfrak{g}=\mathfrak {so}_{2n}$.}\end{cases} 
\end{equation}
Let $\mathfrak h^*$ be the linear dual  of $\mathfrak h$ with the dual basis $\{\epsilon_i\mid 1\le i\le n\}$ such that  $\epsilon_i(h_j)=\delta_{i, j}$ for all  $1\le i, j\le n$.
The  simple root system   is   $\Pi=\{\alpha_1, \alpha_2, \ldots, \alpha_{n-1}, \alpha_n\}$, where 
 \begin{equation}\label{aln} \alpha_i=\epsilon_i-\epsilon_{i+1},  1\le i\le n-1, \ \text{and} \ \alpha_n=\begin{cases} \epsilon_n &\text{if $\mathfrak g=\mathfrak{so}_{2n+1}  $, }\\
2\epsilon_n &\text{if $\mathfrak g=\mathfrak{sp}_{2n}   $, }\\
\epsilon_{n-1}+\epsilon_n &\text{if $\mathfrak g=\mathfrak{so}_{2n}   $. }\\
\end{cases}
\end{equation}
The root system is $\Phi =\Phi^+ \cup \Phi^-$, where $\Phi^-=-\Phi^+$, and  the  set of positive roots  $\Phi^+$ is  \begin{equation}\label{posrt} \begin{cases} 
\{\varepsilon_i\pm \varepsilon_j\mid 1\leq i<j\leq n\}\cup \{\varepsilon_i\mid 1\leq i\leq n\} & \text{if  $\frg=\frso_{2n+1}  $,} \\
\{\varepsilon_i\pm \varepsilon_j\mid 1\leq i<j\leq n\}\cup \{2\varepsilon_i\mid 1\leq i\leq n\} & \text{if   $\frg=\mathfrak {sp}_{2n}  $,} \\ 
\{\varepsilon_i\pm \varepsilon_j\mid 1\leq i<j\leq n\} &\text{if $\frg=\mathfrak {so}_{2n}   $.}\\
\end{cases}\end{equation}
It is known that  $\Phi$ is of type  $B_n$ (resp.,  $C_n $,  $D_n$)  if $\frg$ is $\mathfrak {so}_{2n+1}$ (resp., $ \mathfrak {sp}_{2n}, \mathfrak{so}_{2n}$).  

An  element $\lambda\in \frh^\ast$  is called a \emph{dominant integral  weight} if  $$\langle \lambda, \alpha^\vee\rangle=\frac{2(\lambda, \alpha)} {(\alpha, \alpha)}\in \mathbb N, \ \ \text{$\forall \alpha\in \Pi$,}$$  where $\alpha^\vee=\frac{2\alpha}{(\alpha, \alpha)}$ is the coroot of $\alpha$ and 
$(\ , \ )$ is the  symmetric bilinear form   on $\mathfrak h^*$ such that $(\epsilon_i, \epsilon_j)=\delta_{i, j}$. 

Let $C$ be  the  quadratic Casimir element in  $\frg$, 
and define 
$$\Omega=\frac {1}{2} ( \Delta(C)-C\otimes 1-1\otimes C)$$ where $\Delta:\U(\mfg)\rightarrow \U(\mfg)\otimes \U(\mfg)$ is the co-multiplication, and $\U(\mathfrak g)$ is the universal enveloping algebra associated with $\frg$.  Then 
\begin{equation} \label{omega1}\Omega=\frac{1}{2} \sum_{i, j\in \underline N} f_{i, j}\otimes f_{j, i}, \end{equation} as shown in  \cite[(2.11)]{DRV}.  

\subsection{Parabolic category $\mathcal O^\frp $ }\label{para}
Let $\frp$ be a  parabolic subalgebra  of  $\frg$   containing  the Borel subalgebra $\frb=\frh\oplus \frn^+$.  
Write $\frp=\frl\oplus\fru$, where $\fru$ is the  nil-radical of $\frp$, and $\frl$ is its  Levi subalgebra.
There exists a unique subset $I\subset\Pi$ such that $\frp=\frp_I$.
Denote $\Phi_{I}= \Phi\cap \mathbb Z I$, and  $\Phi_{I}^+=\Phi^+\cap \mathbb Z I$. 
For any  $\lambda\in\Lambda^{\frp_{I}}$, where $\Lambda^{\frp_I}$ is defined as in \eqref{pdom},  there exists   a unique  irreducible $\mathfrak l$-module $F( \lambda)$, which  can be considered as a $\frp$-module by letting $\mathfrak u$ acting trivially.   The corresponding   \textit{parabolic Verma module} is  
$$
M^\frp(\lambda):=\mathbf U(\frg)\otimes_{\mathbf U(\frp)} F( \lambda).$$ 
Let $L(\lambda)$ be the simple head of $M^\frp(\lambda)$.
Throughout, we fix the following notations. 
\begin{Defn}\label{assum12} Let  $q_1,q_2,\ldots,q_k$ be positive integers  such that $\sum_{j=1}^{k}q_j=n$. 
Denote  \begin{itemize} 
\item [(1)] $I_1$ and $I_2$ by two subsets of $\Pi$ as in \eqref{i1i2d} such that  $ p_j=\sum_{l=1}^jq_l$,   $1\le j\le k$. 
\item[(2)] $\lambda_{I_i, \mathbf c}\in \Lambda^{\mathfrak p_{I_i}}$ as in \eqref{deltac} for any  	$\mathbf c=(c_1, c_2, \ldots, c_k)\in  \mathbb C^k$, where   $c_k=0$ if   
 $i=2$. 
 \item [(3)]  $\mathbf p_j=\{p_{j-1}+1, p_{j-1}+2,\ldots, p_j\}$ for $1\leq j\leq k$, and $p_0=0$.  
\end{itemize} 
\end{Defn}
 In all cases,  $\text{dim}_{\mathbb C } F(\lambda_{I_i, \mathbf c})=1$.

\subsection{Tensor modules in $\mathcal O^{\frp_{I_i}}$}  For any $r\in \mathbb N$,  let 
$M_{I_i, r}\in \mathcal O^{\frp_{I_i}}$ be defined as in \eqref{mir}. 
Following \cite[(4.9)-(4.10)]{RS}\label{uij}, we  define 
 \begin{equation}\label{ujso}\bar u_j=\begin{cases}  c_j-p_{j-1}+n  &\text{if $1\le j\le k$,}\\ 
0 &\text{if $j=k+1$},\\
-c_{2k-j+2} +p_{2k-j+2}-n &\text{if $k+2\le j\le 2k+1$,}\\
\end{cases}
\end{equation} if $\Phi$ is $B_n$, and    \begin{equation}\label{ujsp}\bar u_j=\begin{cases} \epsilon_\frg (c_j-p_{j-1}+ n-\frac12 \epsilon_\frg) &\text{if $1\le j\le k$,}\\ 
\epsilon_\frg (-c_{2k-j+1} +p_{2k-j+1}-n+\frac 12\epsilon_\frg) &\text{if $k+1\le j\le 2k$,}\\
\end{cases}
\end{equation}
if $\Phi$ is either $C_n$ or $D_n$. 
From this point on, we always assume that  \begin{equation}\label{higge} \mathbf{m}_i \text{  is the highest weight vector of $M^{\frp_{I_i}} (\lambda_{I_i, \mathbf c})$, up to a scalar.}\end{equation} 

\begin{Prop}\label{polyofx}\cite[Lemmas~4.11-4.12]{RS}  
	\begin{itemize} \item [(1)] Suppose  $\Phi\in \{C_n, D_n\}$. There is  a parabolic Verma flag   
			$$		0=N_0\subset N_1\subset N_2\subset \cdots\subset N_{2k}= M^{\frp_{I_i}} (\lambda_{I_i, \mathbf c})   \otimes V$$ of $ M^{\frp_{I_i}} (\lambda_{I_i, \mathbf c}) \otimes V$    
		such that $$N_j/N_{j-1}\cong \begin{cases} 
		M^{\mathfrak p_i}(\lambda_{I_i, \mathbf c}+\varepsilon_{p_{j-1}+1}) & \text{if $1\leq j\leq k$,} \\
		\delta_{i, 1} M^{\mathfrak p_i}(\lambda_{I_i, \mathbf c}-\varepsilon_{p_{k}}) & \text  {if $j=k+1$,}\\		
		M^{\mathfrak p_i}( \lambda_{I_i, \mathbf c}-\varepsilon_{p_{2k+1-j}})  & \text{if $k+2\leq j\leq 2k$.}\\
		\end{cases}  $$
		Moreover, $\mathbf{m}_i\otimes v_j\in N_t$,  and $\mathbf{m}_i\otimes v_{-j}\in N_{2k+1-t}$ if  $j\in \mathbf p_t$ for some $t\le k$.

		\item [(2)] Suppose  $\Phi=B_n$. There is  a parabolic Verma flag 		 $$ 0=N_0\subset N_1\subset N_2\subset \cdots\subset N_{2k+1}=  M^{\frp_{I_i}} (\lambda_{I_i, \mathbf c})   \otimes V$$  of 
		 $ M^{\frp_{I_i}} (\lambda_{I_i, \mathbf c})   \otimes V$ such that 
		$$N_j/N_{j-1}\cong \begin{cases} 
		M^{\mathfrak p_i}(\lambda_{I_i, \mathbf c} +\varepsilon_{p_{j-1}+1}) & \text{if $1\leq j\leq k$,} \\
		\delta_{i, 1}  M^{\mathfrak p_i}(  \lambda_{I_i, \mathbf c}-\delta_{j, k+2}\varepsilon_{p_{k}}) &\text{if $k+1\le j\le k+2$, } \\	
		M^{\mathfrak p_i}(\lambda_{I_i, \mathbf c}-\varepsilon_{p_{2k+2-j}})  & \text{if $k+3\leq j\leq 2k+1$.}\\
		\end{cases}
		$$
Moreover,  $\mathbf{m}_i\otimes v_j\in N_t$,  and $\mathbf{m}_i\otimes v_{-j}\in N_{2k+2-t}$ if $j\in \mathbf p_t$ for some $t\le k$. 
  \end{itemize}
	In both  cases, $\prod_{j=1}^l (X_1-u_j)$ acts on 
$N_{l+c\delta_{i, 2}(1+\delta_{\frg, \mathfrak{so}_{2n+1}})  }$  trivially for all admissible $l$, where
\begin{equation}\label{ujjj}
u_{j}=\begin{cases}   \bar u_j &\text{if  $i=1$, $1\le j\le 2k+\delta_{\frg, \mathfrak{so}_{2k+1}}$,}\\
 \bar u_j &\text{if  $i=2$, $1\le j\le k$,}\\
 \bar u_{j+1+\delta_{\frg, \mathfrak{so}_{2n+1}}} &\text{if  $i=2$, $k+1\le j\le 2k-1$, }\\ \end{cases}
\ \ \text{  and }\ \  c=\begin{cases} 0 &\text{if $l\le k-1$,}\\
1 &\text{otherwise.}\\
\end{cases} 
\end{equation}
In  particular, 
	$f_i(X_1)$ acts on  $M^{\frp_{I_i}} (\lambda_{I_i, \mathbf c})   \otimes V   $  trivially, where 
 \begin{equation}\label{f1f2} f_1(X_1) =\prod_{j=1}^{2k+\delta_{\mathfrak g, \mathfrak{so}_{2n+1}}} (X_1-u_j) \text{  and $f_2(X_1) =\prod_{j=1}^{2k-1} (X_1-u_j)$.}\end{equation} \end{Prop}


It follows from Theorem~\ref{thmA} that $M_{I_i, r}$ is a $(\U(\mfg), \mathcal B_{a, r}(\mathbf u))$-bimodule. 
Furthermore,  from    \cite{RS},    $E^f $ acts on $M_{I_i, r}$ using 
\begin{equation}\label{eiact} E^f:=
(\text{Id}_{M_{I_i,  r-2f}} \otimes \alpha^{\otimes f})\circ(\text{Id}_{M_{I_i, r-2f}} \otimes \beta^{\otimes f})\end{equation} for any $0\le f\le \lfloor r/2\rfloor$,
where $\alpha: \mathbb C\rightarrow V^{\otimes 2}$ is  the co-evaluation map, and  $\beta: V^{\otimes 2}\rightarrow \mathbb C$ is  the evaluation map. These maps satisfy 
\begin{equation} \label{alp123} \alpha(1)=\sum_{i\in \underline N} v_{i}\otimes v_i^\ast, \quad 
\beta (u\otimes v)=(u, v),\end{equation}
for all $u, v\in V$, where $(\ , \ )$ denotes  the non-degenerate  bilinear form satisfying \eqref{invri1}, and $v_i^\ast$ represents  the dual basis element  in \eqref{dualb}.

For any $M\in \mathcal O^{\frp_{I_i}}$, we denote $[M: L(\lambda)]$ the multiplicity of the  simple $\frg$-module $L(\lambda)$ in a composition series of $M$. 

From this point to the end of this section, we keep  condition~\eqref{simple111}.  
Consequently,  $\scI_{i, j}$ is saturated for any $0\le j\le r$.  Notably, this condition is well-justified by Theorem~\ref{saturated}. For details, see the Appendix by Wei Xiao. 

\begin{Lemma} \label{ext2} For any  {$\nu\in\Lambda^{\mathfrak p_{I_i}}$},  $[M_{I_i, r} \langle E^f \rangle : L(\nu)]=0$ unless $\nu\in \mathscr {I}_{i, r-2f}$. 
\end{Lemma} 
\begin{proof} Notably,   $ \text{Id}_{M_{I_i, r-2f} } \otimes \alpha^{\otimes f}$ 
can be considered as a morphism in  
$\text{Hom}_{\mathcal O^{\frp_{I_i}}} ( M_{I_i, r-2f}, M_{I_i, r} )$. By Theorem~\ref{thmA}, any element in  $ \mathcal B_{a, r}(\mathbf u)^{op} \circ (\text{Id}_{M_{I_i, \otimes r-2f}} \otimes \alpha^{\otimes f}) $
can also be viewed as morphism in $\text{Hom}_{\mathcal O^{\frp_{I_i}}} ( M_{I_i, r-2f}, M_{I_i, r}  )$. This implies that the  composition factor of the image of such a morphism   has to be a composition factor of $M_{I_i, r-2f}$.
Since 
	$$M_{I_i, r} \langle E^f \rangle \subseteq (\mathcal B_{a, r}(\mathbf u)^{op} \circ\text{Id}_{M_{I_i,r-2f}} \otimes \alpha^{\otimes f})M_{I_i, r-2f}
 ,$$   any composition factor $L(\nu)$ of $M_{I_i, r} \langle E^f \rangle $ has to be a composition factor of $M_{I_i, r-2f}$, forcing $\nu\in \mathscr{I}_{i, r-2f}$  by condition ~\eqref{simple111}. 
 \end{proof}

\begin{Lemma} \label{ext1} Suppose $\mu\in\Lambda^{\mathfrak p_{I_i}}$ and $X\in \mathcal O^{\frp_{I_i}}$. If  $\text{Ext}^1_{\mathcal O^{\frp_{I_i}} } (M^{\frp_{I_i}}  (\mu), X)\neq 0$, then $X$ has a composition factor $L(\nu)$ satisfying $\mu\prec \nu$.
\end{Lemma}	
\begin{proof}First, we assume that $X$ is simple  in $\mathcal O^{\frp_{I_i}}$. Then $X=L(\nu)$ for some $\nu\in\Lambda^{\mathfrak p_{I_i}}$.  There is a 
	 short exact sequence $$0\rightarrow M \rightarrow P_{I_i} (\mu)\rightarrow M^{\frp_{I_i}}(\mu)\rightarrow 0$$ where $P_{I_i}(\mu)$ is the projective cover of $L(\mu)$.  Applying  $\text{Hom}_{\mathcal O^{\frp_{I_i}}} (-, L(\nu))$ to  the short exact sequence, and noting that $\text{Ext}^1_{\mathcal O^{\frp_{I_i}}} (P_{I_i}  (\mu), L(\nu))=0$, and 
  $\text{Ext}^1_{\mathcal O^{\frp_{I_i}}} (M^{\frp_{I_i}}  (\mu), X)\neq 0$,  we  obtain   
	$$\text{Hom}_{\mathcal O^{\frp_{I_i}}} 	(M, L(\nu))\neq 0.$$ 
 From~\cite[Theorem~9.8]{Hum},   $M$ has a parabolic Verma flag such that each subquotient is of form $M^{\frp_{I_i}}(\xi)$ satisfying $\xi > \mu$, where  $\geq$ is the dominance order defined on $\mathfrak h^*$. 
 
 If the length of the parabolic Verma flag  is $1$, then $M=M^{\frp_{I_i}}(\nu)$. Otherwise, there is a short exact sequence 
 $$0\rightarrow M_1\rightarrow M\rightarrow M^{\frp_{I_i}}(\gamma)\rightarrow 0$$
where $M_1$ has a parabolic Verma flag of  shorter  length. If $\gamma=\nu$,   $(M: M^{\frp_{I_i}} (\nu))\neq 0$. Otherwise, 
applying the functor $\Hom_{\mathcal O^{\frp_{I_i}}} (-, L(\nu))$ to the short exact sequence, we obtain 
$$0\rightarrow \Hom_{\mathcal O^{\frp_{I_i}}} (M, L(\nu))\rightarrow \Hom_{\mathcal O^{\frp_{I_i}}} (M_1, L(\nu)),$$ which makes $\Hom_{\mathcal O^{\frp_{I_i}}} (M_1, L(\nu))\neq 0$.  By the induction assumption on the length of a parabolic Verma flag of  $M_1$, we have  $(M_1: M^{\frp_{I_i}} (\nu))\neq 0$. 

In all cases,  
 $(M: M^{\frp_{I_i}} (\nu))\neq 0$. From \cite[Theorem~9.8]{Hum},    
	$$[M^{\frp_{I_i}} (\nu): L(\mu)]=(P_{I_i}(\mu):  M^{\frp_{I_i}}(\nu))\neq 0,$$ forcing $\mu\prec \nu$.  
 
 Suppose  $X$ is not simple. Then   there is a short exact sequence 
	\begin{equation}\label{ssho} 0\rightarrow X_1\rightarrow X\rightarrow L(\nu)\rightarrow 0\end{equation} for some $\nu\in\Lambda^{\mathfrak p_{I_i}}$.  
	Applying  $\text{Hom}_{\mathcal O^{\frp_{I_i}}} (M^{\frp_{I_i}} (\mu), -)$ to \eqref{ssho}, and noting that $$\text{Ext}^1_{\mathcal O^{\frp_{I_i}}} (M^{\frp_{I_i}}  (\mu), X)\neq 0,$$   we conclude that 
	either $\text{Ext}^1_{\mathcal O^{\frp_{I_i}}} (M^{\frp_{I_i}}  (\mu), L(\nu))\neq 0$    or  $\text{Ext}^1_{\mathcal O^{\frp_{I_i}}} (M^{\frp_{I_i}}  (\mu), X_1)\neq 0$. In the first case, we have already established  the result. In the second case,   
	the result follows from  standard arguments using  the  inductive assumption on the length of a composition series of $X$. \end{proof}

\textbf{Proof of Theorem~\ref{main111}:}  If $	\Hom_{\mathcal O^{\frp_{I_i}}} ( M^{\frp_{I_i}}(\mu),  M_{I_i, r} \langle E^f \rangle )\neq 0$,  then $L(\mu)$ has to be  a composition factor of  $M_{I_i, r} \langle E^f \rangle$.   If $\text{Ext}^1_{\mathcal O^{\frp_{I_i}} } (M^{\frp_{I_i}}  (\mu), M_{I_i, r} \langle E^f \rangle )\neq 0$, by 
Lemma~\ref{ext1}, $M_{I_i, r}\langle E^f\rangle $ has a composition factor 
$\nu$ such that $\mu\prec \nu$. In all cases, since we keep  condition \eqref{simple111}, by  Lemmas~\ref{ext2}--\ref{ext1},   $\mu\in \mathscr I_{i, r-2f}$,   a contradiction.
 So 
	\begin{equation}\label{key1} \Hom_{\mathcal O^{\frp_{I_i}}} ( M^{\frp_{I_i}}(\mu),  M_{I_i, r} \langle E^f \rangle )=\text{Ext}^1_{\mathcal O^{\frp_{I_i}} } (M^{\frp_{I_i}}  (\mu), M_{I_i, r} \langle E^f \rangle )=0.\end{equation} 
	Now, applying the functor  $\text{Hom}_{\mathcal O^{\frp_{I_i}}} ( M^{\frp_{I_i}} (\mu), -)$ to  the following  short exact sequence
	$$0\rightarrow  M_{I_i, r} \langle E^f \rangle \rightarrow M_{I_i, r}\rightarrow M_{I_i, r}/  M_{I_i, r} \langle E^f \rangle \rightarrow  0$$
	of $(\U(\frg), \mathcal B_{a, r}(\mathbf u))$-bimodules, 
	we have  Theorem~\ref{main111}, as required. \qed

\section{Classification of singular vectors in $M_{I_i, r}/M_{I_i, r} \langle E^f\rangle $} 
This section aims to classify singular vectors in $M_{I_i, r}/M_{I_i, r} \langle E^f\rangle $ for any $0{<} f\le \lfloor r/2\rfloor$, where  $I_1$ and  $I_2$ are  defined as in Definition~\ref{assum12}. Importantly,  we will need   Theorem~\ref{main111} to compute the dimensional of $\Hom_{\mathcal O^{\frp_{I_i}}}(M^{\frp_{I_i}}(\mu), M_{I_i, r}/M_{I_i, r} \langle E^f\rangle) $. This is the only place  that we need   condition~\eqref{simple111}. 

For any   integer $j$ and positive integer  $l$, we denote $(j)^l$ by   $\overset{l}  {\overbrace {j, j, \ldots, j}}$.  If $l=0$, we denote  $(j)^l=\emptyset$. The following definition is well-defined since we keep  Assumption~\ref{keyassu}. This implies that $p_t-p_{t-1}\ge 2r$, $1\le t\le k$,  where $p_j$'s are defined as in Definition~\ref{assum12}. 
\begin{Defn}\label{vt}  For any $\lambda\in\Lambda_{a}^+(r-2f)$, 
 define $\mathbf i_\lambda=(\mathbf i_{\lambda^{(1)}}, \mathbf i_{\lambda^{(2)}},\cdots, \mathbf i_{\lambda^{(a)}})\in \underline{N}^{r-2f}$ such that 
$$\mathbf i_{\lambda^{(j)}}=\begin{cases} (( p_{j-1}+1)^{\lambda^{(j)}_{1}}, (p_{j-1}+2)^{\lambda^{(j)}_{2}},\cdots, ( p_{j})^{\lambda^{(j)}_r})  &\text{if $1\le j\le k$,}\\  ((- p_{2k-j+\delta_{i, 1}})^{\lambda^{(j)}_{1}}, 
  \cdots, ( -p_{2k-j+\delta_{i, 1}}+r-1)^{\lambda^{(j)}_r})  &\text{if $k+1\le j\le a$.
}\\   
\end{cases}
   $$ 
\end{Defn}
 
  For any  $\mathbf i=(i_1, i_2, \ldots, i_{r-2f})\in \underline{N}^{r-2f}$, define \begin{equation}\label{vboldi} v_\mathbf i=v_{i_1}\otimes v_{i_2}\otimes\cdots \otimes v_{i_{r-2f}}\end{equation} where $\{v_j\mid j\in \underline N\}$ is the  basis of the natural $\frg$-module $V$ in \eqref{vbasis}. Then, the weight of  $v_{\mathbf i_\lambda}$ is 
  \begin{equation}\label{weightp} \tilde \lambda =\sum_{j=1}^k \sum_{l=1}^r \lambda_l^{(j)} \varepsilon_{p_{j-1}+l}-
\sum_{j=k+1}^a  \sum_{l=1}^r \lambda_l^{(j)} \varepsilon_{p_{2k-j+\delta_{i, 1}} -l+1}.  
\end{equation}  
  For any $(f, \lambda)\in \Lambda_{a, r}$, define \begin{equation}\label{vlam} v_{\lambda}=\mathbf m_i\otimes v_{\mathbf i_\lambda} \otimes (v_1\otimes v_{-1})^{\otimes f}, \end{equation}  
where $\mathbf m_i$ is defined as in \eqref{higge}. Then the weight of $v_\lambda$ is 
\begin{equation}\label{hatlambda}\hat \lambda:=\lambda_{I_i, \mathbf c}+\tilde \lambda    
\end{equation}
where $\lambda_{I_i, \mathbf c}$ is defined in Definition~\ref{assum12}(2). For any $ \lambda,  \mu\in \Lambda_a^+(r-2f)$,  by 
 \eqref{weightp}  we have \begin{equation}\label{lm} \lambda \trianglerighteq \mu\ \  \text{if and only if} \ \   \hat \lambda\geq \hat \mu.
\end{equation}
The following definition of  $ v_{\t, \xi, d}$ is motivated by the basis of $S^{f, \lambda}$ in Proposition~\ref{bas}.
\begin{Lemma} \label{vtd} For any $(\t, \xi, d)\in \delta(f, \lambda')$, define $ v_{\t, \xi, d}= v_\lambda 
E^f w_\lambda n_{\lambda'}d(\t) X^\xi d$.
Then  $v_{\t, \xi, d}$ has  weight $\hat\lambda$.

\end{Lemma} 
\begin{proof} By Theorem~\ref{thmA}, $M_{I_i, r}$ is a $(\mathfrak g, \mathcal B_{a, r}(\mathbf u))$-bimodule. Consequently,  $ v_{\t, \xi, d}$ and  $v_\lambda$ have the same weight, which completes the proof. 
\end{proof}

\begin{Lemma}\label{liehwv} Let $V$ be the  natural $\mathfrak{gl}_n$-module  with basis $\{v_j\mid 1\le j\le n\}$.  Then the linear dual  $W$ of $V$ has  dual basis 
$\{v^\ast_j\mid 1\le j\le n\}$ defined by $v^\ast_j (v_l)=\delta_{j, l}$. If $n\ge r$, then
there exists   a bijection between the set of  dominant weights of $V^{\otimes r}$ (resp., $W^{\otimes r}$)  and $\Lambda_1^+(r)$. Furthermore,  
the $\mathbb C$-space of highest weight vectors in $V^{\otimes r}$ (resp.,  $W^{\otimes r}$) with the highest weight $\lambda:=\sum_{i=1}^r \lambda_i\varepsilon_i$ (resp.,  $\lambda^\ast:=-\sum_{i=1}^r \lambda_i \varepsilon_{n-i+1}$) has   basis 
	$\{ v_{\mathbf i_\lambda} w_\lambda n_{\lambda'} d(\t)\mid \t\in \Std(\lambda')\}$ (resp., $\{ v^\ast_{\mathbf j_\lambda} w_\lambda n_{\lambda'} d(\t)\mid \t\in \Std(\lambda')\}$ where $\mathbf i_\lambda=((1)^{\lambda_1}, \ldots, (r)^{\lambda_r})$ and $\mathbf j_\lambda=((n)^{\lambda_1}, \ldots, (n+r-1)^{\lambda_r})$. 
\end{Lemma}
\begin{proof}
    By setting   either $r=0$ or $s=0$ in  \cite[Proposition~4.10, Lemma~4.11]{RS-wall}, we have the corresponding result for $\U_q(\mathfrak {gl}_n)$, where  $\U_q(\mathfrak {gl}_n)$ is the quantum general linear group. For $\mathfrak{gl}_n$, one can handle it similarly.  \end{proof}
Restricting $V^{\otimes r}$ and $W^{\otimes r} $ to $\mathfrak{sl}_n$, the results concerning the highest weight vectors in Lemma~\ref{liehwv} remain valid.   
     Let $V$ be the natural  $\mathfrak g$-module, where  $\frg\in \{\mathfrak {so}_{2n+1}, \mathfrak{sp}_{2n},   
 \mathfrak {so_{2n}} \}$. Then, we have the following isomorphism of  $\mathfrak{sl}_n$-modules   \begin{equation}\label{isoo} \bigoplus_{i=1}^n \mathbb C v_{-i}\cong W.\end{equation}
   The required  isomorphism sends $v_{-i}$ to $v_i^\ast$, as described in Lemma~\ref{liehwv}.

\begin{Prop}\label{hiofcyche}   
 For any $(\t, \xi, d)\in \delta(f, \lambda')$, $\bar{ v_{\t, \xi, d}}\in M_{I_i, r}/M_{I_i, r} \langle E^{f+1} \rangle  $ is annihilated by any element in    the positive part  $\n^+$ of $\frg$.  
\end{Prop}

\begin{proof} By Theorem~\ref{thmA},  $M_{I_i, r}$ is a $(\mathfrak g, \B_{a, r}(\mathbf u))$-bimodule, and so is  $M_{I_i, r} /M_{I_i, r}\langle E^{f+1}\rangle $. Therefore, it suffices  to prove that $\bar{v_\lambda E^f w_\lambda n_{\lambda'}}$ is annihilated by the root  vectors in $\mathfrak n^+$ corresponding to the simple roots in  \eqref{posrt}, where $v_\lambda$ is defined as in \eqref{vlam}.

\Case {1. $f=0$ and $i=2$}   The root vector in $\n^+$  corresponding to $\alpha_n$  is  
  $f_{n, -n}$ (respectively, $ f_{n-1, -n}$, and  $f_{0, -n}$) if $\Phi$ is $C_n$ (resp., $ D_n, B_n$). 
By Definition~\ref{vt}, 
$v_{-n}$ does not appear as a tensor factor of    $v_{\mathbf i_\lambda}$ if  $\Phi$ is either  $B_n$ or $C_n$. When $\Phi=D_n$,   neither $v_{n}$ nor $v_{-n+1}$  appears as a tensor factor of    $v_{\mathbf i_\lambda}$. 
 Thus, $v_\lambda$   is annihilated by  such a root vector, and 
so is ${v_\lambda w_\lambda n_{\lambda'}}$.

It remains to consider the root vectors $f_{j, j+1}$  corresponding  to  $\alpha_j$,  $1\le j\le n-1$. There are two cases to discuss. 

(1)  $f_{j,j+1}\in\mathfrak l$. 

By slightly abusing of notations, we consider  $\tilde \pi_{[\lambda']}$ and $ y_{\lambda'}$ in $\B_{a, r}(\mathbf u)$, obtained from those in 
$\mathscr H_{a, r}(\mathbf u)$ by using $X_t$ and $S_j$ instead of $x_t$ and $s_j$, respectively. Since $i=2$, we have  $a=2k-1$ by  \eqref{defa}.
From \eqref{cycHiso}, 
\begin{equation} \label{kkk1} \tilde \pi_{[\lambda']} y_{\lambda'}\equiv y_{\lambda'}\tilde \pi_{[\lambda']} \pmod {\langle E^1\rangle }, 
\end{equation} 
Let   $\mu^{(t)}$ be  the conjugate of $\lambda^{(t)}$. 
We have 
$$\begin{aligned} v_\lambda w_\lambda n_{\lambda'} &  = \mathbf m_i \otimes v_{\mathbf i_\lambda}  w_{(1)}w_{(2)}\cdots w_{(a)} w_{[\lambda]}\tilde \pi_{[\lambda']} y_{\lambda'}  \ \  \text{by \eqref{wlaex}} \\
&  \equiv \mathbf m_i \otimes v_{\mathbf i_\lambda}  w_{(1)}w_{(2)}\cdots w_{(a)} w_{[\lambda]} y_{\lambda'}\tilde \pi_{[\lambda']} \pmod {M_{I_i, r}  \langle E^1\rangle} \ \ \text{by \eqref{kkk1}} \\ 
& \equiv \mathbf m_i \otimes v_{\mathbf i_\lambda}  w_{(1)}w_{(2)}\cdots w_{(a)} y_{\mu^{(1)}\vee \mu^{(2)}\vee \cdots \vee \mu^{(a)}} 
w_{{[\lambda] }} \tilde \pi_{[\lambda']}
\pmod {M_{I_i, r} \langle E^1\rangle} \  \text{by \eqref{wll}}.   \\
\end{aligned} $$

Since  $f_{j,j+1}\in\mathfrak n^+\cap \mathfrak l$, and the special linear Lie algebra $\mathfrak{sl}_n$ is a subalgebra of $\mathfrak{gl}_n$,   by  Lemma~\ref{liehwv} for  $\mathfrak{sl}_n$, and \eqref{isoo}, we have $$f_{j, j+1} (\mathbf m_i \otimes v_{\mathbf i_\lambda}  w_{(1)}w_{(2)}\cdots w_{(a)} y_{\mu^{(1)}\vee 
\mu^{(2)}\vee \cdots \vee \mu^{(a)}})=0 ,$$ forcing   $f_{j, j+1} \overline{\mathbf m_i \otimes v_{\lambda}  w_\lambda n_{\lambda'}}=\bar 0 $.

(2) $f_{j, j+1}\not\in \frl$. 

Then $j=p_l$ for some $1\le l\le k-1$.  We claim  $$f_{p_l,p_l+1}v_{\lambda} w_\lambda n_{\lambda'}\in M_{I_i, r}\langle E^1\rangle.$$  This   is trivial if  neither   $v_{p_l+1}$ nor $v_{-p_l}$  appears as a tensor factor of    $v_{\mathbf i_\lambda}$.  In this case, we have  $f_{p_l,p_l+1}v_{\lambda} w_\lambda n_{\lambda'}=0$. Otherwise, by Definition~\ref{vt}, 	at least one of  $v_{p_l+1}$ or  $v_{-p_l}$ must  appear, which forces  at least  one of 
 $\lambda^{(l+1)}$ and  $\lambda^{(a-l+1)}$ to be non-empty.  
 
 To verify the claim,  
 we write $$ f_{p_j,p_j+1} v_\lambda w_\lambda n_{\lambda'} =(1-\delta_{\lambda^{(j+1)}, \emptyset})A+(\delta_{\lambda^{(a-j+1)}, \emptyset} -1) B,$$ 
where \begin{equation}\label{abc} \begin{aligned} A&=\sum_{c=1}^{ \lambda^{(l+1)}_1} \mathbf m_i\otimes  v_{\mathbf i_{\lambda^{(1)}}}\otimes \cdots\otimes v_{\mathbf i_{\lambda^{(l)}}}\otimes v_{\mathbf i_c}\otimes v
	_{\mathbf i_{\lambda^{(l+2)}}}\otimes \cdots\otimes v_{\mathbf i_{\lambda^{(a)}}} w_\lambda n_{\lambda'}\\
 B&=\sum_{b=1}^{ \lambda^{(a-l+1)}_1} \mathbf m_i\otimes  v_{\mathbf i_{\lambda^{(1)}}}\otimes \cdots\otimes v_{\mathbf i_{\lambda^{(a-l)}}}\otimes v_{\mathbf i_b}\otimes v
	_{\mathbf i_{\lambda^{(a-l+2)}}}\otimes \cdots\otimes v_{\mathbf i_{\lambda^{(a)}}} w_\lambda n_{\lambda'}.\\
  \end{aligned}\end{equation} 
 Here    $\mathbf i_c$  is obtained from $\mathbf i_{\lambda^{(l+1)}}$ by replacing   $p_l+1$  with $p_l$ at $(b_l+c)$-th position, and the $\mathbf i_b$ is obtained from $\mathbf i_{\lambda^{(a-l+1)}}$ by replacing $-p_l$ with
 $-(p_l+1)$  at $(b_{a-l}+b)$-th position, where $b_l$ is defined in \eqref{blam}.
 Thus, it suffices to verify $A, B\in M_{I_i, r} \langle E^{1}\rangle $. 
  We provide a detailed proof for  $A$ and   a  brief for  $B$ since the 
  arguments are similar. 
  
   Suppose   $\lambda^{(l+1)}\neq \emptyset$. Let  $\mathbf a$  be  obtained from $\mathbf i_\lambda$ by 
 replacing $\mathbf i_{\lambda^{(l+1)}}$ with $ \mathbf i_1$, where $\mathbf i_1$ is defined as in  the expression of $A$ in  \eqref{abc}. It is well-known (see e.g. \cite[(2.13)]{RS}) that each $s_t\in \mathfrak S_r$ acts on $V^{\otimes r} $ using a sign permutation  if $\Phi$ is of type $C_n$, and a permutation if $\Phi$ is of type  $B_n$ or $ D_n$.  Thus,  we have  \begin{equation}\label{k1} A=(-1)^{\delta_{\frg, \mathfrak{sp}_{2n}}} \mathbf m_i\otimes  v_{\mathbf a} \sum_{p=1}^{\lambda^{(l+1)}_1} (b_l+1, b_l+p)w_\lambda n_{\lambda'}. \end{equation}
 Define  $$h=\sum_{p=1}^{\lambda^{(l
+1)}_1}  (b_l+1, b_l+p) w_{\lambda^{(1)}} \cdots w_{\lambda^{(a)}}\in \mathbb C \mathfrak S_{[\lambda]}.$$ Thanks to \eqref{wll}, $h w_{[\lambda]}=w_{[\lambda]} h_1$ for some  $h_1\in \mathbb C \mathfrak S_{[\lambda']}$. Using \eqref{kkk1} and \eqref{wll}, we have  (up to a sign)  \begin{equation}\label{k12} A 
\equiv \mathbf m_i \otimes v_{\mathbf i} (1, r-b_{l+1}+1)^2  \tilde \pi_{[\lambda']} h_1 y_{\lambda'}  
 \pmod {M_{I_i, r}  \langle E^1\rangle }\end{equation}
 where  $\mathbf i=(\mathbf i_{\lambda^{(a)}},\ldots, \mathbf i_{\lambda^{(l+2)}}, \mathbf i_1, 
 \mathbf i_{\lambda^{(l)}}, \ldots, \mathbf i_{\lambda^{(1)} })$.  
Labeling  $\mathbf m_i$  at the $0$-th position,
the tensor factor of $\mathbf m_i \otimes v_{\mathbf i} (1, r-b_{l+1}+1)$ at the $1$-th  position is $v_{p_l}$ 
 Since $\lambda^{(l+1)}\neq \emptyset$, $r-b_{l+1}< r-b_l$. By  \eqref{pic} we have 
$$(1, r-b_{l+1}+1) \tilde \pi_{[\lambda']} \equiv (X_1-u_1) (X_1-u_2)\cdots (X_1-u_l) h_2 \pmod {\langle E^1\rangle }$$
for some $h_2\in \B_{a, r}(\mathbf u)$. As $\textbf{m}_i\otimes v_{{p_l}}\in N_l$ (defined as in Proposition~\ref{polyofx}),
by Proposition~\ref{polyofx},  $\mathbf m_i\otimes v_{\mathbf i} (1, r-b_{l+1}+1)$   is annihilated by 
$ (X_1-u_1) (X_1-u_2) \cdots (X_1-u_l)$, which makes 
$A\in M_{I_i, r} \langle E^{1}\rangle $. 

For $\lambda^{(a+1-l)}\neq \emptyset$, we replace $l$ with   $a-l$ in the arguments above.  Consequently, we obtain 
 the corresponding expression  for $B$ by substituting  $a-l$ for  $l$ in \eqref{k1}. The corresponding $h$ is 
 $$h=\sum_{p=1}^{\lambda^{(a-l+1)}_1}  (b_{a-l}+1, b_{a-l} +p) w_{\lambda^{(1)}} \cdots w_{\lambda^{(a)}}.$$
 We still have $h w_{[\lambda]}=w_{[\lambda]} h_1$, where $h_1\in \mathbb C\mathfrak S_{[\lambda']}$. Therefore, the  resulting analog to   \eqref{k12} holds with   $l$ replaced by $a-l$.   
 
 In this case, the tensor factor of $\mathbf m_i \otimes v_{\mathbf i} (1, r-b_{a-l+1}+1)$ at the $1$-th  position is $v_{-p_l-1}$. Since $\lambda^{(a-l+1)}\neq \emptyset$, $r-b_{a-l+1}< r-b_{a-l}$. Consequently,  there exists  $h_2\in \B_{a, r}(\mathbf u)$ such that 
$$(1, r-b_{a-l+1}+1) \tilde \pi_{[\lambda']} \equiv (X_1-u_1) (X_1-u_2)\cdots (X_1-u_{a-l }) h_2 \pmod {\langle E^1\rangle }.$$
Since $p_l+1\in \mathbf p_{l+1}$, we have  $\mathbf m_i\otimes v_{-p_l-1} \in N_{2k-l+\delta_{\mathfrak g, \mathfrak{so}_{2n+1} }}$.  
By Proposition~\ref{polyofx}, $${\mathbf m_i\otimes v_{\mathbf i}} (1, r-b_{a-l+1}+1) \prod_{i=1}^{a-l} (X_1-u_i)=0,$$
 which makes  $B\in M_{I_i, r}  \langle E^1\rangle $. 
 This completes the proof for $f=0$ and $i=2$.

 \Case{2. $f=0$ and $i=1$}  In this case, we have $\Phi\neq B_n$ and $\alpha_n\not\in I_1$. By  \eqref{defa}, $a=2k$. 

Since the arguments used in the proof of (1) and (2) in Case~1  depend only on   whether the simple root is in $I_2$ or not,    one can verify that 
$\bar{v_\lambda E^f w_\lambda n_{\lambda'}}$ is annihilated by the root  vectors corresponding to the simple roots in  \eqref{posrt}, similarly.  The difference here is that we need to use arguments from  the proof of (2) specifically to handle  the root vector corresponding to $\alpha_n$ since $\alpha_n\not\in I_1$. We leave details to the reader.

 \Case {3. $f>0$} Since $E^{f}$ acts on the tenor factors of $M_{I_i, r}$ labeled by $r-2f+1, r-2f+2, \ldots, r-1, r$, by abusing of notion, we  have    
  $$\alpha^{\otimes f} (1^{\otimes f})=(v_1\otimes v_{-1})^{\otimes f} E^f ,$$ where $\alpha$ is defined as in \eqref{alp123}. 
  This gives rises to   a $\frg$-homomorphism  $$\psi:= Id_{M_{I_i, {r-2f}}}\otimes \alpha^{\otimes f}: M_{I_i, {r-2f}}\rightarrow M_{I_i, r}\langle E^f\rangle \hookrightarrow M_{I_i, r}.$$ 
  The restriction of $\psi$
   to $M_{I_i, r-2f} \langle E_{r-2f-1} \rangle $ maps  $M_{I_i, r-2f} \langle E_{r-2f-1} \rangle $ to $M_{I_i, r}  \langle E^{f+1}\rangle$. It induces a $\frg$-homomorphism 
  $$\bar{\psi}  : M_{I_i, r-2f}/ M_{I_i, r-2f}\langle E_{r-2f-1}\rangle  \rightarrow M_{I_i, r}\langle E^f\rangle /M_{I_i, r} \langle E^{f+1}\rangle\hookrightarrow M_{I_i, r}/M_{I_i, r} \langle E^{f+1}\rangle, $$ which maps   $\bar{\mathbf m_i \otimes v_{\mathbf i_\lambda}  w_\lambda n_{\lambda'}} $ to 
  $\bar {v_\lambda E^f w_\lambda n_{\lambda'}}$. By previous results established in Case 1,  $\bar{\mathbf m_i \otimes v_{\mathbf i_\lambda}  w_\lambda n_{\lambda'}} $ is annihilated by any element in $\n^+$. Consequently,       
  $\bar {v_\lambda E^f w_\lambda n_{\lambda'}}$ is also annihilated by any element in $\mathfrak n^+$.  
  \end{proof}

We establish some preliminary results before proving that all the elements in Proposition~\ref{hiofcyche} are linearly independent, as stated in Theorem~\ref{main123}. 

For any $\beta=-\sum_{\gamma\in \Pi} b_\gamma \gamma\in -\mathbb N\Pi $, define 
 \begin{equation} \label{inequality2}  |\beta|_j=\begin{cases} \sum_{\gamma\in \Pi\setminus I_i} b_\gamma & \text{if $j=1$,}\\ 
\text{max}
\{t\mid t\in c_\beta\cup \{0\} 
\} &\text{if $j=2$,}\\
\end{cases}
\end{equation}  
where \begin{equation}\label{cbta} c_\beta=\left\{\sum_{\alpha\in\Phi^+\setminus \Phi_{I_i}} a_\alpha\mid \beta=-\sum_{\alpha\in \Phi^+\setminus \Phi_{I_i}} a_\alpha \alpha, \text{ and } a_\alpha\in \mathbb N \right \}.  \end{equation}

\begin{Lemma}
    \label{inequality1}For any  $\alpha, \beta\in -\mathbb N\Pi$,   
      $|\alpha+\beta|_1=|\alpha|_1+|\beta|_1$, and  $|\beta|_1\ge |\beta|_2$. 
\end{Lemma} 
\begin{proof} The first equality follows from \eqref{inequality2}. When
 $c_\beta=\emptyset$, the second result is trivial. If $c_\beta\neq \emptyset$, then $|\beta|_2=j$ for some positive integer $j$.  We can write 
 $\beta=-\sum_{\gamma\in \Phi^+\setminus \Phi_{I_i}} b_\gamma\gamma$ for some $b_\gamma\in \mathbb N$ such that 
 $\sum_{\gamma\in \Phi^+\setminus \Phi_{I_i}} b_\gamma =j$.  
For such a $\gamma$,  $\gamma=\sum_{\eta\in \Pi} c_{\gamma, \eta} \eta$, such that $c_{\gamma, \eta_0}\neq 0$ for some  $\eta_0\in \Pi\setminus I_i$, forcing   $|\beta|_1\ge \sum_{\gamma\in \Phi^+\setminus \Phi_{I_i} }  b_\gamma c_{\gamma, \eta_0 }\ge \sum_{\gamma\in \Phi^+\setminus \Phi_{I_i}} b_\gamma=|\beta|_2$.  \end{proof}

Following~\cite[Definition~4.4]{RS}, we define 
 \begin{equation}\label{bii} \mathcal  B_{I_i}=\{f_{ -j, k}, f_{-j, -k}\mid 1\le j<k \le n, \text{ and } \varepsilon_j\pm\varepsilon_k\in \Phi^+\setminus \Phi_{I_i}\}\cup T_\Phi,\end{equation}
where $I_i$ is defined as  in Definition~\ref{assum12}, and $$
T_\Phi=\begin{cases} \emptyset, &\text{if $\Phi=D_n$,}\\
\{f_{0,j}\mid 1\le j\le n, \text{ and } \varepsilon_j\in \Phi^+\setminus \Phi_{I_i} \} & 
\text{if  $\Phi=B_n$,}\\
\{f_{-j,j}\mid 1\le j\le n, \text{ and } 2\varepsilon_j\in \Phi^+\setminus \Phi_{I_i}\} &
\text{if $\Phi= C_n$.}\\
\end{cases} 
$$
It is known that  $\mathcal  B_{I_i}$ forms  a basis for  $\mathfrak u^-_{I_i}$.
For any  $\mathbf l=(l_1, l_{1}, \ldots, l_b)\in \mathbb N^b$ and any positive integer $b$, 
we denote $$f_{{\mathbf i, \mathbf j}}^\mathbf l :=f_{i_1,j_1}^{l_1}f_{i_{2},j_{2}}^{l_{2}}\cdots f_{i_b,j_b}^{l_b}$$  if
 $f_{i_lj_l}\in \mathcal B_{I_i}$, $1\le l\le b$. Here $\mathbf i=(i_1, i_{2}, \ldots, i_b)$ and 
 $\mathbf j=(j_1, j_{2}, \ldots, j_b)$.   If $b=0$, we set $ f_{{\mathbf i, \mathbf j}}^\mathbf l=1$.
  Fix a total order  $\prec$  on  $\mathcal B_{I_i}$, and 
let $$\mathcal M_{I_i}=\{f_{{\mathbf i, \mathbf j}}^\mathbf l\mid  f_{i_{l+1},j_{l+1}}\prec f_{i_{l}, j_{l}} \text{ for  $1\le l\le b-1$, and  $\mathbf l\in \mathbb N^b,  b\in \mathbb N$}\}.$$
It follows from \cite[Lemma~4.5]{RS}  that  $M_{I_i, r}$ has  basis  
\begin{equation} \label{tsmodule} \mathcal S_{i,r}=\{f_{{\mathbf i, \mathbf j}}^\mathbf l \textbf{m}_i\otimes v_{\mathbf k}\mid f_{{\mathbf i, \mathbf j}}^\mathbf l\in \mathcal M_{I_i},  \mathbf k\in \underline N^r\},\end{equation}  where  $\textbf{m}_i$ is defined as in \eqref{higge}.
For any $j\in \mathbb N$, let 
\begin{equation} \label{grade1} M_{I_i, r}^{\le j}:=\mathbb C\text{-span} \{f_{\mathbf i, \mathbf j}^\mathbf l \textbf{m}_i\otimes v_{\mathbf k}\mid f_{\mathbf i, \mathbf j}^\mathbf l\in \mathcal M_{I_i}, \mathbf k\in \underline N^r, |\mathbf l|\le j \},\end{equation} where $|\mathbf l|:=\sum_tl_t$. Similarly  $M_{I_i, r}^{<j}$ is defined analogously by replacing the condition         $ |\mathbf{l}| \leq j$  with $ |\mathbf{l}| < j $.

\begin{Defn} \label{ppt} For any  $\mathbf i\in \underline N^r$, define $\text{deg} v_{\mathbf i} =\sum_{j=1}^r \text{deg} v_{i_j}$, where $\text{deg} v_{i_j} =t-1$  and 
 $\text{deg} v_{-i_j} =a-t$ if $i_j\in \mathbf p_t$, and $\text{deg} v_0=\frac12 (a-1)$. In the latter case, $\Phi=B_n$, and $i=2$.   \end{Defn} 

\begin{Lemma}
    \label{filt1} 
     For any $\lambda\in \Lambda_a^+(r-2f)$,    $ M_{I_i, r, \hat \lambda}\subseteq  M_{I_i, r}^{\le \text{deg} v_{\mathbf i_\lambda} +f(a-1)} $,  where $M_{I_i, r, \hat \lambda}$ is the $\hat\lambda$-weight space of $M_{I_i, r}$. 
    \end{Lemma}
    
\begin{proof}  Suppose  $f_{{\mathbf i, \mathbf j}}^\mathbf l\textbf{m}_i\otimes v_{\mathbf k}\in \mathcal S_{i, r}$. Then each $f_{i_s, j_s}$ is a root vector in $\U(\frg)^-$ with respect to a positive root, say $\beta_s\in \Phi^+\setminus \Phi_{I_i}$ for  $ 1\le s\le b$, and $b\in \mathbb N$. If  the weight of $f_{{\mathbf i, \mathbf j}}^\mathbf l\textbf{m}_i\otimes v_{\mathbf k}$ is $\hat\lambda$, then   $-\sum_{s=1}^b l_s \beta_s=\tilde \lambda-\text{wt} v_{\mathbf k}$. By \eqref{inequality2}, and   Lemma~\ref{inequality1}, we have 
\begin{equation}\label{degine1} |\mathbf l| \le 
    |\tilde\lambda -\text{wt} (v_{\mathbf k})|_2 \le |\tilde\lambda -\text{wt} (v_{\mathbf k})|_1.\end{equation}
Since 
    $\text{wt}(v_{\mathbf k}) -r\varepsilon_1\in -\mathbb N\Pi$ for any admissible $\mathbf k$,  it follows from Lemma~\ref{inequality1} that  
    \begin{equation} \label{define2} |\tilde \lambda-\text{wt} (v_\mathbf k)|_1=|\tilde \lambda-r\varepsilon_1|_1 - |\text{wt}(v_{\mathbf k})-r\varepsilon_1|_1.\end{equation} 
 When $\Phi=B_n$ and  $i=2$,  we have \begin{equation}\label{iden33}|-\epsilon_1|_1= \frac12 (a-1)=\text{deg}~ v_0, 
\end{equation} 
where $a$ is defined as in \eqref{defa}. 
    If $t\in \mathbf p_j$ for some $j$, it follows from Definition~\ref{ppt} and \eqref{inequality2} that   \begin{equation} \label{iden3} |-(\varepsilon_1+\varepsilon_t)|_1=\text{deg} v_{-t}=a-j \text{ and 
$|-(\varepsilon_1-\varepsilon_t)|_1=\text{deg} v_{t}=j-1$.}\end{equation} 
Write $x=|\tilde \lambda-r\varepsilon_1|_1 $.  Then,  
 $$\begin{aligned} x \overset{\eqref{weightp}} =&  |\sum_{j=1}^k \sum_{s=1}^r \lambda_s^{(j)} (\varepsilon_{p_{j-1}+s}-\varepsilon_1)  -\sum_{j=k+1}^a \sum_{s=1}^r \lambda_s^{(j)} (\varepsilon_{p_{2k-j+\delta_{i, 1}-s+1}}+\varepsilon_1) -2f\epsilon_1|_1\\
\overset{(1)} =&  \sum_{j=1}^k \sum_{s=1}^r \lambda_s^{(j)} |(\varepsilon_{p_{j-1}+s}-\varepsilon_1)|_1  +\sum_{j=k+1}^a \sum_{s=1}^r \lambda_s^{(j)} |(-\varepsilon_{p_{2k-j+\delta_{i, 1}-s+1}}-\varepsilon_1)|_1 +f|-2\varepsilon_1|_1\\
\overset{(2)}=& \sum_{j=1}^k \sum_{s=1}^r \lambda_s^{(j)} \text{deg}\  v_{p_{j-1}+s}  +\sum_{j=k+1}^a \sum_{s=1}^r \lambda_s^{(j)}\text{deg}\  v_{-(p_{2k-j+\delta_{i, 1}-s+1})} +f|-2\varepsilon_1|_1\\
\overset{(3)}=& \text{deg}\ v_{\mathbf i_\lambda}+f|-2\varepsilon_1|_1=\text{deg}\ v_{\mathbf i_\lambda}+f(a-1).\\ \end{aligned} 
 $$
Here (1) follows from  Lemma~\ref{inequality1},  (2) is a consequence of  \eqref{iden3}, and (3) follows from Definition~\ref{vt} and \eqref{iden3}. On the other hand, we have \begin{equation}\label{sst} |\text{wt} (v_{\mathbf k})-r\varepsilon_1|_1 =\sum_{t=1}^r |\text{wt} (v_{k_t})-\varepsilon_1|_1=\sum_{t=1}^r \text{deg}\ v_{k_t}=\text{deg}\ v_{\mathbf k},\end{equation} where the second equality  follows from   \eqref{iden33}--\eqref{iden3}.  Thanks to \eqref{degine1}--\eqref{define2}, we have 
$$|\mathbf l|\le |\tilde \lambda-r\varepsilon_1|_1-|\text{wt}(v_{\mathbf k}) -r\varepsilon_1|_1=\text{deg} \ v_{\mathbf i_\lambda} +f(a-1)-\text{deg} \ v_{\mathbf k} \le \text{deg}v_{\mathbf i_\lambda}+f(a-1) .$$
 Thus,  the required inclusion follows immediately
from 
\eqref{grade1}.  
\end{proof}

\begin{Lemma}\label{eM filtered de }\cite[Lemma~4.6]{RS}
    Suppose that   $h,l \in \underline N$ and $j\in\mathbb N$.
     Then  $$f_{ h,  l}(M^{\mathfrak p_{I_i}}(\lambda_{I_i, \mathbf c})^{\leq j} )\subseteq M^{\mathfrak p_{I_i}}(\lambda_{I_i, \mathbf c})^{\leq x}, $$  where 
     $x=j+1$ if  $ f_{h,l}\in \mathfrak u_{I_i}^-$,  and $j$ if $ f_{h,l}\notin \mathfrak u_{I_i}^-$.
\end{Lemma}

Suppose that  $y_1,y_2 $ are two PBW monomials  in  $\U(\mathfrak u^-_{I_i})$. Following \cite[p537, line -8]{RS},  we write  $y_1 \approx y_2$ if $y_1$ can be obtained from   $y_2$ by permuting its  factors. From \cite[(4.30)]{RS},   
 \begin{equation}\label{RS430}
    y_1\mathbf m_i=y_2 \mathbf m_i ,
\end{equation} up to a linear combination of terms with  lower degree  if $y_1 \approx y_2$.

We say that an element $w\in \mathcal S_{i, r}$ is a term of an element $v\in M_{I_i, r}$ if  when   $v$ is expressed as a linear combination of elements in $\mathcal S_{i, r}$, $w$ appears with a non-zero coefficient. 
\begin{Lemma}\label{high1} 
    If $f_{{\mathbf i, \mathbf j}}^{\mathbf l}\mathbf{m}_i\otimes v_{\mathbf k}\in \mathcal S_{i, r}$, and $(t_1, \ldots, t_r)\in \mathbb N^r$, then 
 \begin{equation}\label{incbasic} f_{{\mathbf i, \mathbf j}}^\mathbf l\mathbf{m}_i\otimes v_{\mathbf k}\prod_{s=1}^r X_s^{t_s}\in M_{I_i, r}^{\le |\mathbf l|+\sum_{s=1}^r t_s}.\end{equation} In particular, when  $r=1$, we have  
    
    \begin{itemize}
    \item [(1)]   $|\mathbf l'|\le |\mathbf l| +t_1$ if     
    $f_{{\mathbf i', \mathbf j'}}^{\mathbf l'}\mathbf{m}_i\otimes v_{\mathbf k'}$  is a term of  $f_{{\mathbf i, \mathbf j}}^\mathbf l\mathbf{m}_i\otimes v_{\mathbf k} X_1^{t_1}$. The   equality holds  if and only if $f_{{\mathbf i', \mathbf j'}}^{\mathbf l'}=\widetilde {f_{{\mathbf i, \mathbf j}}^{\mathbf l} y}$ for some $y=\overset{\rightarrow} \prod_{j=1}^{t_1}f_{a_j, b_j}\in \U(\mathfrak u_{I_i}^-)$   such that $\overset{\rightarrow}\prod_{j=1}^{t_1}  f_{b_j, a_j} v_{\mathbf k}=\pm v_{\mathbf k'}$.
    Here $\widetilde {f_{{\mathbf i, \mathbf j}}^{\mathbf l} y}$  is the unique element in $\mathcal M_{I_i}$
    satisfying $f_{{\mathbf i, \mathbf j}}^{\mathbf l} y\approx \widetilde {f_{{\mathbf i, \mathbf j}}^{\mathbf l} y}$.
    \item[(2)]
    $f_{{\mathbf i, \mathbf j}}^\mathbf l\mathbf{m}_i\otimes v_{\mathbf k}X_1^{t_1} \in M_{I_i, r}^{\le |\mathbf l|+t_1-1   }$, if   $\text{deg}v_{\mathbf k}<t_1$.      
    \end{itemize} \end{Lemma}
\begin{proof} 
From \cite[(3.17)]{RS},  $X_j$ acts on 
$(M^{\frp_{I_i}}(\lambda_{I, c})\otimes V^{\otimes j-1})\otimes V$ using  $\epsilon_\frg (\Omega+\frac12(N-\epsilon_\frg))$, where $\Omega$ is defined as in \eqref{omega1}. Thus, \eqref{incbasic} follows immediately from Lemma~\ref{eM filtered de }. (1) follows from Lemma~\ref{eM filtered de } and \eqref{RS430}.

    If (2) were false, we would have  $ f_{\mathbf i, \mathbf j}^\mathbf l \textbf{m}_i\otimes v_{\mathbf k}X_1^{t_1} \notin M_{I_i, r} ^{\le |\mathbf l|+t_1-1}$. Then,    there exists  
$\overset{\rightarrow} \prod_{j=1}^{t_1} f_{g_j, h_j}\in \U(\mathfrak u_{I_i}^-)$ such that $\overset{\rightarrow} \prod_{j=1}^{t_1} f_{h_j, g_j} v_{\mathbf k}=\pm v_{\mathbf k'}$ for some $\mathbf k'\in \underline{N}$. Thus, we have $$\text{wt} (v_{\mathbf k'})-\varepsilon_1 +\sum_{j=1}^{t_1} \text{wt} (f_{g_j, h_j})=\text{wt}(v_{\mathbf k})-\varepsilon_1. $$ Using  \eqref{iden33}--\eqref{iden3} and noting that $r=1$, we have    
$$|\sum_{j=1}^{t_1} \text{wt} (f_{g_j, h_j})|_1=
|\text{wt} (v_{\mathbf k})-\varepsilon_1      |_1- |\text{wt} (v_{\mathbf k'})-\varepsilon_1|_1  =\text{deg} v_{\mathbf k}-\text{deg} v_{\mathbf k'}.$$ 
Since   $f_{g_j, h_j}\in \U(\mathfrak u_{I_i}^-)$, for all $1\le j\le t_1$, we have $$\text{deg} \ v_{\mathbf k}\ge |\sum_{j=1}^{t_1} \text{wt} (f_{g_j, h_j})|_1=\sum_{j=1}^{t_1} |\text{wt} (f_{g_j, h_j})|_1\ge t_1,$$ which leads to a contradiction. This completes the proof of    (2). 
\end{proof}

\begin{Lemma}\label{lc1} For any   $\lambda\in \Lambda_a^+(r-2f)$, let $\prod_{j=1}^{r-2f} X_j^{a_j}$ be the unique term in $\tilde \pi_{[\lambda']}$ such that $\sum_{j=1}^{r-2f} a_j$ is maximal.  Then \begin{equation}\label{lckey} l_c\in \begin{cases}  \mathbf p_{a_c+1} &\text{if $a_c<k$,}\\ 
-\mathbf p_{2k+\delta_{i, 1}-a_c-1} &\text{if 
$a_c\ge k$,}\\
\end{cases} \end{equation} where  $l_1, l_2, \ldots, l_{r-2f}$ are defined by     $\mathbf l:=\mathbf i_\lambda w_\lambda=(l_1, l_2, \ldots, l_{r-2f}) $. 
\end{Lemma}
\begin{proof}
Since $\lambda'$ represents the conjugate of $\lambda$, we have $[\lambda'] = [b_a - b_a, b_a - b_{a-1}, \ldots, b_a - b_0]$ if  $[\lambda] = [b_0, b_1, \ldots, b_a]$, as in \eqref{blam}. Here   $b_0 = 0$ and $b_a = r - 2f$. 
For each $c$, $1\le c\le r-2f$, there exists a unique $j$ such that \begin{equation}\label{accond} b_a - b_{a-j} \geq c > b_a - b_{a-j+1}, \text{ and } a_c = a - j. \end{equation} The last equality in \eqref{accond} follows from   \eqref{piu}.
Denote 
\begin{equation}\label{ilam123} 
    \mathbf{i}_\lambda = (i_1, i_2, \ldots, i_{b_1}, \ldots, i_{b_{a-1} + 1}, \ldots, i_{b_a}),\end{equation} where $\mathbf{i}_\lambda$ is defined as in Definition~\ref{vt}.
    Then we have:
\begin{equation} \label{ibt}
i_{b_t + l} \in \begin{cases}
\mathbf{p}_{t+1} &\text{if $t < k$,} \\
-\mathbf{p}_{2k + \delta_{i, 1} - t - 1} &\text{if $k  \leq t < a$,} \\
\end{cases}
\end{equation}
for all $1 \leq l \leq b_{t+1} - b_t$.

Let $\mathbf l=\mathbf{i}_\lambda w_{\lambda}$ and $\mathbf{l}' = 
\mathbf{i}_\lambda w_{[\lambda']}$,
where $w_\lambda$ and $w_{[\lambda']}$ are defined as in \eqref{wlaex}. If $b_a - b_{t + 1} < s \leq b_a - b_{t}$, then
$l'_s = i_{b_{t} + s - b_a + b_{t + 1}}$.
From \eqref{ibt}, it follows that
\begin{equation}\label{lll1} l'_s \in \begin{cases}
\mathbf{p}_{a_c + 1} &\text{if $a_c < k$,} \\
-\mathbf{p}_{2k + \delta_{i, 1} - a_c - 1} &\text{if $a_c\ge k$.} \\
\end{cases}\end{equation}
Here $i$ is either $1$ or $2$.
Since  $\mathbf{l} = \mathbf{l'} \tilde{w}_{(a)} \cdots \tilde{w}_{(1)}$ , we have that for any $b_a - b_{t + 1} < s \leq b_a - b_{t}$ 
\begin{equation}\label{lll12} l_s \in \begin{cases}
\mathbf{p}_{t + 1} &\text{if $t < k$,} \\
-\mathbf{p}_{2k + \delta_{i, 1} - t - 1} &\text{if $  t\ge k$,} \\
\end{cases}\end{equation}
where $\tilde w_{(j)}$ is defined as in \eqref{wlaex}. Now, \eqref{lckey}  follows immediately from \eqref{accond}, and 
\eqref{lll12}.
\end{proof}

From this point to the end of this section, we fix  $a_j, l_j$, $1\le j\le r-2f$ as those in Lemma~\ref{lc1}. 
For any $1\le c\le r-2f$ {such that $a_c\geq k$}, 
denote
\begin{equation}\label{AB} \begin{cases} A_c & = f_{-(z_c+p_{k-1}), -l_c+\sum_{{s}=0}^{a_c-k-1} q_{2k-a_c+\delta_{i, 1}+s}    } \\
B_c & =\overleftarrow {\prod}_{t=0}^{a_c-k-1} f_{-l_c+\sum_{{s}=0}^{t} q_{2k-a_c+\delta_{i, 1}+s}, -l_c+\sum_{s=0}^{t-1} q_{2k-a_c-\delta_{i, 1}+s}}\\
\end{cases} \end{equation}  
where 
$z_c=1+l_c+p_{2k-(a_c+1)+\delta_{i, 1}}$, and   $i$ is either $1$ or $2$, and $q_1, q_2,\ldots, q_k$ are defined as in Definition~\ref{assum12}. 

\begin{Defn}\label{defj} 
For any $\lambda \in \Lambda_a^+(r-2f)$, we  define    $\mathbf j=(j_1, j_2, \ldots, j_{r-2f})$, and  $y_{{l_c}, a_c, c}$,  $1\le c\le r-2f$ such that 
 \begin{equation}\label{jc} j_c=\begin{cases}
        l_c-  p_{a_c}+ b_{a_c}  & \text{if $a_c<k$,}\\
       1+ l_c+ p_{2k-a_c-1+\delta_{i, 1}} + b_{a_c}  & \text{if $a_c\ge k$,}\\        \end{cases}  \end{equation}
   and     
\begin{equation}\label{ylc} y_{l_c, a_c, c}=\begin{cases} 
1 &\text{if $a_c=0,$}\\
f_{l_c-\sum_{t=1}^{a_c-1}  q_{a_c-t+1},j_c }\times\overleftarrow{ \prod}_{s=1}^{a_c-1} f_{l_c-\sum_{t=1}^{s-1} q_{a_c-t+1} , l_c-\sum_{t=1}^s  q_{a_c-t+1}} &\text{if $0<a_c<k$,}\\ 
f_{-j_c, -(z_c+p_{1})}\times \overleftarrow{\prod}_{s=1}^{k-2} f_{-(z_c+p_{k-s-1}), -(z_c+p_{k-s})} \times A_c \times B_c &\text{if $a_c\ge k$,}  
\end{cases} \end{equation}
where  $[\lambda]=[b_0, b_1, \ldots, b_a]$, and $z_c$, $A_c$,  and $B_c$ are defined as in \eqref{AB}. 
\end{Defn}

\begin{Defn}\label{defrho} For any $\xi\in \mathbb N_a^{f}$  and any integer $s$ such that $1\leq s\leq f$, denote
\begin{multicols}{2}
\begin{enumerate}
\item [(1)]$\xi_{r, s}=\xi_{r-2f+2s-1}$,
\item[(2)]  $z=r-f+s$,
\item[(3)]$A=\overleftarrow{\prod}_{t=1}^{\xi_{r, s}-k} 
f_{p_{k-t-\delta_{i, 2}}+z, p_{k-t-1-\delta_{i, 2}}+z}$,
\item[(4)] $B=\overleftarrow {\prod}_{t=1}^{k-1} f_{p_t+z-f, p_{t-1}+z-f}$.
\end{enumerate}
\end{multicols}
Define $\mathbf j^\xi=(j_1^\xi, j_2^\xi, \ldots, j_{2f}^\xi)$, and $y_{\xi, 1}, y_{\xi, 2}, \ldots, y_{\xi, f}$, where 
\begin{equation}\label{yxis} y_{\xi, s}=\begin{cases}
1 &\text{if $\xi_{r, s}=0$,}\\
\overleftarrow{\prod}_{t=1}^{ \xi_{r, s}} f_{p_{t}+z-f, p_{t-1}+z-f} &\text{if $\xi_{r, s}\le k-1$, }\\
A  \cdot f_{-p_{k-1}{-z+f}, p_{k-1-\delta_{i, 2}}+z
 }\cdot B  &\text{if $k\le \xi_{r, s}$,}
\end{cases}\end{equation} and 
\begin{equation}\label{jxi}  j_{2s-l}^\xi= \begin{cases}
   r-2f+s &\text{if  $l=0$,}\\ 
    -p_{ \xi_{r, s}}-r+2f-s &\text{if $l=1$, and  $0\le\xi_{r, s}\le k-1$,}\\    
    r-f+s+p_{2k-1-\delta_{i,2}-\xi_{r, s}} &\text{if $l=1$, and   $ k\le \xi_{r, s}$.}\\    
    \end{cases}\end{equation}
 \end{Defn}

Since we keep Assumption~\ref{keyassu}, $\mathbf j^\xi\in \underline {N}^{2f}$. 

\begin{Lemma}\label{lc2}
    Suppose $(f, \lambda)\in\Lambda_{a, r}$, and $\xi\in \mathbb N_a^{f}$, and $1\le   c_1, c_2\le r-2f$, and $1\le c\le 2f$.  Then 
    \begin{multicols}
        {2}
    
    \begin{itemize}\item[(1)] $j_{c_1}=j_{c_2}$ if and only if $l_{c_1}=l_{c_2}$,
    \item [(2)] $1\le j_{c_1}\le r$,
    \item [(3)] $j_{c_1}^\xi\neq j_{c_2}^\xi$ if  $c_1<c_2$, 
    \item [(4)] $j_{c}^\xi\neq j_{c_2}$,  
    \item [(5)] $1\le j_c^\xi\le r$  if 
    $\xi_{r-2l+1}=a-1,  1\le l\le f$,
    \item [(6)] $\text{deg}\ v_{\mathbf i_\lambda}=\text{deg}\ v_{\mathbf l}=\sum_{t=1}^{r-2f} a_t$, 
    \end{itemize} \end{multicols}
    where $j_c$, $j_c^\xi$ and $\mathbf l$ are defined as in \eqref{jc}, \eqref{jxi} and Lemma~\ref{lc1}, respectively. 


\end{Lemma}
\begin{proof}
    (1)-(5) follow  from Definition~\ref{defrho} and (6) follows  from Lemma~\ref{lc1}. \end{proof}

\begin{example}\label{rem1} Suppose $i=1$,  $a=2k=4$, and $(q_1, q_2, r)=(20, 21, 10)$,  $\xi=(0^6, 1,0,3,0)$
and $\lambda=((0), (2), (2,1), (1))\in \Lambda_4^+(6)$. Then $\lambda'=((1), (2,1), (1,1), (0))$. The  term in $\tilde \pi_{[\lambda']}$ with the highest degree is $X_1^3X_2^2X_3^2 X_4^2X_5X_6$. We have 
\begin{multicols}{2}\item  $\t^\lambda=(\emptyset,\  \young(12),\  \young(34,5), \ \young(6))$, 
\item  $\mathbf i_\lambda=(21, 21, -41, -41, -40, -20)$, \item  $\mathbf l=\mathbf i_\lambda w_\lambda=(-20, -41, -40, -41, 21, 21)$,\item 
  $(a_1, a_2, \ldots, a_6)=(3,2,2,2,1,1)$, \item $\mathbf j=(6,3,4,3,1,1)$, \item $\mathbf j^\xi=(-27, 7, 10, 8)$, \item $y_{l_1, a_1, 1}=f_{-6, -21}f_{-21, 41} f_{41, 20}$, \item $y_{l_2, a_2, 2}=
    y_{l_4, a_4, 4}    =f_{-3, -21}f_{-21, 41}$, \item  $y_{l_3, a_3, 3}=f_{-4, -20}f_{-20, 40} $,\item  $y_{l_5, a_5, 5}=y_{l_6, a_6, 6}=f_{21, 1}$, \item $y_{\xi, 1}=f_{27, 7}$,
    \item  $y_{\xi, 2}=f_{30, 10} f_{-28, 30} f_{28, 8}$.\end{multicols} 
\end{example}

\begin{Lemma}\label{key21}
    Suppose $(y, \xi, w)\in \mathfrak S_{\lambda'}\times \mathbb N_a^{f}\times H_f$, where $(f, \lambda)\in \Lambda_{a, r}$, and $H_f$ 
is  the subgroup of $\mathfrak S_{r}$  generated by 
$\{s_{r-1}, s_{r-2}s_{r-1}s_{r-3}s_{r-2}, \ldots, s_{r-2f+2}s_{r-2f+1}s_{r-2f+3}s_{r-2f+2} \}$. Denote 
 $\mathbf j^{\lambda, \xi}=(\mathbf j, \mathbf j^{\xi})$.
For any $(\s, d), (\t, e)\in $ $\Std(\lambda')\times \mathcal D_r^f$, we have 
   \begin{itemize} \item[(1)]  $\mathbf j yd(\t)=\mathbf jd(\s)$ if and only if $\t=\s$ and $y=1$.
     \item[(2)]    Suppose $f\neq 0$. Then  $ \mathbf j^{\lambda, \xi}   d(\s)d=\mathbf j^{\lambda, \xi} wyd(\t) e$ if and only if $w=y=1$ and $(
\s, d)=(\t, {e})$.  \end{itemize}
    \end{Lemma}
\begin{proof}
    Clearly,  the  ``if part'' of both  statements hold. Conversely, we have $\mathbf j=\mathbf j yd(\t)d(\s)^{-1} $. By  Lemma~\ref{lc2}{(1)}, we have $\mathbf i_\lambda w_\lambda =\mathbf i_\lambda w_\lambda yd(\t)d(\s)^{-1}$, forcing   
    $\t^\lambda w_\lambda yd(\t)d(s)^{-1} =\t^\lambda w_\lambda$. Therefore,   $yd(\t)=d(\s)$. Since  $y\in \mathfrak S_{\lambda'}$ and $\s, \t\in \Std(\lambda')$, it follows that  $y=1$ and $\s=\t$, proving the ``only if " part of (1).    

    If $ \mathbf j^{\lambda, \xi} =\mathbf j^{\lambda, \xi} wyd(\t) {e} d^{-1}d(\s)^{-1}$, by   Lemma~\ref{lc2}(4),
$ed^{-1}=bc$ for some $b\in \mathfrak S_{r-2f}$ and 
some $c$ in the subgroup $\mathfrak S_{2f}'$ of $\mathfrak S_r$  generated by $\{ 
s_{r-2f+1}, s_{r-2f+2}, \ldots, s_{r-1}\}$. Since $b, y, d(\t), d(\s)\in \mathfrak S_{r-2f}$ and $w, c\in \mathfrak S_{2f}'$, we have \begin{equation}\label{jjj1}  \mathbf j=\mathbf j yd(\t)b d(\s)^{-1} \text{and $(\underset{r-2f}{ \underbrace{0,0, \cdots, 0}} , \mathbf j^\xi) =(\underset{r-2f}{ \underbrace{0,0, \cdots, 0}} , \mathbf j^\xi)
 wc$.}\end{equation} By Lemma~\ref{lc2} (3), $c=w^{-1}\in H_f$. Thus $e=b w^{-1}d$, implying that $d=e$, $b=w=c=1$. Now, the first equation in \eqref{jjj1} simplifies to 
$\mathbf j=\mathbf j y d(\t)d(s)^{-1}$. By (1),  we have  $y=1$ and $\t=\s$. This completes the proof of the ``only if " part of (2).   
\end{proof}

\begin{Lemma}
    For any $(\lambda, \xi)\in \Lambda_a^+(r-2f)\times  \mathbb N_a^{f}$, define $y_{\lambda, \xi}=\overrightarrow{\prod}_{c=1}^{r-2f} y_{l_c, a_c, c} \overrightarrow{\prod}_{s=1}^f y_{\xi, s}$.
    Then $$\widetilde{ y_{\lambda, \xi}} \mathbf {m}_i\otimes v_{\mathbf j^{\lambda, \xi}}\in M_{I_i, r}^{\le \sum_{c=1}^{r-2f}a_c+\sum_{s=1}^f \xi_{r-2s+1}}\setminus   M_{I_i, r}^{< \sum_{c=1}^{r-2f}a_c+\sum_{s=1}^f \xi_{r-2s+1}}    
    ,$$ where $\widetilde{ y_{\lambda, \xi}}$ is defined as in Lemma~\ref{high1}.  
\end{Lemma}
\begin{proof} The result follows immediately from the definition of $y_{\lambda, \xi}$, and \eqref{grade1}, and
\eqref{RS430}.
\end{proof}

\begin{Lemma}\label{k41} Suppose  $(\t, \xi , d), (\s, \xi, d')\in \delta(f, \lambda')$ such that   $\xi_{r-2j+1}=a-1$,   $1\le j\le f$.  Then, (up to a sign only in type $C_n$)
$\widetilde {y_{\lambda, \xi}} \mathbf{m}_i\otimes v_{\mathbf j^{\lambda, \xi }} d(\t) d$ is a term in $v_{\s, \xi, d' }$ satisfying 
\begin{equation}\label{incl1} \widetilde {y_{\lambda, \xi}} \mathbf {m}_i\otimes v_{\mathbf j^{\lambda, \xi }} d(\t) \in M_{I_i, r}^{\le \sum_{c=1}^{r-2f} a_c+f(a-1)}\setminus M_{I_i, r}^{< \sum_{c=1}^{r-2f} a_c+f(a-1)}\end{equation} if and only if $(\t, d)=(\s, d')$.
\end{Lemma}

\begin{proof} Recall $v_\lambda$ in \eqref{vlam}, and  $\mathbf l$ in Lemma~\ref{lc1}. By  \eqref{incbasic},
\begin{equation}\label{keykey1} v_{\lambda}  w_\lambda \tilde\pi_{[\lambda']} E^f X^\xi \equiv \mathbf{m}_i\otimes v_{\mathbf l} \otimes (v_1\otimes v_{-1})^{\otimes f}  E^f \overset{\rightarrow }\prod_{i=1}^{r-2f} Y_i^{a_i}  Y^\xi \pmod {M_{I_i, r}^{\le \sum_{c=1}^{r-2f} a_c +f(a-1)}}\end{equation}    where    \begin{equation}\label{YY}
    \text{$Y_1=X_1$, and  $Y_j=S_{j-1}Y_{j-1}S_{j-1}$, and 
$Y^\xi=\overset{\rightarrow}\prod_{j=f}^1 Y_{r-2j+1}^{\xi_{r-2j+1}}$.}
\end{equation} 

If $f=0$, then $d=d'=1$ and $ \mathbb N_a^{f}=\emptyset$, which makes     $Y^\xi=1$. From  Lemma~\ref{high1},   $$\mathbf{m}_i\otimes v_{\mathbf l}\overset{\rightarrow }\prod_{j=1}^r Y_j^{a_j} 
\in M_{I_i, r}^{\le \sum_{c=1}^{r} a_c}.$$ Thanks to Lemma~\ref{high1}(1), $\widetilde{y_{\lambda, \emptyset}} \mathbf{m}_i\otimes v_{\mathbf h} $ 
    is a term  of  $\mathbf{m}_i\otimes v_{\mathbf l}\overset{\rightarrow} \prod_{j=1}^r Y_j^{a_j} $ such that 
    $$\widetilde{y_{\lambda, \emptyset}} \mathbf{m}_i\otimes v_{\mathbf h} \in  M_{I_i, r}^{\le \sum_{c=1}^{r} a_c}   \setminus   M_{I_i, r}^{< \sum_{c=1}^{r} a_c}    
     $$   
    if and only if $\mathbf h=\mathbf j^{\lambda, \emptyset} $.
    
     Suppose  $f\neq 0$.  Using Lemma~\ref{high1}(1), and \eqref{eiact}, we see that    $\widetilde {y_{\lambda, \xi} } \mathbf{m}_i\otimes v_{\mathbf h} $ 
is a term of $\mathbf{m}_i\otimes v_{\mathbf l} \otimes (v_1\otimes v_{-1})^{\otimes f}  E^f\overset{\rightarrow} \prod_{j=1}^{r-2f}  Y_j^{a_j}  Y^\xi$ 
such that $$\widetilde {y_{\lambda, \xi} } \mathbf{m}_i\otimes v_{\mathbf h}
\in M_{I_i, r}^{\le \sum_{c=1}^{r-2f} a_c+f(a-1)}\setminus M_{I_i, r}^{< \sum_{c=1}^{r-2f} a_c+f(a-1)}$$
if and only if $\mathbf h=\mathbf j^{\lambda, \xi} w $ (up to a sign only in type $C_n$ case), for some $w\in {H_f}$. In any case,  \eqref{incl1} follows from Lemma~\ref{key21}, \eqref{keykey1} and the definition of $v_{{\s}, \xi, d'}$.
\end{proof}

\begin{Theorem}\label{main123}
	Suppose $\mu\in \Lambda_a^+(r-2f)$ for $0\le f\le \lfloor r/2\rfloor$. Under condition~\eqref{simple111},  
 $\{\overline {v_{\t, \xi, d}}\mid (\t, \xi, d)\in \delta(f, \mu') \}$ forms a  basis for 
 the $\mathbb C$-space $V_{\hat\mu}$ of all singular vectors in $M_{I_i, r}/M_{I_i, r}\langle E^{f+1}\rangle $   with the highest weight $\hat\mu$,  defined as \eqref{hatlambda}.
 \end{Theorem}
\begin{proof}
Since we maintain Assumption~\ref{keyassu},  we have $p_t-p_{t-1}\ge 2r$, $1\le t\le k$, and $M^{\frp_{I_i}}(\lambda_{I_i, \mathbf c })$ is simple (and hence tilting). By  the proof of\cite[Theorem 5.4]{RS}, we have   \begin{equation}
    \label{card}  | \delta(f, \mu')| =\dim\Hom_{\mathcal O^{\frp_{I_i}}} (M^{\frp_{I_i}}(\mu), M_{I_i, r}), 
\end{equation}
where $|\delta(f, \mu')|$ the cardinality   of $\delta(f, \mu')$. Furthermore, 
there is a bijective map
\begin{equation}\label{bijecmap} \iota_j: \Lambda_{a, j} \rightarrow \scI_{i, j},~~ \lambda\mapsto \hat \lambda,\    \text{for any $j, 0\leq j\leq r$,}\end{equation}
where $\hat\lambda$ is defined as in \eqref{hatlambda}, and $\mathscr{I}_{i, j}$ is given in \eqref{sciI}. Therefore, $\hat \mu\in \scI_{i, r}\setminus \scI_{i, r-2f-2}$. 
By Theorem~\ref{main111} and the  universal property of parabolic Verma modules, we have $\mathbb C$-linear isomorphisms 
\begin{equation} \label{iso1} \Hom_{\mathcal O^{\frp_{I_i}}} (M^{\frp_{I_i}}(\mu), M_{I_i, r})  \cong\Hom_{\mathcal O^{\frp_{I_i}}} (M^{\frp_{I_i}}(\mu), M_{I_i, r}/M_{I_i, r}\langle E^{f+1}\rangle)\cong V_{\hat \mu},\end{equation} if $\mu\in \Lambda_a^+(r-2f)$.  
This is the only place that we need condition~\eqref{simple111} in Section~4 so that we can use Theorem~\ref{main111} to count the dimension of $V_{\hat\mu}$.

By Proposition~\ref{hiofcyche}, \eqref{card}, \eqref{iso1}, it suffices to  prove that $\{\overline {v_{\t, \xi, d}}\mid (\t, \xi, d)\in \delta(f, \mu') \}$ is linear independent over $\mathbb C$. 
Suppose $\sum_{(\t, \xi, d)\in \delta(f, \mu')}   a_{\t, \xi, d} \overline{v_{\t, \xi, d}}=\bar 0$. Then \begin{equation} \label{llid}\sum_{(\t, \xi, d)\in \delta(f, \mu')} {a_{\t, \xi, d} }{v_{\t, \xi, d}}\in M_{I_i, r, {\hat \mu}}\langle E^{f+1}\rangle ,\end{equation} where $M_{I_i, r, {\hat \mu}}$ is the $\hat\mu$-weight space of $M_{I_i, r}$. We claim   
    ${a_{\t, \xi, d} }=0$ for all $(\t, \xi, d)$. Otherwise, we define the following non-empty set \begin{equation} \label{S1}S=\left\{\xi\mid a_{\t, \xi, d}\neq 0 \text{ for some $(\t, \xi, d)\in \delta(f, \mu') $ and $\sum_{s=1}^f \xi_{r-2s+1}$ is maximal}\right\}.\end{equation}  Pick a fixed $\eta$ such that $a_{\s, \eta, e}\neq 0 $ for some $\s, e$.
     
        \Case{1.  $f=0$} Then $\mathbb N_a^f=\emptyset$,  and  $e=1$. By \eqref{incl1}, (up to a sign only in type $C_n$ case) $\widetilde{y_{\mu, \emptyset}} \mathbf{m}_i \otimes v_{\mathbf j^{\mu, \emptyset}} d(\s)$ is a term of the summation in \eqref{llid} such that 
\begin{equation}\label{ssss1} \widetilde{y_{\mu, \emptyset}} \mathbf{m}_i \otimes v_{\mathbf j^{\mu, \emptyset}} d(\s)\in  M_{I_i, r}^{\le \sum_{c=1}^{r-2f} a_c}   \setminus   M_{I_i, r}^{< \sum_{c=1}^{r-2f} a_c} .\end{equation}

  \Case{2.  $f\neq 0$ and $\eta$ is the  $\xi$ in Lemma~\ref{k41}}  
  By Lemma~\ref{k41}(2), (up to a sign only in type $C_n$ case) $\widetilde{y_{\mu, \xi}}\mathbf{m}_i\otimes v_{\mathbf j^{\mu, \xi}}d(\s)e$ is a term in  the   
  summation in \eqref{llid} such that  \begin{equation}\label{sss2} \widetilde{y_{\mu, \xi}}\mathbf{m}_i\otimes v_{\mathbf j^{\mu, \xi}}d(\s)e\in M_{I_i, r}^{\le \sum_{c=1}^{r-2f} a_c+f(a-1)}   \setminus   M_{I_i, r}^{< \sum_{c=1}^{r-2f} a_c+f(a-1)}.\end{equation}

   \Case{3.  $f\neq 0$ and $\eta$ is not the $\xi$ in Lemma~\ref{k41}} We denote the $\xi$ by $\tilde \xi$ to avoid the $\xi$ in \eqref{llid}. By \eqref{llid},
   we have  
 \begin{equation} \label{llid1} \sum_{(\t, \xi, d)\in \delta(f, \mu')} {a_{\t, \xi, d} }{v_{\t, \xi, d}} e^{-1} d(\s)^{-1} Y^{\tilde \xi-\eta}\in M_{I_i, r, \hat \mu}\langle E^{f+1}\rangle ,\end{equation} where $Y^{\tilde \xi-\eta}$ is given in \eqref{YY}.
  Define $A, B, C$ such that 
   \begin{itemize} \item [(1)]  $(\t, \xi, d)\in A$ if $(\t, \xi, d)\in \delta(f, \mu')$ and $\xi\in S$. 
\item [(2)] $(\t, \xi, d)\in B$ if $(\t, \xi, d)\in A$ and the $j$-th component, say $z_j$ of $(\tilde \xi-\eta) d(\t) d e^{-1} d(\s)^{-1} $, is zero for all $1\le j\le r-2f$.
\item [(3)] $(\t, \xi, d)\in C$ if $(\t, \xi, d)\in B$ and $\eta_{r-2s+1}+\xi_{r-2s+2}+ z_{r-2s+1}+z_{r-2s+2}=a-1$ for all $1\le s\le f$.
\end{itemize} 
We have  
  $$  \begin{aligned}  \text{LHS of  \eqref{llid1}}   & \equiv  \sum_{(\t, \xi, d)\in A } {a_{\t, \xi, d} }v_\mu   w_\mu   E^f \overrightarrow {\prod}_{j=1}^{r-2f}  Y_j^{a_j}  Y^{\xi}  y_{\mu'} d(\t) d e^{-1} d(\s)^{-1}Y^{\tilde \xi-\eta}\ \text{by Lemma~\ref{high1}}\\ 
 &  \equiv  \sum_{(\t, \xi, d)\in A} {a_{\t, \xi, d} }v_\mu   w_\mu     E^f \overrightarrow {\prod}_{j=1}^{r-2f}  Y_j^{a_j}  Y^{\xi} \sum_{w\in \mathfrak S_{\bar{\mu'}}} (-1)^{l(w)} (z_w  Y^{\tilde \xi-\eta} z_w^{-1})  z_w \\ 
&  \equiv  \sum_{(\t, \xi, d)\in B}  {a_{\t, \xi, d} }v_\mu   w_\mu   E^f \overrightarrow {\prod}_{i=1}^{r-2f}  Y_j^{a_j}  Y^{\xi} (z_1  Y^{\tilde \xi-\eta} {z_1^{-1}}) y_{\mu'} {z_1}\  \text{Lemmas~\ref{lc2}(6),~\ref{high1}} \\ 
& \overset{(a)}\equiv    \sum_{(\t, \xi, d)\in C}  {a_{\t, \xi, d} }v_\mu   w_\mu   E^f \overrightarrow {\prod}_{j=1}^{r-2f}  Y_j^{a_j}  Y^{\xi} ({z_1}   Y^{\tilde \xi-\eta} {z_1^{-1}} ) y_{\mu'} {z_1}   \\ 
& \equiv    \sum_{(\t, \xi, d)\in C}  {a_{\t, \xi, d} }v_\mu   w_\mu    E^f \overrightarrow {\prod}_{j=1}^{r-2f}  X_j^{a_j}  X^{\xi} ({z_1}  X^{\tilde \xi-\eta} {z_1^{-1}}) y_{\mu'}{z_1}\ \  \text{Lemma~\ref{high1}}  \\
&  \equiv \sum_{(\t, \xi, d)\in C}\pm  {a_{\t, \xi, d} }v_\mu   w_\mu    E^f {\prod}_{j=1}^{r-2f}  X_j^{a_j}  X^{\tilde \xi } y_{\mu'} d(\t) d e^{-1} d(\s)\ \  \text{Definition~\ref{cba1}(13)}\\ 
& \equiv   \sum_{(\t, \xi, d)\in C}\pm   a_{\t, \xi, d} v_{\t, \tilde \xi, d}  \pmod{M_{I_i, r}^{\le \sum_{c=1}^{r-2f} a_c+f(a-1)}}  \\  \end{aligned} $$
where $a_i, 1\le i\le r-2f$ are as defined in  Lemma~\ref{lc1}, and $z_w= w d(\t) d e^{-1} d(\s)^{-1} $ {for any $w\in \mathfrak S_{\bar{\mu'}}$. In particular, $z_1$ is  $z_w$ for $w=1$.} 
Here (a)  is due to  Lemma~\ref{high1}, and Lemma~\ref{ppt}, which makes   $\text{deg} v_j\otimes v_{-j}=a-1$.

Thus, by \eqref{incl1} 
$\widetilde{y_{\mu, \tilde \xi}}\mathbf{m}_i\otimes v_{\mathbf j^{\mu, \tilde \xi}}$
is a term of  the 
  summation in \eqref{llid1} such that 
\begin{equation}\label{sss3}
\widetilde{y_{\mu, \tilde \xi}}\mathbf{m}_i\otimes v_{\mathbf j^{\mu, \tilde \xi}}\in  M_{I_i, r}^{\le \sum_{c=1}^{r-2f} a_c+f(a-1)}   \setminus   M_{I_i, r}^{< \sum_{c=1}^{r-2f} a_c+f(a-1)}.
\end{equation}

We will  use \eqref{ssss1}--\eqref{sss3} to prove that $S=\emptyset$, and therefore all $a_{\t,  \xi, d}$
in \eqref{llid} are zero. This implies that $\{\overline {v_{\t, \xi, d}}\mid (\t, \xi, d)\in \delta(f, \mu') \}$ is linear independent over $\mathbb C$.

Thanks to Theorem~\ref{cellular-1}, $\langle E^{f+1}\rangle $  has basis 
\begin{equation} \label{fbasis} \{C_{(\t_1,\xi,d_1),(\t_2,\gamma,d_2)}\,|\,
	(\t_1,\xi,d_1),(\t_2,\gamma ,d_2)\in\delta(c,\lambda), c>f, \ 
	\forall (c,\lambda)\in\Lambda_{a, r}\}.\end{equation}
Thus, $M_{I_i, r, \hat \mu } \langle E^{f+1} \rangle   $ is spanned by all  $yz$, where $y\in \mathcal S_{i, r}$ with $\text{wt} (y)=\hat \mu$ and  $$z=C_{(\t_1,\xi,d_1),(\t_2,\gamma, d_2)}=d_1^{-1} X^{\xi} n_{\t_1\t_2} E^c X^{\gamma} d_2.$$ 
 Thanks to   \eqref{llid} and \eqref{ssss1}--\eqref{sss3},   (up to a sign only in type $C_n$ case)   $\widetilde {y_{\mu, \tilde \xi }} \mathbf{m}_i\otimes v_{\mathbf j^{{\mu,  {\tilde \xi}}}} \sigma$ appears as a term in  some $ yz    $ with $\text{wt}(y)=\hat\mu$, where 
 \begin{equation} \label{ssig} \sigma=\begin{cases} 1 &\text{if $f\neq 0$ and $\eta \neq \tilde \xi$,}\\  d(\s)e & \text{otherwise.}\\
 \end{cases}
 \end{equation}
 Here $(\s, \eta, e)$ is the fixed triple chosen earlier such that   $a_{\s, \eta, e}\neq 0$.

 If  
$f_{\mathbf i,\mathbf j}^\mathbf l \mathbf{m}_i\otimes v_{\mathbf k}$ is a term of  $y d_1^{-1} X^{\xi} n_{\t_1\t_2} E^c $, then by \eqref{eiact}, \begin{equation} \label{ksub} \mathbf k_{r-2s+1}=-\mathbf k_{r-2s+2} \quad \text{ for all $1\le s\le c$ }.\end{equation}
Thus, (up to a sign only in type $C_n$ case)  $\widetilde{ y_{\mu, {\tilde \xi} }} \mathbf{m}_i\otimes v_{\mathbf j^{\mu, {\tilde \xi} }}\sigma    $ is a term of $f_{\mathbf i, \mathbf j}^\mathbf l \mathbf{m}_i \otimes v_{\mathbf k} X^\gamma d_2$ for some  $f_{\mathbf i, \mathbf j}^\mathbf l\mathbf{m}_i \otimes v_{\mathbf k} \in \mathcal S_{i, r}$ satisfying  $\text{wt} ( 
f_{\mathbf i, \mathbf j}^\mathbf l \mathbf{m}_i \otimes v_{\mathbf k} )=\hat \mu$, and  \eqref{ksub}. Thanks to Lemma~\ref{high1}, 
\begin{equation} \label{key321}  \widetilde{ y_{\mu, {\tilde \xi} }} \mathbf{m}_i\otimes v_{\mathbf j^{\mu, {\tilde \xi} }}\sigma 
\in M_{I_i, r}^{\le |\mathbf l|+\sum_{s=1}^c \gamma_{r-2s+1}}.  \end{equation}
On the other hand, we have 
$$\begin{aligned}  |\mathbf l| & \le \text{deg}  v_{\mathbf i_\mu}+f(a-1)-\text{deg} v_{\mathbf k}\ \text{ by Lemma~\ref{filt1}} \\
&=\sum_{s=1}^{r-2f} a_s+f(a-1)-\text{deg} v_{\mathbf k} \ \text{ by Lemma~\ref{lc2}(6)} \\
& \le \sum_{s=1}^{r-2f} a_s+f(a-1)-c(a-1) \ \text{ by Definition~\ref{ppt}, \eqref{ksub}}. \\
\end{aligned} $$
Combining this with \eqref{key321}, and using Lemma~\ref{k41},
we obtain  
$$|\mathbf l|=\sum_{s=1}^{r-2f} a_s+f(a-1)-c(a-1)  , \text{ and  $\gamma_{r-2s+1}=a-1$ for all $1\le s\le c$.}$$ By  Lemma~\ref{high1}, we conclude that (up to a sign only in type $C_n$ case)   $\widetilde{ y_{\mu, {\tilde \xi} }} \mathbf{m}_i\otimes v_{\mathbf j^{\mu, {\tilde \xi} }}\sigma    $ appears as  a term of $f_{\mathbf i, \mathbf j}^\mathbf l \mathbf{m}_i \otimes v_{\mathbf k}  Y^\gamma d_2$. 
However, by Lemma~\ref{high1}, 
and \eqref{ksub},  $\widetilde{ y_{\mu, {\tilde\xi} }} \mathbf{m}_i\otimes v_{\mathbf h }$ appears as  a term of 
$f_{\mathbf i, \mathbf j}^\mathbf l\mathbf{m}_i \otimes v_{\mathbf k} Y^\gamma d_2$ only if   
$\mathbf h_t=q_1-q+1$ for some $1\le t\le r$ and $1\le q\le r$. Since we assume $q_1\ge 2r$,  $\mathbf h_t\ge r+1$,
it contradicts to Lemma~\ref{lc2}(2)(5) if we replace $v_{\mathbf h}$ with  $v_{\mathbf j^{\mu, {\tilde \xi} }}\sigma $, forcing  $S=\emptyset$.
\end{proof}

\section{Proof of Theorem~\ref{main4} and Theorem~\ref{first} }
This section aims to provide   an explicit decomposition of $M_{I_i, r}$ into a direct sum of indecomposable tilting modules, and to compute the decomposition numbers of $\mathcal B_{a, r}(\mathbf u)$ under condition \eqref{simple111}.  
\vskip0.3cm


\noindent \textbf{Proof of  Theorem~\ref{main4} and Theorem~\ref{first}(1):} 
Obviously,  Theorem~\ref{first}(1) immediately follows 
from Theorem~\ref{main111} and Theorem~\ref{main4}.
We will prove Theorem~\ref{main4} in two  cases:  $f=0$ and $f\neq 0$. 


\Case{1. $f=0$} By Proposition~\ref{wallbcell} and Theorem~\ref{main123}, there is a $\mathbb C$-linear isomorphism 
$$\psi: V_{\hat \lambda}\rightarrow S^{0, \lambda}, \quad \bar {v_{\t, \emptyset, 1}}\rightarrow m_\lambda w_\lambda n_{\lambda'} d(\t),$$ where $S^{f, \lambda}$ is defined as in Proposition~\ref{bas}, and $V_{\hat \lambda}$ in Theorem~\ref{main123}. 
We aim  to prove that $\psi$ is an isomorphism of   $\mathcal B_{a, r}(\mathbf u)$-modules. If this holds, then   $V_{\hat\lambda}\cong S^{0, \lambda}$, and consequently,   $$\Hom_{\mathcal O^{\frp_{I_i}}}(M^{\frp_{I_i}}(\hat \lambda), M_{I_i, r}/M_{I_i, r}
\langle E^{1}\rangle)\cong S^{0, \lambda}.$$ More explicitly,  
the required isomorphism  sends each $\varphi_{\t, \emptyset, 1}$ to $m_\lambda w_\lambda n_{\lambda'} d(\t)$ where $$\varphi_{\t, \emptyset, 1}\in \Hom_{\mathcal O^{\frp_{I_i}}}(M^{\frp_{I_i}}(\hat \lambda), M_{I_i, r}/M_{I_i, r}\langle E^{1}\rangle)$$
 such that $\varphi_{\t, \emptyset, 1} (\mathbf m_i)= 
 \bar{v_{\t, \emptyset, 1}} 
 $. 
 Now, Theorem~\ref{main4} follows immediately  from Proposition~\ref{bas}.

 For the simplification of notation, we denote $v_{\t, \emptyset, 1}$ by $v_\t$. By \eqref{cycHiso} and Theorem~\ref{basisofhecke}, we have 
\begin{equation}\label{nhh}
n_{\lambda'}d(\t) h\equiv \sum_{\s\in \Std(\lambda')} a_\s n_{\lambda'} d(\s) +\sum_{\nu\in \Lambda_a^+(r), \nu\rhd \lambda'}\sum_{\t', \s'\in \mathscr T^{std}(\nu)}  a_{\t',\s'} n_{\t', \s'}
\pmod {\langle E^1\rangle },
\end{equation}
for some $a_\s, a_{\t', \s'}\in \mathbb C$.
Note that  $\nu\rhd \lambda'$ is  equivalent to \begin{equation}\label{big123} \lambda\rhd \nu'\end{equation} since both of them are $a$-multipartitions of $r$. It is well-known that $m_\lambda \mathscr H_{a, r}(\mathbf u) n_{\nu}=0$ unless $\lambda\unlhd\nu'$. Thus by \eqref{cycHiso}, \begin{equation}\label{ppsi1}  \psi (\bar v_\t) h=\sum_{\s\in \Std(\lambda')} a_s \psi (\bar v_\s). \end{equation} 
To prove  that $\psi$ is a $\mathcal B_{a, r}(\mathbf u)$-homomorphism, by \eqref{nhh}, it suffices to verify \begin{equation}\label{kkks1} \overline{v_\lambda} w_\lambda d(\t')^{-1} n_{\nu}=\bar 0
\end{equation} for any $\t'\in \Std(\nu)$ such that $\nu\rhd \lambda'$. 

We have $[\lambda']\preceq [\nu]$, where $\prec $ is the lexicographic order. Write 
$[\nu]=[b_0, b_1. \ldots, b_a]$ and 
$[\lambda']=[c_0, c_1, \ldots, c_a] $ in the sense of \eqref{blam}. 

\subcase{1. $[\nu]=[\lambda']$}  
Then
$|\lambda^{(j)}|=|\mu^{(j)}|$, and hence $\lambda^{(j)}\unrhd \mu^{(j)} $ for all  $1\le j\le a$, where $\mu^{(j)} $ is the conjugate of $\nu^{(a-j+1)}$. Furthermore, from \eqref{big123},  there is at least one of $l$ such that $\lambda^{(l)}\rhd \mu^{(l)}$. Note that  
$d(\t')$ can be either in $\mathfrak S_{[\lambda']}$ or not. 

In the first case, 
$$\begin{aligned} \bar{ v_\lambda}  w_\lambda d(\t')^{-1} n_{\nu}
& \overset{(1)} ={\pm}\frac{1}{\prod_{j=1}^a \lambda^{(j)}!}  \overline {\mathbf m_i\otimes v_{\mathbf i_\lambda}}x_{\lambda^{(1)}\vee \lambda^{(2)}\vee \cdots\vee \lambda^{(a)}}  w_{\lambda}d(\t')^{-1} y_{\nu^{(1)}\vee \nu^{(2)}\vee \cdots\vee \nu^{(a)}}  \tilde \pi_{[\nu]}\\
& \overset{(2)} ={\pm}\frac{1}{\prod_{j=1}^a \lambda^{(j)}!}  \overline {\mathbf m_i\otimes v_{\mathbf i_\lambda}} w_{[\lambda]}x_{\lambda^{(a)}\vee \cdots\vee \lambda^{(1)}}\prod_{j=a}^1 \tilde w(j)  d(\t')^{-1} y_{\nu^{(1)}\vee \cdots\vee \nu^{(a)}}  \tilde \pi_{[\nu']}\\
& \overset{(3)}=\bar 0.\\ \end{aligned}
$$ 
Here (1) follows from  $v_{\mathbf i_\lambda} ={\pm}\frac{1}{\prod_{j=1}^a \lambda^{(j)}!}   {v_{\mathbf i_\lambda}}x_{\lambda^{(1)}\vee \lambda^{(2)}\vee \cdots\vee \lambda^{(a)}} $, and (2) is a consequence of  \eqref{wlaex}, and (3) follows from  $x_{\lambda^{(l)}}\mathfrak S_{c_l-c_{l-1}} y_{\nu^{(a-l+1)}}=0$ since $\lambda^{(l)} \rhd \mu^{(l)}$ under our assumption. We remark that the $``-"$ may appear only when we consider the symplectic Lie algebra $\mathfrak{sp}_{2n}$.

In the second case, since $d(\t')\not \in \mathfrak S_{[\nu]}$,
there is an $h$, $1 \le h\le r$,  and a $j$,  $1\le j\le a$ such that   \begin{equation}\label{bin1} 
h\le b_j, \text{ and } \ 
(h) d(\t')\ge b_j+1.\end{equation}  

\subcase {2.    $[\lambda']\prec [\nu]$} Then there is a minimal $j$ such that $c_j<b_j$ and $c_i=b_i$ for all $i<j$. 
Thus,  there is a positive integer  $h$ satisfying  \begin{equation}\label{bin2}  h\le c_j+1, \text{ and } (h)d(\t')\ge c_j+1.\end{equation}  

In each of two cases, by \eqref{bin1}--\eqref{bin2}, there is a unique $d$ such that 
\begin{equation} \label{invkey} c_d+1\le (h) d(\t')\le c_{d+1}, \text{  and $d\ge j$.}\end{equation} 
Suppose $\mathbf l=\mathbf i_\lambda w_\lambda$, 
and $\mathbf j=\mathbf l d(\t')^{-1}$.  
We have  
$$
    \begin{aligned} \bar v_{\lambda} w_\lambda d(\t')^{-1} n_\nu &= \overline{\mathbf m_i\otimes v_\mathbf j} (1, h) (1, h)  \tilde \pi_{[\nu]} y_\nu\\
&   =\overline{\mathbf m_i\otimes v_\mathbf j} (1, h) \prod_{s=j}^{a-1} \pi_{b_s}(u_{a-s})  (1, h)  \prod_{s=1}^{j-1}
\pi_{b_s}(u_{a-s})
 y_\nu,
  \end{aligned} 
$$
where the second equality follows from  the inequality  $h\leq b_j$,  \eqref{cycHiso} and \eqref{pic}. To obtain \eqref{kkks1}, we have to discuss two cases as follows. 

\subcase{a. $a-1-d<k$ where $d$ is defined  in \eqref{invkey}} 
By   \eqref{lll12}, we have $j_h=l_{(h)d(t')} \in \mathbf p_{a-d}$ and $\mathbf m_i\otimes v_{j_h}\in N_{a-d}$, where $N_{a-d}$ is defined  in Proposition~\ref{polyofx}. In this case, \eqref{kkks1} follows from  Proposition~\ref{polyofx} since $\prod_{t=1}^{a-d} (X_1-u_t)$ is a factor of  $\prod_{s=j}^{a-1} \pi_{b_s} (u_{a-s})   $ and $b_j>0$.

\subcase{b.  $a-1-d\ge k$}
By  \eqref{lll12}, we have $j_h=l_{(h)d(\t')}\in -\mathbf p_{2+d} $ and $\mathbf m_i\otimes v_{j_h}\in 
N_{2k-1+\delta_{\mathfrak  g, \mathfrak{so}_{2n+1}-d}}$. Since $d\ge j$, we have $a-j\ge a-d\ge k+1$. Thus  \eqref{kkks1} follows from Proposition~\ref{polyofx} since 
$ \prod_{s=1}^{a-j} (X_1-u_s)  $ is a factor of $\prod_{s=j}^{a-1}  \pi_{b_s}(u_{a-s})  $. 

This completes the proof of the result when  $f=0$.  
 
\Case{2. $f\neq 0$} By \cite[Lemma~8.3]{AMR}, \begin{equation}\label{AMR}
    E^f \mathcal B_{a, r}(\mathbf u) E^f=E^f \mathcal B_{a, r-2f}(\mathbf u)
\end{equation} for any $0<f\le [r/2]$. 
Thus, we have right exact tensor functor $?\otimes_{\mathcal B_{a, r-2f}(\mathbf u)} E^f \mathcal B_{a, r}(\mathbf u)$
sending any $\mathcal B_{a, r-2f}(\mathbf u)$-module $N$ to $  N\otimes_{\mathcal B_{a, r-2f}(\mathbf u)} E^f B_{a, r}(\mathbf u)$. Thanks to \cite[Proposition~3.29(b)]{RSi}, we have $$ C(0, \lambda')\otimes_{\mathcal B_{a, r-2f}(\mathbf u) } E^f \mathcal B_{a, r}(\mathbf u)\cong C(f, \lambda') .$$ By the result in Case~1, and Theorem~\ref{main111},  $\text{Hom}_{\mathcal O^{\frp_{I_i}}} (M^{\frp_{I_i}}(\hat \lambda), M_{I_i, r-2f})\cong C(0, \lambda')$, forcing 
 \begin{equation}\label{iso333} \text{Hom}_{\mathcal O^{\frp_{I_i}}} (M^{\frp_{I_i}}(\hat \lambda), M_{I_i, r-2f})\otimes_{\mathcal B_{a, r-2f}(\mathbf u) } E^f \mathcal B_{a, r}(\mathbf u)\cong C(f, \lambda') .\end{equation} 
 Define 
\begin{equation}\label{keypsi} 
    \gamma: \text{Hom}_{\mathcal O^{\frp_{I_i}}} (M^{\frp_{I_i}}(\hat \lambda), M_{I_i, r-2f})\otimes_{\mathcal B_{a, r-2f}(\mathbf u) } E^f \mathcal B_{a, r}(\mathbf u)\rightarrow \text{Hom}_{\mathcal O^{\frp_{I_i}}} (M^{\frp_{I_i}}(\hat \lambda), M_{I_i, r})\end{equation}
such that $\gamma(y\otimes E^f b)=\tau(b) \circ (Id_{M_{I_i, r-2f} } \otimes \alpha^f)\circ y$
 where $\alpha$ is defined as in \eqref{alp123},  $b\in \mathcal B_{a, r}(\mathbf u)$  and $y\in \text{Hom}_{\mathcal O} (M^{\frp_{I_i}}(\hat \lambda), M_{I_i, r-2f})$, and $\tau $ is the anti-involution defined as in Lemma~\ref{anti-inv}. 

 We verify that $\gamma$ is well-defined. 
 Following Theorem~\ref{thmA},  we can  view $E^f$ as a morphism in $\text{End}_{\mathcal O^\frp_{I_i}} (M_{I_i, r})$. Thus, by \eqref{eiact}--\eqref{alp123}  $$ Id_{M_{I_i, r-2f} } \otimes \alpha^f=(\epsilon_\frg N)^{-f} E^f \circ (Id_{M_{I_i, r-2f} } \otimes \alpha^f), 
 $$ and  $$\tau(b) \circ (Id_{M_{I_i, r-2f} } \otimes \alpha^f)\circ y= \tau(b_1) \circ (Id_{M_{I_i, r-2f} } \otimes \alpha^f)\circ y$$ if $E^f b=E^f b_1$. This proves that $\gamma$ is well-defined, 
  and  $\gamma$ is a right $\mathcal B_{a, r}(\mathbf u)$-homomorphism.   By Theorem~\ref{main111}, Theorem~\ref{main123}, and \eqref{keypsi}, there is an epimorphism $\bar \gamma$:  
$$\text{Hom}_{\mathcal O^{\frp_{I_i}}} (M^{\frp_{I_i}}(\hat \lambda), M_{I_i, r-2f})\otimes_{\mathcal B_{a, r-2f}(\mathbf u) } E^f \mathcal B_{a, r}(\mathbf u)\twoheadrightarrow  
\text{Hom}_{\mathcal O^{\frp_{I_i}}} (M^{\frp_{I_i}}(\hat \lambda), M_{I_i, r}/M_{I_i, r}\langle E^f\rangle).$$ 
By comparing the dimensions  using \eqref{iso333},  Theorem~\ref{main111}, and \eqref{card}, we conclude that $\bar\gamma$ is an isomorphism. 
Now, the required isomorphism in Theorem~\ref{main4} follows immediately from 
\eqref{iso333}.\qed

 We aim to  describe the highest weight $\mu$, and the multiplicity $n_\mu$ in \eqref{decten1} under  condition \eqref{simple111}. To do it, we introduce the functor \begin{equation}\label{funcf} 
\mathcal  F:=\Hom_{\mathcal O}(-, M_{I_i, r}): \mathcal O^{\mathfrak p_{I_i}}\rightarrow \text{End}_{\mathcal O}(M_{I_i, r})\text{-mod},\end{equation} where 
$\text{End}_{\mathcal O}(M_{I_i, r})\text{-mod}$ is  the category of left $\text{End}_{\mathcal O}(M_{I_i, r})$-modules.
By Theorem~\ref{thmA}, we can use $\mathcal B_{a, r}(\mathbf u)$-mod, the category of right  $\mathcal B_{a, r}(\mathbf u)$-modules to replace $\text{End}_{\mathcal O}(M_{I_i, r})\text{-mod}$.

\textbf{Proof of Theorem~\ref{first}(2)-(4):} We have $\mu\in \mathscr{I}_{i, r} $ since  
 $n_\mu\neq 0$ and $$(T^{\frp_{I_i}}(\mu): (M^{\frp_{I_i}}(\mu))=1.$$  It follows from  \eqref{bijecmap} that  $\mu=\hat\nu$ for some $(\ell, \nu)\in {\Lambda_{a, r}}$. 
Suppose  \begin{equation}\label{simiso1}  D(\hat \nu)\cong D(f, \lambda) \end{equation} 
as right $\mathcal B_{a, r}(\mathbf u)$-modules for some $(f, \lambda)\in\bar\Lambda_{a, r}$. 
By \eqref{simiso1} and Theorem~\ref{first}(1), we have  \begin{equation}\label{simiso2}\begin{aligned}  
\relax [ C(\ell, \nu'): D(f, \lambda) ]& =[S(\hat \nu): D(\hat \nu)]\neq 0,  \\
 [C(f, \lambda): D(f, \lambda)] & =[S(\hat {\lambda'}): D(\hat \nu)] \neq 0. \end{aligned} 
\end{equation} 
Thus, we have 
   $ (\ell, \nu')\unrhd (f, \lambda)$  and \begin{equation}\label{ineq1} \hat {\lambda'}\leq \hat\nu,\end{equation} which implies $\ell\geq f$. 
   
   We claim that $\ell=f$. Otherwise, we have $\ell>f$, which makes   $C(f, \lambda)E^\ell=0$. We have $\phi(D(\hat \nu)E^\ell)=D(f, \lambda)E^\ell=0 $, where   $\phi$ is  the isomorphism in \eqref{simiso1}. Thus, 
    \begin{equation}\label{zero1}
    D(\hat\nu ) E^\ell=0.\end{equation}
  Since $\omega_0=-2n$ if $\Phi=C_n$, and $2n$ (resp., $2n+1$) if $\Phi=D_n$ (resp., $B_n$), we have $\omega_0\neq 0$.  Note  that  $$ E^\ell m_{\nu}w_{\nu} n_{\nu'}E^\ell=E^\ell E^\ell m_{\nu}w_{\nu} n_{\nu'}=(\omega_0)^\ell E^\ell m_{\nu}w_{\nu} n_{\nu'}. $$ 
  By  Proposition~\ref{bas}, the cell module $C(\ell, \nu')$ is generated by $C(\ell, \nu')E^\ell$. The isomorphism in  Theorem~\ref{first}(1) implies that  $S(\hat \nu)$ is generated by $S(\hat \nu)E^\ell$. Consequently,  $D(\hat \nu)$ is generated by $D(\hat \nu)E^\ell$, forcing   $D(\hat \nu)E^\ell\neq 0$. It contradicts to \eqref{zero1}.  This completes the proof of our claim. 
  
We have $\nu'\unrhd \lambda$. Since $\nu, \lambda\in \Lambda_a^+(r-2f)$, it follows that     $ \lambda'\unrhd  \nu$, which is equivalent to $\hat {\lambda'}\geq \hat \nu$ by 
 \eqref{lm}. Combining \eqref{ineq1}, we have  $\hat {\lambda'}=\hat\nu$, which forces   $\lambda'=\nu$. Now, Theorem~\ref{first}(2) follows.

Clearly, 
    Theorem~\ref{first}(3)  follows from Theorem~\ref{first}(2) except for  the multiplicity of 
    $T^{\frp_{I_i}}(\hat \lambda)$ in $M_{I_i, r}$. Since $\mathcal F(T^{\frp_{I_i}}(\hat\lambda))$ is the project cover of $D(\hat\lambda)$, by 
Theorem~\ref{thmA}  and Theorem~\ref{first}(2), the multiplicity of $T^{\frp_{I_i}}(\hat\lambda )$ is equal to the dimension of 
$\Hom_{\mathcal B_{a, r}(\mathbf u)}(\mathcal B_{a, r}(\mathbf u),  D(f, \lambda'))$. Now, Theorem~\ref{first}(3) follows immediately since 
$\Hom_{\mathcal B_{a, r}(\mathbf u)}(\mathcal B_{a, r}(\mathbf u),  D(f, \lambda'))
\cong  D(f, \lambda')$.  
Finally,   Theorem~\ref{first}(4) follows from Theorem~\ref{first}(1) and \cite[Corollary~5.10]{RS}.\qed

  \section{Appendix: Proof of Theorem~\ref{saturated} by Wei Xiao}
In this section, we focus on the parabolic subalgebra $\frp_I$ associated  with $I\subset \Pi$.  
Throughout, we fix $\lambda\in \Lambda^{\frp_I}$   such that 
\begin{equation}\label{tilass} \langle\lambda+\rho, \beta^\vee\rangle\not\in\bbZ_{>0}, \quad  \forall \beta\in \Phi^+\setminus \Phi_I, \quad \text{and } \langle \lambda+\rho, \alpha^\vee\rangle =1,\ \forall \alpha\in I.\end{equation} 
Under this condition,  
$M^\frp(\lambda)$ is simple and  
$\dim_\mathbb C F(\lambda)=1$, where $F(\lambda)$ is given in \S \ref{para}.
Notably,   condition  \eqref{tilass} will only  be needed  in the proof of Lemma~\ref{1lem4}.
\begin{Defn} For any  anti-dominant $\lambda\in\Lambda^{\frp_I}$ such that $\dim F(\lambda)=1$, we define 
 \begin{itemize}
   \item [(1)] $\mathscr{K}_r= \{\mu\in\frh^*\mid \ [V^{\otimes r} : F(\mu)]\neq0\}$, \item [(2)] $\mathscr{I}_r=\{\mu\in\frh^*\mid \ (M^\frp(\lambda)\otimes V^{\otimes r} : M^\frp(\mu))\neq0\}$.
   \end{itemize}
 \end{Defn} 
 It follows that  $\mathscr{I}_r=\lambda+\mathscr{K}_r$.  For convenience, we define 
\begin{equation}\label{1eq1}
		S_\mu=\begin{cases} 
	\{\mu+h\eps_i\in\Lambda^{\frp_I}\mid 1\leq i\leq n, h=0, \pm1\} &  \text{if $\Phi=B_n$, and either  $\eps_n\not\in  I$  or  $\mu_n\neq0$,}\\ 
	\{\mu+h\eps_i\in\Lambda^{\frp_I}\mid 1\leq i\leq n, h=\pm1\} & \text{otherwise. }
\end{cases}
\end{equation}

Recall that the \textit{dot action} of the Weyl group $W$ on $\frh^*$ is defined  by $$s_\beta\cdot\lambda=s_\beta(\lambda+\rho)-\rho$$ for $\beta\in\Phi$ and $\lambda\in\frh^*$.

\begin{Lemma}\label{1lem21}
	Let $\mu\in\Lambda^{\frp_I}$. Then $F(\mu)\otimes V=\bigoplus_{\nu\in S_\mu} F(\nu).$

\end{Lemma}

\begin{proof}
	First, by  \cite[Proposition 4.12]{X1}, we have 
	\[
		F(\mu)\otimes V=\bigoplus_{\nu\in\Lambda^{\frp_I}}m_\nu F(\nu),
	\]
	where $m_\nu=\sum_{w\in W_I}(-1)^{\ell(w)}\dim V_{w\cdot\nu-\mu}$, and $W_I$ is the parabolic subgroup of $W$ associated with  $I$.
	
	If $m_\nu\neq0$ for some $\nu\in\Lambda^{\frp_I}$, then $\dim V_{w\cdot\nu-\mu}\neq0$ for some $w\in W_I$.
  In this case, we can assume that $w\cdot\nu-\mu=h\eps_i$ for some $1\leq i\leq n$ with $h\in\{0, \pm 1\}$. Notably,  $h\neq0$ when $\Phi=C_n$ or $D_n$. Thus, $w(\nu+\rho)=\mu+\rho+h\eps_i$. 
  
  If $\mu+\rho+h\eps_i\in\Lambda^{\frp_I}$, this forces $w=1$ and $\nu=\mu+h\eps_i\in\Lambda^{\frp_I}$. 
   Now suppose $\mu+\rho+h\eps_i\not\in\Lambda^{\frp_I}$. This implies  $\rho+h\eps_i\not\in\Lambda^{\frp_I}$, which can only occur when    $\Phi=B_n$, $h=-1$, and $\eps_i=\eps_n\in I$. To make $\mu+\rho+h\eps_i\not\in\Lambda^{\frp_I}$, we also need $\mu_n=0$. Under these conditions, the weight $s_{\eps_n }(\mu+\rho-\eps_n)=\mu+\rho\in\Lambda^{\frp_I}$. Therefore, $w=s_{\eps_n}$ and $\nu=\mu$. 
	
	To summarize, if $\nu=\mu+h\eps_i\in\Lambda^{\frp_I}$ for some $1\leq i\leq n$ and $h\in\{1, -1\}$, then $$m_\nu=\dim V_{1\cdot\nu-\mu}=\dim V_{h\eps_i}=1.$$ In the remaining cases, we have $m_\nu=0$ unless $\nu=\mu$ and $\Phi=B_n$. In this exceptional case, if $\eps_n\in I$ and $\mu_n=0$, then $m_\mu=\dim V_{1\cdot\nu-\mu}-\dim V_{s_{\eps_n}\cdot\nu-\mu}=\dim V_{0}-\dim V_{-\eps_n}=0$. If either $\eps_n\not\in I$ or $\mu_n\neq0$, then $m_\nu=\dim V_{1\cdot\nu-\mu}=\dim V_{0}=1$. In summary, we obtain the following result as required: 
	\begin{equation*}
		m_\nu=\begin{cases} 
			1 &\qquad  \nu=\mu+h\eps_i\in\Lambda^{\frp_I}, 1\leq i\leq n, h\in\{\pm1\}\\
			1 &\qquad  \nu=\mu, \Phi=B_n, \text{either }\eps_n\not\in I\text{ or }\mu_n\neq0\\ 
			0 &\qquad \text{otherwise. }
		\end{cases}
	\end{equation*}
 This completes the proof of  the lemma.
\end{proof}

Lemma~\ref{1lem21}  implies that
$$\mathscr{K}_r=\cup_{\mu\in\mathscr{K}_{r-1}} S_{\mu}\ \ \text{ for $r\geq1$.}$$ To explicitly describe the set $\scK_r$, we need additional  notation. Let $\Pi\setminus I=\{\alpha_{p_1}, \cdots, \alpha_{p_k}\}$ for $0=p_0< p_1<\cdots<p_k\leq p_{k+1}=n$. If $\Phi=D_n$, we can assume that $p_k\neq n-1$ by symmetry. The following result can be verified, easily.

 \begin{Lemma}\label{lem1}
 	The weight $\mu\in\Lambda^{\frp_I}$ if and only if the following conditions are satisfied:
 	\begin{multicols}{2}
 		\item [(1)] $\mu_{p_{i-1}+1}\geq\cdots\geq\mu_{p_i}$ for $i\leq k+1$,
 		\item [(2)] $\mu_n\geq0$ if  $\Phi$ is either $B_n$ or $C_n$ and $p_k<n$,
 		\item [(3)] $\mu_{n-1}\geq|\mu_n|$ if  $\Phi=D_n$ and $p_k<n$.
 	\end{multicols}
 \end{Lemma}

 To explicitly describe the set $\scI_r$, we define the following sets:
 
\begin{equation}\label{XYX}
\begin{aligned}
	\mathcal X_r=&\{(a_1, \cdots, a_n)\in\bbZ^n\mid\textstyle\sum_{i=1}^n|a_i|\leq r\};\\
	\mathcal X'_r=&\{(a_1, \cdots, a_n)\in \mathcal X_r\mid\textstyle\sum_{i=1}^na_i\equiv r(\mathrm{mod}2)\};\\
	\mathcal X'_{r, j}=&\{(a_1, \cdots, a_n)\in \mathcal X'_{r}\mid a_{n-j}\neq0\}, \text{ $0\leq j<n-p_k$ and $\mathcal X'_{r, n-p_k}=\mathcal X'_r$.}
\end{aligned}
\end{equation}
We aim to show that $\mathscr{K}_r=\mathcal Y_r$ or $\mathcal Y'_r$, 
where $\mathcal Y_r:=\mathcal X_r\cap\Lambda^{\frp_I}$ and $\mathcal Y'_r:=\mathcal X'_r\cap\Lambda^{\frp_I}$.  If $\Phi=B_n$, then by  Lemma \ref{lem1}, we have  
\begin{equation}\label{1eq2}
\mathcal Y'_{r, 0}\subset \mathcal Y'_{r, 1}\subset \cdots \subset \mathcal Y'_{r, n-p_k}=\mathcal Y_r,
\end{equation}
where $\mathcal Y'_{r, j}=\mathcal X'_{r, j}\cap\Lambda^{\frp_I}$. We  also define  $\mathcal Y'_{r, j}=\emptyset$ if $r<0$. The following result will be  useful.

\begin{Lemma}\label{1lem31}
	Let $r\geq0$. 
	\begin{itemize}
		\item [(1)] If $\Phi=B_n$ and $\eps_n\not\in I$, then $\scK_r=\mathcal Y_r$;
		\item [(2)] If $\Phi=C_n$ or $D_n$, then $\scK_r=\mathcal Y'_r$;
		\item [(3)] If $\Phi=B_n$ and $\eps_n\in I$, then $\scK_r=\mathcal Y'_r\cup_{0\leq j\leq n-p_k}\mathcal Y'_{r-2j-1, j}$. In particular, if $r\leq n-p_k$, then $\scK_r=\mathcal Y'_r$.
	\end{itemize} 
	
\end{Lemma}
\begin{proof}
	We proceed by  induction on $r$. The case $r=0$ is straightforward, as  $\scK_0=\{\textbf{0}\}$. 
	
	(1) Assume  $\scK_{r-1}=\mathcal Y_{r-1}$ holds.  From  (\ref{1eq1}) and Lemma \ref{1lem21}, we obtain   $$\mathcal Y_r\supset\cup_{\mu\in \scK_{r-1}}S_\mu=\scK_r.$$ For the reverse inclusion, choose any $\nu\in Y_{r}$. We need to show $\nu\in S_\mu$ for some $\mu\in \mathcal Y_{r-1}$. If $\nu=0$, we can simply choose $\mu=0\in \mathcal Y_{r-1}$ by (\ref{1eq1}).
    Now suppose $\nu\neq 0$. Let $i$ be the smallest integer such that $\nu_i\neq0$. If $\nu_i<0$, it can be easily verified that $\mu=\nu+\eps_i\in \mathcal X_{r-1}\cap\Lambda^{\frp_I}=\mathcal Y_{r-1}$, keeping in mind of Lemma \ref{lem1}. So $\nu=\mu-\eps_i\in S_{\mu}$.
    
    Now suppose $\nu_i>0$. By Lemma \ref{lem1}(1), we can assume that $i=p_s+1$ for some $1\leq s\leq k$. Choose the largest $j\leq p_{s+1}$ such that $\nu_j=\nu_i$. Again by Lemma \ref{lem1}, we have $\mu=\nu-\eps_j\in \mathcal X_{r-1}\cap\Lambda^{\frp_I}=\mathcal Y_{r-1}$. Hence $\nu=\mu+\eps_j\in S_\mu$.
	
	(2) The reasoning here is similar to (1), with a key difference  in the proof showing $\mathcal Y'_r\subset\scK_r$. When $\nu=0\in \mathcal Y'_r$, we do not have $0\in \mathcal  Y'_{r-1}$ and $0\in S_{0}$. Fortunately, now $r\equiv 0(\mathrm{mod}2)$, which implies $0\in S_{\mu}$ for $\mu=\eps_{1}\in \mathcal Y'_{r-1}$.
	
	(3) Suppose $\scK_{r-1}=\mathcal Y'_{r-1}\cup_{0\leq j\leq n-p_k}\mathcal Y'_{r-2j-2, j}$. We start by showing $$\mathcal Y'_r\bigcup_{0\leq j\leq n-p_k}\mathcal Y'_{r-2j-1, j}\supset\bigcup_{\mu\in \scK_{r-1}}S_\mu=\scK_r.$$ Assume  $\nu\in S_\mu$ for some $\mu\in \scK_{r-1}$, so   $\nu=\mu+h\eps_i\in\Lambda^{\frp_I}$ for some $h\in\{0, \pm1\}$ and $1\leq i\leq n$. 
 
 First, consider the case $\nu=\mu$. One has $\mu_n\neq0$ by (\ref{1eq1}) since we assume $\eps_n\in I$. If $\mu\not\in \mathcal Y'_{r-1}$, then $\mu\in \mathcal Y'_{r-2j-2, j}\subset \mathcal Y'_{r-2j-2}\subset \mathcal Y'_r$  for some  $0\leq j\leq n-p_k$. Here the first inclusion follows from 
  (\ref{1eq2}).  
  If $\mu\in\mathcal  Y'_{r-1}$, then $\nu=\mu\in \mathcal Y'_{r-1, 0}$ since $\mu_n\neq0$. 
  
  Next, consider the case $\nu=\mu\pm\eps_i$. If $\mu\in\mathcal  Y'_{r-1}$, then $\nu\in \mathcal Y'_r$ is evident. If $\mu\in \mathcal Y'_{r-2j-2, j}$ for a smallest $j$, then $\nu_{n-j+1}=0$ when $j>0$ and $\nu_{n-j}\neq0$ when $j<n-p_k$. We have $\nu\in\mathcal  Y'_{r-2j-1, j}$ unless $\nu_{n-j}=0$ with $j<n-p_k$. In this exception case, one obtains $\mu_{n-j}=1$ and $\nu=\mu-\eps_{n-j}$. This means $\nu\in\mathcal  Y'_{r-2(j+1)-1, j+1}$. 
	
	For the reverse direction, choose any $\nu\in \mathcal Y'_r\cup_{0\leq j\leq n-p_k}\mathcal Y'_{r-2j-1, j}$. We need to prove $\nu\in S_\mu$ for some $\mu\in \scK_{r-1}$. If $\nu\in \mathcal Y'_r$, we can show that $\nu\in S_\mu$ for some $\mu\in \mathcal Y'_{r-1}\subset\scK_{r-1}$ as in (2). If $\nu\in\mathcal  Y'_{r-2(n-p_k)-1, n-p_k}=\mathcal Y'_{r-2(n-p_k)-1}$, the argument is similar. Now assume that $\nu\in \mathcal Y'_{r-2j-1, j}$ for a smallest $j<n-p_k$. So $\nu_{n-j+1}=0$ when $j>0$ and $\nu_{n-j}\neq0$. If $j=0$, then $\nu_n\neq0$ yields $\nu\in S_\mu$ for $\mu=\nu\in \mathcal Y'_{r-1, 0}\subset \mathcal Y'_{r-1}\subset\scK_{r-1}$. If $j>0$, then $\mu=\nu+\eps_{n-j+1}\in \mathcal Y'_{r-2j, j-1}\subset\scK_{r-1}$. In any case,  $\nu\in \scK_r$.
	
	Finally, suppose $r\leq n-p_k$. We need to show $\mathcal Y'_{r-2j-1, j}=\emptyset$ for any ${0\leq j\leq n-p_k}$. Indeed, if $\nu\in \mathcal Y'_{r-2j-1, j}$, then $2j+1\leq r\leq n-p_k$ and $\nu_{n-j}\neq 0$. By Lemma {\ref{lem1}}, we obtain $\nu_{p_k+1}\geq \nu_{p_k+2}\geq\cdots{\ge}  \nu_{n-j}\geq1$. This means $r-2j-1\geq n-j-p_k$ and thus $r\geq n+j+1-p_k>n-p_k$, a contradiction.
\end{proof}

\begin{Lemma}\label{1lem5}
	Let $r\geq0$. Then $\scI_r\subset (\lambda+\mathcal X_r)$ unless $\Phi=B_n$ and $r>n-p_k$. Moreover, we have $(\scI_r\setminus \scI_{r-2})\cap(\lambda+\mathcal X_{r-2})=\emptyset$. 
\end{Lemma}
	\begin{proof}
		Since $\dim F(\lambda)=1$, this follows straightforward from Lemma \ref{1lem31}. 
	\end{proof}

	\begin{Lemma}\label{1lem4} Suppose $\langle\mu+\rho, \beta^\vee\rangle\in\bbZ_{>0}$ for some $\beta\in  \Phi^+\backslash\Phi_I$ and $\mu\in \lambda+\mathcal X_r$.  Then $s_\beta\cdot\mu\in\lambda+\mathcal X_r$.
	\end{Lemma}
	\begin{proof}
		Write $\lambda+\rho=-\sum_{i=1}^n c_i\eps_i$. We can assume that $\mu-\lambda=\sum_{i=1}^na_i\eps_i$ with $a_i\in\bbZ$ and $\sum_{i=1}^n|a_i|\leq r$. Then $$\mu+\rho=\sum_{i=1}^n(a_i-c_i)\eps_i.$$ Moreover, we must have $\langle\lambda+\rho, \beta^\vee\rangle\in\bbZ_{\leq0}$, keeping in mind that $\langle\lambda+\rho, \beta^\vee\rangle\not\in\bbZ_{>0}$ in equation \eqref{tilass}. Notably, this is the place in section~6, where we need condition~\eqref{tilass}.
        We have 	\begin{equation}\label{kkk1111} 
	s_\beta\cdot\mu-\lambda=\begin{cases} \sum_{i\neq k}^na_i\eps_i+(2c_k-a_k)\eps_k & \text{if $\beta=2\epsilon_k$ or $\epsilon_k$,}\\
	\sum_{i\neq k, l}^na_i\eps_i+(a_l-c_l+c_k)\eps_k+(a_k-c_k+c_l)\eps_l & \text{if $\beta=\epsilon_k-\epsilon_l$,}\\
	\sum_{i\neq k, l}^na_i\eps_i+(c_l-a_l+c_k)\eps_k+(c_l-a_k+c_k)\eps_l & \text{if $\beta=\epsilon_k+\epsilon_l$,}\\
	\end{cases} 
	\end{equation}

		If $\beta=2\eps_k$ for some $k\leq n$, then $\Phi$ has to be $C_n$, and 
  $a_k-c_k\in\bbZ_{>0}$ since $\langle\mu+\rho, \beta^\vee\rangle\in\bbZ_{>0}$, and $c_k\in\bbZ_{\geq0}$ since $\langle\lambda+\rho, \beta^\vee\rangle\in\bbZ_{\leq0}$. This means $|2c_k-a_k|\leq |a_k|$. 
 
  If $\beta=\eps_k$ for some $k\leq n$, then $\Phi$ has to be $B_n$. We have  $2(a_k-c_k)\in \mathbb Z_{>0}$ and $2c_k\in \mathbb Z_{\ge 0}$. We still have $|2c_k-a_k|\leq |a_k|$.   
      
       If $\beta=\eps_k-\eps_l$ for $k<l\leq n$, then $(a_k-c_k)-(a_l-c_l)\in\bbZ_{>0}$ and
  $c_k-c_l\in\bbZ_{\geq0}$. 
		Likewise, we have $|a_l-c_l+c_k|+|a_k-c_k+c_l|\leq|a_k|+|a_l|$.  
		
        If $\beta=\eps_k+\eps_l$ for $k<l\leq n$, then $(a_k-c_k)+(a_l-c_l)\in\bbZ_{>0}$ and $c_k+c_l\in\bbZ_{\geq0}$. 
		We have $|c_k-a_l+c_l|+|c_l-a_k+c_k|\leq|a_k|+|a_l|$. 
        
        In any case, by \eqref{kkk1111},  $s_\beta\cdot\mu\in \lambda+\mathcal  X_n$.
	\end{proof}
	
	\begin{Lemma}\label{1lem6}
		Suppose 
 $\mu\in\lambda+\mathcal X_r$. If $\langle\mu+\rho, \alpha^\vee\rangle\in\bbZ_{<0}$ for some $\alpha\in\Phi_I^+$, 
  then  $s_\alpha\cdot\mu\in\lambda+\mathcal X_r$.
	\end{Lemma}
	\begin{proof} 	Write $\lambda+\rho=-\sum_{i=1}^n c_i\eps_i$, and   $\mu-\lambda=\sum_{i=1}^na_i\eps_i$ with $a_i\in\bbZ$ and $\sum_{i=1}^n|a_i|\leq {r} $. Then $$\mu+\rho=\sum_{i=1}^n(a_i-c_i)\eps_i.$$
 Then \eqref{kkk1111} is still hold if we replace $\beta$ by  $\alpha$. 
 Since  $\langle\lambda+\rho, \alpha^\vee\rangle\in\bbZ_{>0}$ and   $\langle\mu+\rho, \alpha^\vee\rangle\in\bbZ_{<0}$,  we have 
  $c_k\in \mathbb Z_{<0}$, and $a_k-c_k\in \mathbb Z_{<0}$ if $\alpha=2\eps_k$ and $2c_k\in \mathbb Z_{<0}$, and $2(a_k-c_k)\in \mathbb Z_{<0}$ if $\alpha=\eps_k$, and $c_k-c_l\in \mathbb Z_{<0}$, and $(a_k-c_k)-(a_l-c_l)\in\bbZ_{<0}$  
  if $\alpha=\eps_k-\eps_l$, and $c_k+c_l\in \mathbb Z_{<0}$, $(a_k-c_k)+(a_l-c_l)\in\bbZ_{<0}$ if $\alpha=\eps_k+\eps_l$.
  
  In any case, $s_\alpha \cdot \mu \in \lambda+ \mathcal X_r$. 
	\end{proof}
	
	\begin{Lemma}\label{lem6}
		Let $\mu, \nu\in\Lambda^{\frp_I}$. Suppose $\nu=(ws_\beta)\cdot\mu\in\Lambda^{\frp_I}$ for some $\beta\in\Phi^+\backslash\Phi_I$ and $w\in W_I$.  Assume that $\langle\mu+\rho, \beta^\vee\rangle\in\bbZ_{>0}$. If $\mu\in\lambda+\mathcal X_r$, then $\nu\in\lambda+\mathcal X_r$.
	\end{Lemma}
	\begin{proof}  Let $s_{\alpha_1}\cdots s_{\alpha_l}$ be a reduced expression of $w$.  Since $\nu\in \Lambda^{\frp_I}$, $\langle \nu, \alpha_1^\vee \rangle \ge 0$, we have  $$(s_{\alpha_2}\cdots s_{\alpha_l}s_\beta, \alpha_1)<-1.$$ Note that $s_{\alpha_1}\gamma\in \Phi^+$ for any $\alpha_1\ne \gamma\in \Phi^+$. This implies $\langle s_{\alpha_2}\cdots s_{\alpha_l} s_{\beta}\cdot \mu, \alpha_2^\vee\rangle \in \mathbb Z_{\ge 0} $, and hence   $\langle s_{\alpha_3}\cdots s_{\alpha_l} s_{\beta}\cdot \mu, \alpha_2^\vee\rangle \in \mathbb Z_{< 0} $. Similarly, we have $\langle s_{\alpha_j}\cdots s_{\alpha_l} s_\beta\cdot\mu,  \alpha_{{j-1}}^\vee\rangle \in \mathbb Z_{<0}$ for all $4\le j\le l$. Here we set $\alpha_{l+1}=\beta$. Now, we can obtain the result by first applying Lemma~\ref{1lem4} to $\beta$, then
 applying Lemma \ref{1lem6} to $\alpha_l, \cdots, \alpha_1$. 
	\end{proof}

For any object $M\in \mathcal O^\frp$, let $\text{Rad}^i M=\text{Rad}(\text{Rad}^{i-1} M )$ for $i\ge 1$ and $\text{Rad}^0 M=M$, where 
$\text{Rad}M$  is the radical of $M$.  
For any $\mu, \nu\in \Lambda^{\frp_I}$ write $\mu>\nu$ if $\text{Hom}_{\frg} (M^\frp(\nu), M^\frp (\mu))\neq 0$. If  $\mu>\nu$, define 
\begin{equation}\label{psi1} \Psi_{\mu, \nu}=\{ \beta\in \Phi^+\setminus \Phi_I\mid \langle \mu+\rho, \beta^\vee\rangle\in  \bbZ_{>0}, \nu=(w_\beta s_\beta)\cdot \mu \text{ for some $w_\beta\in W_I$ } \}.\end{equation}
Let $K_0(\mathcal O^\frp)$ denote the Grothendieck group of the parabolic category $\mathcal O^\frp$. 
For each $M \in \mathcal O^\frp$, let $[M]$   be the  corresponding element in 
$ K_0(\mathcal O^\frp)$.
 
\begin{Prop}\label{HX1}  \cite[Corollary~5.6]{HX}, \cite[Lemma~3.3]{XZ}. Suppose $\mu\in \Lambda^{\frp_I}$.
	\begin{itemize} \item [(1)] 
		$\sum_{i>0} [\text{Rad}^i M^\frp (\mu)  ]=\sum_{\mu>\xi\in \Lambda^{\frp_I}}  c(\mu, \xi)  [M^\frp (\xi)]$, where  $c(\mu, \xi)$ is called the Jantzen coefficient associated with $(\mu, \xi)$    \item [(2)] If $\mu>\xi$, then $c(\mu, \xi)=\sum_{\beta\in \Psi_{\mu, \xi} } (-1)^{\ell(w_\beta)} $, where $\ell(\ ) $ is the length function on $W$ (and hence on $W_I$). 
	\end{itemize} 	
\end{Prop}

\begin{Lemma}\label{lem7}
		Let $\mu, \nu\in\Lambda^{\frp_I}$ such that $\mu\neq \nu $. If $[M^\frp(\mu): L(\nu)]\neq0$, then there exists a series $\nu=\mu^k<\cdots<\mu^1<\mu^0=\mu$ such that 
       $\mu^{i}=(w_is_{\beta_i})\cdot\mu^{i-1}$ for some $\beta_i\in\Phi^+\backslash\Phi_I$ and $w_i\in W_I$, and $1\le i\le k$. Moreover, $\langle\mu^i+\rho, \beta_i^\vee\rangle\in\bbZ_{>0}$.
	\end{Lemma}

\begin{proof} Thanks to Proposition~\ref{HX1}(1), there is a $\xi\in \Lambda^{\frp_I}$ such that $\mu>\xi\ge \nu$ and  $$c(\mu, \xi) [M^\frp (\xi): L(\nu)]\neq 0.$$
	If $\xi=\nu$, then $c(\mu, \nu)\neq 0$. By Proposition~\ref{HX1}(2), $\Psi_{\mu, \nu}\neq \emptyset$, and hence there is a $\beta\in \Phi^+\setminus \Phi_I$ such that $\nu=(w_\beta s_\beta )\cdot \mu$ and $\langle \mu+\rho, \beta^\vee\rangle \in \mathbb Z_{>0}$. 
    
    If $\xi\neq \nu$, then $c(\mu, \xi)\neq 0$ and $  [M^\frp (\xi): L(\nu)]\neq 0$. By Proposition~\ref{HX1}(2), $\Psi_{\mu, \xi}\neq \emptyset$.  Let $\mu^1=\xi$. Replacing $\mu$ by $\xi$, we apply the above procedure.  Since $\xi\in W\cdot \mu$,  this  procedure will end in finite steps. 
\end{proof}


\textbf{Proof of Theorem~\ref{saturated}:} Let  $\lambda_{I_i, \mathbf c}$ be  in \eqref{deltac}  satisfying  Assumption~\ref{simple11}, where $I_1$ and $I_2$ are defined as   in Definition~\ref{assum12}. Then $\lambda_{I_i, \mathbf c}$ is a 
 special case of  current $\lambda$ in \eqref{tilass}. This allows us  to  freely use  previous results in this section. 

Suppose $\mu\in\scI_{i,r}$. If  $\nu\in\Lambda^{\frp_{I_i}}$ satisfies  $\nu\preceq \mu$, then there exists   a sequence $$\nu=\gamma^0, \gamma^1, \ldots, \gamma^j=\mu$$ in $\Lambda^{\frp_{I_i}}$ such that $\relax [M^{\frp_{I_i}} (\gamma^l): L(\gamma^{l-1})]\neq 0$ for all $1\le l\le j$. 
Since we keep the Assumption~\ref{keyassu}, we have  $p_t-p_{t-1}\ge 2r$ for all $1\le t\le k$. This allows us to apply 
Lemma~\ref{1lem5}, which asserts  that $\scI_{i, r}\subset \lambda_{I_i, \mathbf c}+\mathcal X_r$. Consequently, $\mu\in \scI_{i, r}\subset \lambda_{I_i, \mathbf c}+\mathcal X_{r} $.
Applying  Lemmas \ref{lem6} and \ref{lem7} repeatedly, we deduce that   $\nu\in \lambda_{I_i, \mathbf c}+\mathcal X_{r} $. Finally,  since $i\neq 1$ if $\Phi=B_n$, we conclude that $\nu\in \scI_{i,r}$, as 
 $$\begin{aligned} (\lambda_{I_i, \mathbf c}+\mathcal X_{r} )\cap \Lambda^{\frp_{I_i}}&=\lambda_{I_i, \mathbf c}+(\mathcal X_{r}\cap \Lambda^{\frp_{I_i}}) \text{ by \eqref{tilass}}
\\&= \lambda_{I_i, \mathbf c}+\scK_r \text{ by   Lemma~\ref{1lem31}(1)-(2)}\\&=\scI_{i, r}.\end{aligned}$$
Finally, for each $0\le j<r$, $\scI_{i, j}$ is still saturated since  Assumption~\ref{keyassu} and Assumption~\ref{simple11}  are still available for $M_{I_i, j}$.  
\qed

\end{document}